%% file: Principalization_of_2_Ideal_Classes.tex
\documentclass[11pt,reqno]{amsart}
\usepackage[latin1]{inputenc} % Pour pouvoir taper les accent directement et non pas passer par \'
\usepackage{indentfirst}
\usepackage{amsmath} % Pour utiliser la commande \boldsymbol qui permet de mettre des lettres grecs en gras
\usepackage{amsfonts} % Permet l'utilisation de plus de polices de caractères en mode mathématique
\usepackage{amsthm} % Pour faciliter l'écriture de théorèmes mathématiques
\usepackage{epsfig} % Pour utiliser la commande \epsfig
\usepackage{graphicx} % Pour utiliser la commande \includegraphics
\usepackage{multirow} % Pour faire des tableaux contenant des lignes fusionnées
\usepackage{verbatim} % Permet d'insérer un fichier ASCII brut
\usepackage{rotating} % Perment l'utilisation de l'environnement sideways
\usepackage{lettrine}
\usepackage{amscd}
\usepackage{dsfont}
\usepackage{lscape}
\usepackage{ae}% meilleur affichage sous adobe
\usepackage{float}
\usepackage[all]{xy}
\usepackage{icomma} % la virgule agira comme séparateur décimal en mode mathématique
\usepackage{amsmath} % Pour utiliser la commande \boldsymbol qui permet de mettre des lettres grecs en gras
\usepackage{amsfonts} % Permet l'utilisation de plus de polices de caractères en mode mathématique
\usepackage{amsthm} % Pour faciliter l'écriture de théorèmes mathématiques
\usepackage{varioref} % Créer une référence plus confortable
\usepackage{layout}
\usepackage[Lenny]{fncychap} % beaux c h a p i t r e s
\usepackage{enumerate}
\usepackage{amsfonts, amsmath, amssymb, stmaryrd, latexsym}
\usepackage{dsfont}
\usepackage{array}
\usepackage{longtable}
\usepackage{geometry}
\usepackage[frenchb,english]{babel}
\usepackage{mathrsfs}
\usepackage[mathscr]{euscript}
\begin{document}
% THEOREM Environments
\renewcommand{\baselinestretch}{1.3}
\newcommand{\K}{\mathbb{Q}(\sqrt{d},i)}
\newcommand{\q}{\mathbb{Q}(\sqrt{2p_1p_2})}
\newcommand{\pro}{\prod_{p/\Delta_K}e(p)}
\theoremstyle{plain}
\newtheorem{them}{Theorem}[]
\newtheorem{lem}{Lemma}[]
\newtheorem{propo}{Proposition}[]
\newtheorem{coro}{Corollary}[]
\newtheorem{proprs}{properties}[]
%==================================================
\theoremstyle{remark}
\newtheorem{rema}{Remark}[]
\newtheorem{remas}{Remarks}[]
\newtheorem{exam}{\textbf{Example}}[]
\newtheorem{exams}{\textbf{Numerical Examples}}[]
\newtheorem{df}{definition}[]
\newtheorem{dfs}{definitions}[]
% MATH ---------------------------------------------------
\def\NN{\mathds{N}}
\def\RR{\mathbb{R}}
\def\HH{I\!\! H}
\def\QQ{\mathbb{Q}}
\def\CC{\mathbb{C}}
\def\ZZ{\mathbb{Z}}
\def\OO{\mathcal{O}}
\def\kk{\mathds{k}}
\def\KK{\mathbb{K}}
\def\ho{\mathcal{H}_0^{\frac{h(d)}{2}}}
\def\LL{\mathbb{L}}
\def\h{\mathcal{H}_1}
\def\hh{\mathcal{H}_2}
\def\hhh{\mathcal{H}_3}
\def\hhhh{\mathcal{H}_4}
\def\jj{\overline{\mathcal{H}_3}}
\def\jjj{\overline{\mathcal{H}_4}}
\def\L{\mathds{k}_2^{(2)}}
\def\M{\mathds{k}_2^{(1)}}
\def\k{\mathds{k}^{(*)}}
\def\l{\mathds{L}}
%------------------------------------------------------------------------
\title[Principalization of $2$-class groups of type $(2,2,2)$]{Principalization of $2$-class groups of type $(2,2,2)$\\ of biquadratic fields
$\QQ\left(\sqrt{\strut p_1p_2q},\sqrt{\strut -1}\right)$}
%premier auteur
\author[A. AZIZI]{Abdelmalek Azizi}
\address{Abdelmalek Azizi et Abdelkader Zekhnini: Département de Mathématiques, Faculté des Sciences, Université Mohammed 1, Oujda, Morocco }
% deuxième auteur
\author[A. Zekhnini]{Abdelkader Zekhnini}
% troisième auteur
\email{abdelmalekazizi@yahoo.fr}
\email{zekha1@yahoo.fr}
\author[M. Taous]{Mohammed Taous}
\address{Mohammed Taous: Département de Mathématiques, Faculté des Sciences et Techniques, Université Moulay Ismail, Errachidia, Morocco}
\email{taousm@hotmail.com}
\author[D. C. Mayer]{Daniel C. Mayer}
\address{Daniel C. Mayer: Naglergasse 53, 8010 Graz, Austria}
\email{algebraic.number.theory@algebra.at\\
URL: http://www.algebra.at}
\thanks{Research of the last author supported by the Austrian Science Fund (FWF): P 26008-N25}
%\subjclass[2010]{11R16, 11R29, 11R32, 11R37}

\begin{abstract}
Let $p_1\equiv p_2\equiv -q\equiv1 \pmod4$ be different primes such that $\displaystyle\left(\frac{2}{p_1}\right)=
\displaystyle\left(\frac{2}{p_2}\right)=\displaystyle\left(\frac{p_1}{q}\right)=\displaystyle\left(\frac{p_2}{q}\right)=-1$. Put $d=p_1p_2q$ and $i=\sqrt{-1}$, then the bicyclic biquadratic field $\kk=\K$ has an elementary abelian $2$-class group, $\mathbf{C}l_2(\kk)$, of rank $3$. In this paper, we study the principalization of the $2$-classes of $\kk$ in its fourteen unramified abelian extensions $\KK_j$ and $\LL_j$ within $\M$, that is the Hilbert $2$-class field of $\kk$. We determine the nilpotency class, the coclass, generators and the structure of the metabelian Galois group $G=\mathrm{Gal}(\L/\kk)$ of the second Hilbert 2-class field $\L$ of $\kk$. Additionally, the abelian type invariants of the groups $\mathbf{C}l_2(\KK_j)$ and $\mathbf{C}l_2(\LL_j)$ and the length of the $2$-class tower of $\kk$ are given.
\end{abstract}

\maketitle
\section{Introduction}
\label{s:Intro}
The $2$-class tower of imaginary quadratic fields $k=\QQ(\sqrt{d})$ with $2$-class group of type $(2,2,2)$ has been investigated by E. Benjamin, F. Lemmermeyer, and C. Snyder \cite{B.L.S-03}. More recently, H. Nover \cite{No-09} provided evidence of such towers with exactly three stages.
However, nothing was known about base fields of higher degree until A. Azizi, A. Zekhnini, and M. Taous focussed on complex bicyclic biquadratic fields $\kk=\QQ(\sqrt{d},\sqrt{-1})$, which are called {\it special Dirichlet fields} by D. Hilbert \cite{Hb-1894}.
In \cite{AZT14-3}, they determined the maximal unramified pro-$2$ extension of special Dirichlet fields with $2$-class group of type $(2,2,2)$ and radicand $d=2p_1p_2$, where $p_1\equiv p_2\equiv 5\pmod{8}$ denote primes.
It is the purpose of the present article to pursue this research project further for special Dirichlet fields with $\mathbf{C}l_2(\kk)$ of type $(2,2,2)$ and radicand $d=p_1p_2q$ composed of primes $p_1\equiv p_2\equiv 5\pmod{8}$ and $q\equiv 3\pmod{4}$.
As predicted in \cite[\S\ 4.2, pp. 451--452]{Ma-13}, the particular feature of the lattice of intermediate fields between a base field $\kk$ with $2$-class group of type $(2,2,2)$ and its Hilbert $2$-class field $\M$ is the constitution by {\it two layers} of unramified abelian extensions, rather than by just one layer as for a $2$-class group of type $(2,2)$, and it seems to be the first time that complete results are given here for the second layer.
The layout of this paper is the following. First, all main theorems are presented in \S\ \ref{s:Main}.
To be able to compute the Hasse unit index $Q_K=[E_K:W_KE_{K^+}]$ of CM-fields $K$ with maximal real subfield $K^+$
and the unit index $q(K/\QQ)=[E_K:\prod_j^s E_{k_j}]$ of multiquadratic fields $K$ with quadratic subfields $k_j$,
preliminary results on fundamental systems of units (FSU) of such fields are summarized in \S\ \ref{s:Preliminaries}.
In \S\ \ref{s:Tower}, one of the seven unramified quadratic extensions of $\kk$, denoted by $\KK_3=\kk(\sqrt{q})$,
turns out to carry crucial information about the Galois group $G=\mathrm{Gal}(\L/\kk)$ of the second Hilbert $2$-class field $\L$ of $\kk$,
since $\lvert G\rvert=2\cdot h(\KK_3)\ge 64$.
The fact that $\KK_3$, which is contained in the genus field $\kk^{(*)}$ of $\kk$,
has an abelian $2$-class tower permits the conclusion that $\kk$ has a metabelian $2$-class tower of exact length $2$.
The abelian type invariants of $\mathbf{C}l_2(\KK_j)$, in dependence on the $2$-class numbers of the subfields $\QQ(\sqrt{p_1p_2})$ and $\QQ(\sqrt{-p_1p_2})$
of $\KK_3$ and on the Legendre symbol $\left(\frac{p_1}{p_2}\right)$, are determined in \S\ \ref{s:Proofs}.
Explicit generators, in the form of Artin symbols, a presentation of the second $2$-class group $G$, and the structure of $G^\prime\simeq\mathbf{C}l_2(\M)$,
in dependence on the norm of the fundamental unit of $\QQ(\sqrt{p_1p_2})$, are also given in \S\ \ref{s:Proofs}.
The nilpotency class $c(G)$ and the coclass $cc(G)$ turn out to depend on biquadratic residue symbols for $p_1$ and $p_2$.
The principalization kernels $\kappa_{\KK_j/\kk}$ can be given either in terms of generators of $\mathbf{C}l_2(\kk)$ or as norm class groups.
They are determined as kernels of Artin transfers $V_{G,G_j}:G/G^\prime\to G_j/G_j^\prime$ from $G$ to its maximal subgroups
$G_j$, $1\le j\le 7$, using the presentation of $G$. These invariants form the {\it transfer kernel type} (TKT) of $G$ in \cite[Dfn. 1.1, p. 403]{Ma-13}.
Finally, the abelian type invariants of $\mathbf{C}l_2(\LL_j)$, $1\le j\le 7$, are calculated as abelian quotient invariants
$\mathcal{G}_j/\mathcal{G}_j^\prime$ of the subgroups $\mathcal{G}_j$ of index $\lbrack G:\mathcal{G}_j\rbrack=4$ of $G$,
which make up the {\it transfer target type} (TTT) of $G$ in \cite[Dfn. 1.1]{Ma-13}.
We conclude with numerical examples, statistics, and details about the groups $G$ in \S\ \ref{s:Examples}.

 Let $m$ be a square-free integer and $K$ be a number field. Throughout this paper, we adopt the following notation:
 \begin{itemize}
   \item $h(m)$, resp. $h(K)$: the $2$-class number of $\QQ(\sqrt m)$, resp. $K$.
   \item $\varepsilon_m$: the fundamental unit of $\QQ(\sqrt m)$, if $m>0$.
   \item $\mathcal{O}_K$: the ring of integers of $K$.
   \item $E_K$: the unit group of $\mathcal{O}_K$.
   \item $W_K$: the group of roots of unity contained in $K$.
   \item $\omega_K$: the order of $W_K$.
   \item $i=\sqrt{-1}$.
   \item $K^+$: the maximal real subfield of $K$, if $K$ is a CM-field.
   \item $Q_K=[E_K:W_KE_{K^+}]$ is the Hasse unit index, if $K$ is a CM-field.
   \item $q(K/\QQ)=[E_K:\prod_j^s E_{k_j}]$ is the unit index of $K$, if $K$ is multiquadratic and $k_j$ are the quadratic subfields of $K$.
   \item $K^{(*)}$: the absolute genus field of $K$ (over $\QQ$).
   \item $\mathbf{C}l_2(K)$: the 2-class group of $K$.
   \item $\kappa_{L/K}$: the subgroup of classes of $\mathbf{C}l_2(K)$ which become principal in an extension $L/K$
         (the {\it principalization-} or {\it capitulation-kernel} of $L/K$).
 \end{itemize}

\section{Main results}
\label{s:Main}
Let $p_1\equiv p_2\equiv -q \equiv 1 \pmod4$ be different primes satisfying the following conditions \eqref{10:18}:
 \begin{equation}\label{10:18}\displaystyle\left(\frac{2}{p_1}\right)=
      \displaystyle\left(\frac{2}{p_2}\right)=\displaystyle\left(\frac{p_1}{q}\right)=\displaystyle\left(\frac{p_2}{q}\right)=-1.\end{equation}
 Then there exist some positive integers $e$, $f$, $g$ and $h$ such that  $p_1=e^2+4f^2$ and $p_2=g^2+4h^2$. Put  $p_1=\pi_1\pi_2$ and $p_2=\pi_3\pi_4$, where $\pi_1=e+2if$ and $\pi_2=e-2if$ (resp. $\pi_3=g+2ih$ and $\pi_4=g-2ih$) are conjugate prime elements in the cyclotomic field $k=\QQ(i)$ dividing $p_1$ (resp. $p_2$). Denote by $\kk$ the complex bicyclic biquadratic field $\K$, where $d=p_1p_2q$. Its three quadratic subfields are $k=\QQ(i)$, $k_0=\QQ(\sqrt{d})$ and $\overline{k}_0=\QQ(\sqrt{-d})$. Let $\kk_2^{(1)}$ be the Hilbert 2-class field of $\kk$,  $\L$ be its second Hilbert 2-class field and $G$ be the Galois group of $\kk_2^{(2)}/\kk$. According to  \cite{AT09},  $\kk$ has an elementary abelian 2-class group $\mathbf{C}l_2(\kk)$ of rank 3, that is, of type $(2, 2, 2$). The  main results of this paper are  Theorems  \ref{10:2},  \ref{10:3} and  \ref{10:4} below; whereas   Theorem \ref{10:1}  is proved in \cite{AZT12-1},  using  \cite{AT09} and \cite{Ka76}.
\subsection{Unramified extensions of $\kk$}
The first and the second assertion of the following theorem hold according to  \cite{Ka76}  and  \cite{AT09}  respectively, the others are proved in \cite{AZT12-1}. The fields in Theorem  \ref{10:1}  are visualized in Figure  \ref{fig:SubfieldLattice}.
\begin{them}\label{10:1}
Let $p_1$, $p_2$ and $q$ be different primes as specified by equation \eqref{10:18}.
\begin{enumerate}[\rm\indent(1)]
  \item The $2$-class groups   of $k_0$, $\overline{k}_0$ are of type $(2, 2)$.
  \item The $2$-class group,  $\mathbf{C}l_2(\kk)$, of $\kk$ is of type $(2, 2, 2)$.
  \item The discriminant of $\kk$ is: $disc(\kk)=disc(k).disc(k_0).disc(\overline{k}_0)=2^4p_1^2p_2^2q^2.$
  \item $\kk$ has seven unramified quadratic extensions within  $\kk_2^{(1)}$. They are given by:
\begin{center}$\KK_1=\kk(\sqrt{p_1})$,\qquad $\KK_2=\kk(\sqrt{p_2})$,\qquad  $\KK_3=\kk(\sqrt q)$,\\
 $\KK_4=\kk(\sqrt{\pi_1\pi_3})$, \ $\KK_5=\kk(\sqrt{\pi_1\pi_4})$,\ $\KK_6=\kk(\sqrt{\pi_2\pi_3})$ and  $\KK_7=\kk(\sqrt{\pi_2\pi_4})$.
  \end{center}
  \item $\KK_1$, $\KK_2$, $\KK_3$ are intermediate fields between $\kk$ and its genus field $\k$. The fields $\KK_4\simeq\KK_7$
and $\KK_5 \simeq \KK_6$ are pairwise conjugate and thus isomorphic. Consequently $\KK_1$, $\KK_2$, $\KK_3$ are
absolutely abelian, whereas $\KK_4$, $\KK_5$, $\KK_6$, $\KK_7$ are absolutely non-normal over $\QQ$.
  \item $\kk$ has seven unramified bicyclic biquadratic extensions within  $\kk_2^{(1)}$. One of them is
       \begin{center}$\mathbb{L}_1=\KK_1.\KK_2.\KK_3=\k=\QQ(\sqrt{p_1}, \sqrt{p_2}, \sqrt{q}, \sqrt{-1})$,\end{center}
the absolute genus field of $\kk$ and the others are given by:
\begin{center}
$\mathbb{L}_2=\KK_1.\KK_4.\KK_6,  \quad
\mathbb{L}_3=\KK_1.\KK_5.\KK_7$, \quad
$\mathbb{L}_4=\KK_2.\KK_4.\KK_5 \text{ and }
\mathbb{L}_5=\KK_2.\KK_6.\KK_7$.
\end{center}
 Moreover $\LL_2\simeq \LL_3$ and $\LL_4 \simeq\LL_5$ are pairwise conjugate and thus isomorphic and absolutely non-normal, and
 \begin{center} $\mathbb{L}_6=\KK_3.\KK_4.\KK_7$\  \text{ and } \
$\mathbb{L}_7=\KK_3.\KK_5.\KK_6$
  \end{center}
  are absolutely Galois over $\QQ$.
\end{enumerate}
\end{them}

%\newpage
%--------------------------------------------------------------------------------
\begin{figure}[ht]
\caption{Subfield lattice of the Hilbert \(2\)-class field \(\kk_2^{(1)}\) of \(\kk\)}
\label{fig:SubfieldLattice}
\input{AziziZekhniniTaous}
\end{figure}
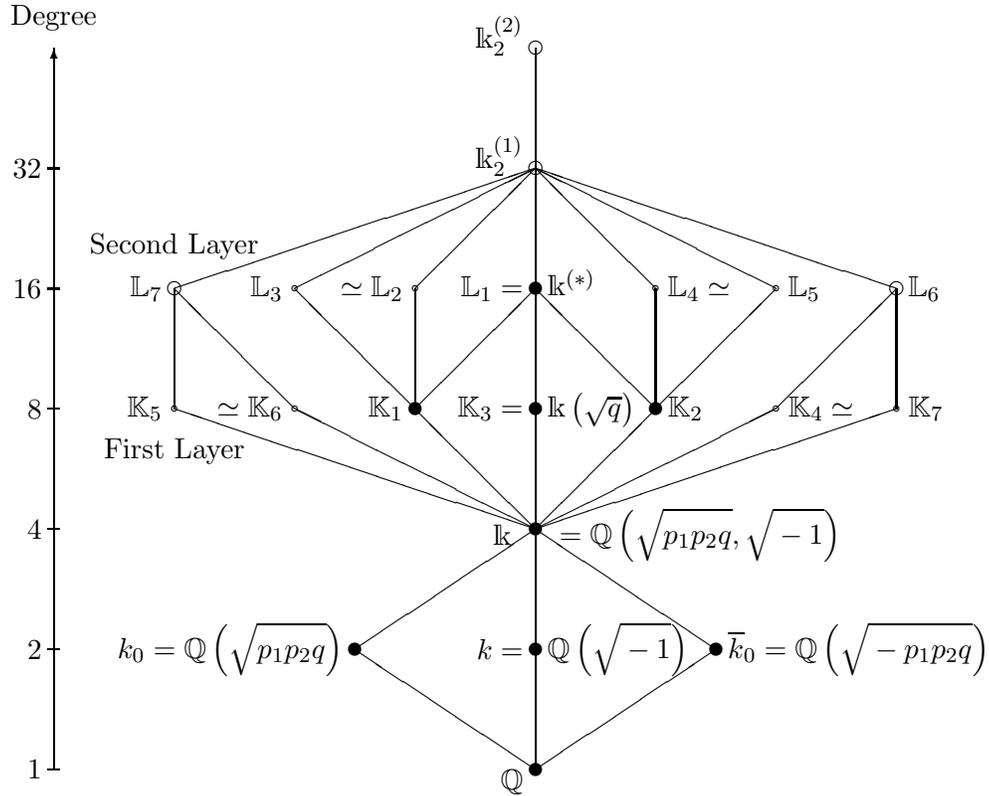
%\newpage
%--------------------------------------------------------------------------------

\subsection{Structure of $G=\mathrm{Gal}(\L/\kk)$}
 Let $\mathcal{H}_1$ (resp. $\mathcal{H}_2$, $\mathcal{H}_3$) be the prime ideal of $\kk$ above  $\pi_1$ (resp. $\pi_2$, $\pi_3$).
\begin{them}\label{10:2}
Keep the preceding  assumptions, then
\begin{enumerate}[\upshape\indent(1)]
  \item  $\mathbf{C}l_2(\kk)=\langle[\mathcal{H}_1],  [\mathcal{H}_2], [\mathcal{H}_3]\rangle\simeq(2, 2, 2)$.
  \item $\mathbf{C}l_2(\KK_3)\simeq\left\{
   \begin{array}{ll}
   (2^{n+1}, 2^{m+1})  & \text{ if } \left(\frac{p_1}{p_2}\right)=-1,\\
   (2^{n+2}, 2^m)  & \text{ if } \left(\frac{p_1}{p_2}\right)=1,
   \end{array}\right.$\\
   where $n$ and $m$  are determined by:\\  $2^{m+1}=h(-p_1p_2)$, $m\geq2$, and $2^{n}=h(p_1p_2)$, $n\geq1$.
  \item The length of the $2$-class field tower of $\kk$ is $2$.
  \item In dependence on the sign of $N(\varepsilon_{p_1p_2})$, the second $2$-class group $G=\mathrm{G}al(\kk_2^{(2)}/\kk)$ of $\kk$ is given by:
  \begin{enumerate}[\upshape\indent(i)]
\item If $N(\varepsilon_{p_1p_2})=-1$, then
\begin{align*}
G=\langle\  \rho, \tau, \sigma: \quad &\rho^4=\sigma^{2^{m+1}}=\tau^{2^{n+2}}=1,\ \sigma^{2^{m}}=\tau^{2^{n+1}},\ \rho^2=\tau^{2^{n}}\sigma^{2^{m-1}},\\
                              & [\tau, \sigma]=1,  \  [\sigma, \rho]=\sigma^{2^m-2},\  [\rho, \tau]=\tau^2\ \rangle.
\end{align*}
\item If $N(\varepsilon_{p_1p_2})=1$, then
\begin{align*}
G=\langle \rho, \tau, \sigma: \quad & \rho^4=\sigma^{2^{m}}=\tau^{2^{n+2}}=1,\   \rho^2=\tau^{2^{n+1}}\sigma^{2^{m-1}} \text{ or } \rho^2=\sigma^{2^{m-1}}, \\
                                    & [\tau, \sigma]=1,\  [\rho, \sigma]=\sigma^{2},\  [\tau, \rho]=\tau^{2^{n+1}-2} \rangle.
\end{align*}
\end{enumerate}
\item The derived subgroup of $G$ is $G'=\langle \sigma^{2}, \tau^2 \rangle\simeq\mathbf{C}l_2(\kk^{(1)}_2)$. It is of type\\
   $\left\{
   \begin{array}{ll}
   \left(2^{\min(n, m-1)},  2^{\max(m, n+1)}\right)=\left(2, 2^{\max(m, n+1)}\right) &  \text{ if } N(\varepsilon_{p_1p_2})=-1,\\
   \left(2^{n+1}, 2^{m-1}\right)  & \text{ if } N(\varepsilon_{p_1p_2})=1.
   \end{array}\right.$
\item The  coclass  of G is equal to $4$ if $\left(\frac{p_1}{p_2}\right)=1$, $N(\varepsilon_{p_1p_2})=1$ and  $\left(\frac{p_1}{p_2}\right)_4\left(\frac{p_2}{p_1}\right)_4=-1$, and it is equal to $3$ otherwise.  The nilpotency class  of $G$ is given by:\\
$\left\{ \begin{array}{ll}
m+1  \text{ if } \left(\frac{p_1}{p_2}\right)=-1,\\
m   \text{ if } \left(\frac{p_1}{p_2}\right)=1\text{ and }\left(\frac{p_1}{p_2}\right)_4\left(\frac{p_2}{p_1}\right)_4=-1,\\
 n+2  \text{ if } \left(\frac{p_1}{p_2}\right)=1 \text{ and }
 \left\{ \begin{array}{ll}
 N(\varepsilon_{p_1p_2})=-1 \text{ or } \\
   N(\varepsilon_{p_1p_2})=1 \text{ and } \left(\frac{p_1}{p_2}\right)_4=\left(\frac{p_2}{p_1}\right)_4=1.
   \end{array}\right.
\end{array}\right.$
\end{enumerate}
\end{them}
\subsection{Abelian type invariants and capitulation kernels}
Let $N_j$ denote the subgroup $N_{\KK_j/\kk}(\mathbf{C}l_2(\KK_j))$ of $\mathbf{C}l_2(\kk)$ and let $\kappa_{\KK/\kk}$ denote the (principalization- or capitulation-)kernel of the natural class
extension homomorphism  $J_{\KK/\kk}: \mathbf{C}l_2(\kk)\longrightarrow \mathbf{C}l_2(\KK)$, where $\KK$ is an unramified extension of $\kk$ within $\kk_2^{(1)}$.
\begin{them}\label{10:3}
Put  $\pi=\left(\frac{\pi_1}{\pi_3}\right)$.
  \begin{enumerate}[\rm\indent(1)]
\item For all $j\neq3$, there are exactly four classes which capitulate in $\KK_j$, but only two classes which capitulate in $\KK_3$. More precisely, we have:
 \begin{enumerate}[\rm\indent(i)]
 \item  $\kappa_{\KK_1/\kk}=\langle[\mathcal{H}_1], [\mathcal{H}_2]\rangle$,\qquad
   $\kappa_{\KK_2/\kk}=\langle[\mathcal{H}_1\mathcal{H}_2], [\mathcal{H}_3]\rangle$,\\
     $\kappa_{\KK_3/\kk}=
     \left\{\begin{array}{ll}
     \langle[\mathcal{H}_1\mathcal{H}_2]\rangle & \text{ if } N(\varepsilon_{p_1p_1})=1,\\
     \langle[\mathcal{H}_1\mathcal{H}_3]\rangle \text{ or  }\langle[\mathcal{H}_2\mathcal{H}_3]\rangle & \text{ if } N(\varepsilon_{p_1p_1})=-1.
     \end{array}
    \right.$\\
     Moreover, we have:\\
    $\left\{\begin{array}{ll}
\kappa_{\KK_1/\kk}=N_2(\mathbf{C}l_2(\KK_2)),\ \
   \kappa_{\KK_2/\kk}=N_1(\mathbf{C}l_2(\KK_1)) \text{ if } \left(\frac{p_1}{p_2}\right)=1,\\
\kappa_{\KK_1/\kk}=N_1(\mathbf{C}l_2(\KK_1)),\ \
   \kappa_{\KK_2/\kk}=N_2(\mathbf{C}l_2(\KK_2)) \text{ if } \left(\frac{p_1}{p_2}\right)=-1.
   \end{array}\right.$\\
\item
 $\kappa_{\KK_4/\kk} =\langle[\mathcal{H}_1], [\mathcal{H}_3]\rangle$, \quad
 $\kappa_{\KK_5/\kk} =\langle[\mathcal{H}_1], [\mathcal{H}_2\mathcal{H}_3]\rangle$, \quad
 $\kappa_{\KK_6/\kk} =\langle[\mathcal{H}_2], [\mathcal{H}_3]\rangle$\  and\
 $\kappa_{\KK_7/\kk} =\langle[\mathcal{H}_2], [\h\mathcal{H}_3]\rangle.$
Moreover, we have:\\
$(a)$ If $\left(\frac{p_1}{p_2}\right)=1$, then\\
 $\kappa_{\KK_4/\kk} =
\left\{\begin{array}{ll}
N_4(\mathbf{C}l_2(\KK_4)) \text{ if } \pi=1,\\
N_7(\mathbf{C}l_2(\KK_7)) \text{ if } \pi=-1,
      \end{array}
    \right.$
 $\kappa_{\KK_5/\kk} =
\left\{\begin{array}{ll}
N_5(\mathbf{C}l_2(\KK_5)) \text{ if } \pi=1,\\
N_6(\mathbf{C}l_2(\KK_6)) \text{ if } \pi=-1,
      \end{array}
    \right.$
 $\kappa_{\KK_6/\kk} =
\left\{\begin{array}{ll}
N_6(\mathbf{C}l_2(\KK_6)) \text{ if } \pi=1,\\
N_5(\mathbf{C}l_2(\KK_5)) \text{ if } \pi=-1,
      \end{array}
    \right.$
 $\kappa_{\KK_7/\kk} =
\left\{\begin{array}{ll}
N_7(\mathbf{C}l_2(\KK_7)) \text{ if } \pi=1,\\
N_4(\mathbf{C}l_2(\KK_4)) \text{ if } \pi=-1.
      \end{array}
    \right.$\\
$(b)$ If  $\left(\frac{p_1}{p_2}\right)=-1$, then\\
 $\kappa_{\KK_4/\kk} =
\left\{\begin{array}{ll}
N_4(\mathbf{C}l_2(\KK_4)) \text{ if } \pi=-1,\\
N_7(\mathbf{C}l_2(\KK_7)) \text{ if } \pi=1,
      \end{array}
    \right.$
 $\kappa_{\KK_5/\kk} =
\left\{\begin{array}{ll}
N_6(\mathbf{C}l_2(\KK_6)) \text{ if } \pi=-1,\\
N_5(\mathbf{C}l_2(\KK_5)) \text{ if } \pi=1,
      \end{array}
    \right.$\\
 $\kappa_{\KK_6/\kk} =
\left\{\begin{array}{ll}
N_5(\mathbf{C}l_2(\KK_5)) \text{ if } \pi=-1,\\
N_6(\mathbf{C}l_2(\KK_6)) \text{ if } \pi=1,
      \end{array}
    \right.$ \
 $\kappa_{\KK_7/\kk} =
\left\{\begin{array}{ll}
N_7(\mathbf{C}l_2(\KK_7)) \text{ if } \pi=-1,\\
N_4(\mathbf{C}l_2(\KK_4)) \text{ if } \pi=1.
      \end{array}
    \right.$
\end{enumerate}
\item  All the extensions $\KK_j$ satisfy  Taussky's condition $(A)$, i.e. $\kappa_{\KK_j/\kk}\cap N_j>1$ $($\cite{Ta-70}$)$.
 \item  For all $j$,   $\kappa_{\LL_j/\kk}= \mathbf{C}l_2(\kk)$  $($total $2$-capitulation$)$,  and each $\LL_j$ is of type $(A)$.
\end{enumerate}
\end{them}

\begin{them}\label{10:4}
Let $2^n=h(p_1p_2)$, $2^{m+1}=h(-p_1p_2)$, where $n\geq1$ and  $m\geq2$. Put $\beta=\left(\frac{1+i}{\pi_1}\right)\left(\frac{1+i}{\pi_3}\right)$.
\begin{enumerate}[\rm\indent(1)]
  \item The abelian type invariants of the $2$-class groups $\mathbf{C}l_2(\KK_j)$ are given by:
\begin{enumerate}[\rm\indent(i)]
\item $\mathbf{C}l_2(\KK_1)$ and  $\mathbf{C}l_2(\KK_2)$  are of type $(2, 2, 2)$ if  $\left(\frac{p_1}{p_2}\right)=1$ and of  type $(2, 4)$ otherwise.
\item If $\left(\frac{p_1}{p_2}\right)=1$,
       then $\mathbf{C}l_2(\KK_4)$,  $\mathbf{C}l_2(\KK_5)$, $\mathbf{C}l_2(\KK_6)$ and  $\mathbf{C}l_2(\KK_7)$ are of  type $(2, 2, 2)$ if $\left(\frac{\pi_1}{\pi_3}\right)=-1$, and of type $(2, 4)$ otherwise.
\item If $\left(\frac{p_1}{p_2}\right)=-1$, then we have:\\
       if $\left(\frac{\pi_1}{\pi_3}\right)=-1$, then  $\left\{\begin{array}{ll} \mathbf{C}l_2(\KK_4)\simeq\mathbf{C}l_2(\KK_7)\simeq(2, 4),\\  \mathbf{C}l_2(\KK_5)\simeq\mathbf{C}l_2(\KK_6)\simeq(2, 2, 2);\end{array}\right.$\\
       if $\left(\frac{\pi_1}{\pi_3}\right)=1$, then   $\left\{\begin{array}{ll} \mathbf{C}l_2(\KK_4)\simeq\mathbf{C}l_2(\KK_7)\simeq(2, 2, 2),\\  \mathbf{C}l_2(\KK_5)\simeq\mathbf{C}l_2(\KK_6)\simeq(2, 4).\end{array}\right.$
\end{enumerate}
  \item The abelian type invariants of the $2$-class groups   $\mathbf{C}l_2(\LL_j)$ are given by:
\begin{enumerate}[\rm\indent(i)]
\item $\mathbf{C}l_2(\LL_1)\simeq
   \left\{\begin{array}{ll}
\left(2^{\min(m,n)}, 2^{\max(m+1, n+1)}\right) & \text{ if } N(\varepsilon_{p_1p_2})=-1,\\
\left(2^{m}, 2^{n+1}\right) & \text{ if } N(\varepsilon_{p_1p_2})=1.
      \end{array}
    \right.$
\item If  $\left(\frac{p_1}{p_2}\right)=-1$ or $\left(\frac{p_1}{p_2}\right)=\left(\frac{\pi_1}{\pi_3}\right)=1$, then $\mathbf{C}l_2(\LL_2)$, $\mathbf{C}l_2(\LL_3)$, $\mathbf{C}l_2(\LL_4)$\\ and $\mathbf{C}l_2(\LL_5)$ are of type $(2, 4)$.\\
     If $\left(\frac{p_1}{p_2}\right)=-\left(\frac{\pi_1}{\pi_3}\right)=1$, then  $\mathbf{C}l_2(\LL_2)$, $\mathbf{C}l_2(\LL_3)$, $\mathbf{C}l_2(\LL_4)$\\ and $\mathbf{C}l_2(\LL_5)$ are of  type $(2, 2, 2)$.
\item If $\left(\frac{p_1}{p_2}\right)=1$, then $\mathbf{C}l_2(\LL_6)$ and $\mathbf{C}l_2(\LL_7)$ are of type $\left(2, 2^{n+2}\right)$  if  $(\frac{\pi_1}{\pi_3})=1$, otherwise we have:\\
$\mathbf{C}l_2(\LL_6)\simeq\left\{ \begin{array}{ll}
 \left(2^{m-1}, 2^{n+2}\right) & \text{ if }\beta=1,\\
  \left(2^{\min(m-1, n+1)}, 2^{\max(m, n+2)}\right)& \text{ if } \beta=-1,
   \end{array}\right.$ \\
   $\mathbf{C}l_2(\LL_7)\simeq\left\{ \begin{array}{ll}
\left(2^{\min(m-1, n+1)}, 2^{\max(m, n+2)}\right)  & \text{ if }\beta=1,\\
 \left(2^{m-1}, 2^{n+2}\right) & \text{ if } \beta=-1,
   \end{array}\right.$ \\
 If $\left(\frac{p_1}{p_2}\right)=-1$, then\\
  $\mathbf{C}l_2(\LL_6)\simeq\left\{ \begin{array}{ll}
 \left(2^{n+1}, 2^{m}\right) & \text{ if } (\frac{\pi_1}{\pi_3})=1,\\
  \left(2^{\min(m-1, n)}, 2^{\max(m+1, n+2)}\right)& \text{ if } (\frac{\pi_1}{\pi_3})=-1,
   \end{array}\right.$\\ and\\
  $\mathbf{C}l_2(\LL_7)\simeq\left\{ \begin{array}{ll}
  \left(2^{\min(m-1, n)}, 2^{\max(m+1, n+2)}\right) & \text{ if } (\frac{\pi_1}{\pi_3})=1,\\
 \left(2^{n+1}, 2^{m}\right)& \text{ if } (\frac{\pi_1}{\pi_3})=-1,
   \end{array}\right.$
 \end{enumerate}
\end{enumerate}
\end{them}
\section{Preliminary results}
\label{s:Preliminaries}
 Let  $p_1\equiv p_2\equiv -q \equiv 1 \pmod4$ be different primes  satisfying the conditions \eqref{10:18}. Put  $k_0=\QQ(\sqrt{p_1p_2q})$, $\overline{k}_0=\QQ(\sqrt{-p_1p_2q})$,  $k_1=\QQ(\sqrt{p_1p_2})$, $\overline{k}_1=\QQ(\sqrt{-p_1p_2})$,  $\varepsilon_{p_1p_2q}=x+y\sqrt{p_1p_2q}$ and $\varepsilon_{p_1p_2}=a+b\sqrt{p_1p_2}$. Let $\displaystyle\left(\frac{g, h}{p}\right)$ denote the quadratic Hilbert symbol for some prime  $p$.
 \begin{lem}\label{10:5}
 Keep the notations above. Then
 \begin{enumerate}[\rm\indent(1)]
 \item If $N(\varepsilon_{p_1p_2})=1$, then $2p_1(a\pm1)$ i.e. $2p_2(a\mp1)$ is a square in $\NN$.
 \item $p_1p_2(x\pm1)$ i.e. $q(x\mp1)$ is a square in $\NN$.
 \end{enumerate}
\end{lem}
\begin{proof}
(1) If  $N(\varepsilon_{p_1p_2})=1$, then $\left(\frac{p_1}{p_2}\right)=1$ and $a^2-1=b^2p_1p_2$. Therefore, according to
\cite[Lemma 5, p. 386]{Az-00}  and the decomposition uniqueness in $\ZZ$, there are three possible cases: $a\pm1$ or $p_1(a\pm1)$ or $2p_1(a\pm1)$ is a square in $\NN$:\\
(a) If  $a\pm1$ is a square in $\NN$, then $\left(\frac{2}{p_1}\right)=-1$, which is false.\\
(b) If $p_1(a\pm1)$ is a square in $\NN$, then $\left(\frac{p_1}{p_2}\right)=\left(\frac{2}{p_1}\right)=-1$, which is false. Thus the result.

 (2) As $N(\varepsilon_{p_1p_2q})=1$, then $x^2-1=y^2p_1p_2q$. Proceeding as in (1) we get that the only possible case is:
 $\left\{
 \begin{array}{ll}
 x\pm1=p_1p_2y_2^2,\\
 x\mp1=qy_1^2,
 \end{array}\right.$ hence the result.
\end{proof}
 \begin{lem}\label{10:7}
 Let  $p_1\equiv p_2\equiv -q \equiv 1 \pmod4$ be different primes  satisfying the conditions \eqref{10:18}.
Put  $\kk=\QQ(\sqrt{p_1p_2q}, i)$, $\KK_3^+=\QQ(\sqrt{q}, \sqrt{p_1p_2})$,  $\KK_3=\QQ(\sqrt{q}, \sqrt{p_1p_2}, i)$ and $F=\QQ(\sqrt{-q}, \sqrt{-p_1p_2})$. Then
 \begin{enumerate}[\rm\indent(1)]
 \item $\{\varepsilon_{p_1p_2q}\}$ is a $\mathrm{FSU}$ of both $\kk$ and $F$.
   \item  The $\mathrm{FSU}$'s of $\KK_3^+$ and $\KK_3$ are   $\{\varepsilon_q, \varepsilon_{p_1p_2}, \sqrt{\varepsilon_q\varepsilon_{p_1p_2q}}\}$ and $\{\varepsilon_{p_1p_2}, \sqrt{\varepsilon_q\varepsilon_{p_1p_2q}}, \sqrt{i\varepsilon_q} \}$ respectively.
   \item $q(F/\QQ)=1$, $q(\KK_3^+/\QQ)=2$ and $q(\KK_3/\QQ)=4$.
   \item  $h(F)=2h(-p_1p_2)$, $h(\KK_3^+)=2h(p_1p_2)$ and $h(\KK_3)=2h(p_1p_2)h(-p_1p_2)$.
 \end{enumerate}
\end{lem}
\begin{proof}
Determine first the  $\mathrm{FSU}$ of $\KK^+_3$. If  $N(\varepsilon_{p_1p_2})=-1$, then only $\varepsilon_q$, $\varepsilon_{p_1p_2q}$ and $\varepsilon_q\varepsilon_{p_1p_2q}$ can be squares in  $\KK^+_3$.\\
 \indent According to  \cite[Lemma 3, p. 2199]{Az-02},  we get $\varepsilon_q$ is not a square in $\QQ(\sqrt q)$; but  $2\varepsilon_q$ is. On the other hand, from Lemma  \ref{10:5}, we have $p_1p_2(x\pm1)$ is a square in $\NN$, thus $\sqrt{2\varepsilon_{p_1p_2q}}=y_1\sqrt q+y_2\sqrt{p_1p_2}$, with some $y_j\in\ZZ$.  This yields that  $\varepsilon_{p_1p_2q}$ and $2\varepsilon_{p_1p_2q}$ are not squares in $\QQ(\sqrt{p_1p_2q})$; and $2\varepsilon_{p_1p_2q}$ is in $\KK^+_3$.  Therefore  $\varepsilon_q\varepsilon_{p_1p_2q}$ is a square in  $\KK^+_3$, which implies that  $\{\varepsilon_q, \varepsilon_{p_1p_2}, \sqrt{\varepsilon_q\varepsilon_{p_1p_2q}}\}$ is a  $\mathrm{FSU}$ of $\KK^+_3$, and thus $q(\KK_3^+/\QQ)=2$.   Moreover, as  $2\varepsilon_q$ is a square in  $\QQ(\sqrt q)$, then \cite[proposition 3, p. 112]{Az-05} implies that  $\{\varepsilon_{p_1p_2}, \sqrt{\varepsilon_q\varepsilon_{p_1p_2q}}, \sqrt{i\varepsilon_q}\}$ is a  $\mathrm{FSU}$ of
 $\KK_3$, and thus  $q(\KK_3/\QQ)=4$. We find the same results if we assume that $N(\varepsilon_{p_1p_2})=1$.\\
\indent We know, from Lemma \ref{10:5}, that $x\pm1$ is not a square in $\NN$; hence \cite[Applications 1 (i), p. 114]{Az-05} implies that  $\{\varepsilon_{p_1p_2q}\}$ is a  $\mathrm{FSU}$ of $\kk$.\\
\indent For the field  $F=\QQ(\sqrt{p_1p_2q}, \sqrt{-q})$, we know, according to  \cite[Applications 2, p. 114]{Az-05},  that $\{\sqrt{-\varepsilon_{p_1p_2q}}\}$ is a  $\mathrm{FSU}$ of  $F$ if and only if  $2q(x\pm1)$ i.e. $2p_1p_2(x\pm1)$  is a square in $\NN$, which is not the case (Lemma \ref{10:5}). Hence   $\{\varepsilon_{p_1p_2q}\}$ is a  $\mathrm{FSU}$ of  $F$, and thus  $q(F/\QQ)=1$.\\
\indent Finally,  under our assumptions, P. Kaplan states in \cite{Ka76} that  $h(p_1p_2q)=h(-p_1p_2q)=4$.  Therefore, the number class formula  implies that $h(\KK_3^+)=2h(p_1p_2)$,  $h(F)=2h(-p_1p_2)$ and   $h(\KK_3)=2h(p_1p_2)h(-p_1p_2).$
\end{proof}
  \begin{lem}\label{10:9}
If $\left(\frac{p_1}{p_2}\right)=1$, then $\left(\frac{p_1}{p_2}\right)_4\left(\frac{p_2}{p_1}\right)_4=\left(\frac{\pi_1}{\pi_3}\right)$.
 \end{lem}
 \begin{proof}
 From \cite{Ka76} we get $\left(\frac{p_1}{p_2}\right)_4\left(\frac{p_2}{p_1}\right)_4=\left(\frac{p_1}{ac+bd}\right)$, where $p_1=a^2+b^2$ and $p_2=c^2+d^2$; on the other hand, according to  \cite{Lm00} we have  $\left(\frac{p_1}{ac+bd}\right)=\left(\frac{\pi_1}{\pi_3}\right)$, which implies the result.
 \end{proof}
The following results are deduced from \cite{Lm94}.
  \begin{them}\label{10:10}
  Let $p_1\equiv p_2\equiv5 \pmod8$ be different primes and put  $F_1=\QQ(\sqrt{p_1p_2}, i)$.
  \begin{enumerate}[\rm\indent(1)]
    \item   $\mathbf{C}l_2(\overline{k}_1)$ is of type $(2, 2^m)$,  $m\geq2$. It is generated  by  $\mathfrak{2}=(2, 1+\sqrt{-p_1p_2})$, the prime ideal of  $\overline{k}_1$ above $2$, and an ideal  $I$ of $\overline{k}_1$ of order $2^m$. Moreover
      \begin{center} $\left\{
       \begin{array}{ll}
       I^{2^{m-1}}\sim\mathfrak{p}_1 \text{ if } \left(\dfrac{p_1}{p_2}\right)=1,\\
       I^{2^{m-1}}\sim\mathfrak{2}\mathfrak{p}_1 \text{ if } \left(\dfrac{p_1}{p_2}\right)=-1;
       \end{array}\right.$\end{center}
        where $\mathfrak{p}_1=(p_1, \sqrt{-p_1p_2})$ is the prime ideal of $\overline{k}_1$ above $p_1$.
    \item  $\mathbf{C}l_2(k_1)$ is of type $(2^n)$,  $n\geq1$, and it is generated by $\mathfrak{2}_1$, a prime ideal of $k_1$ above $2$.
    \item  $\mathbf{C}l_2(F_1)$ is of $2$-rank equal to $2$. It is generated by $I$ and $\mathfrak{2}_{F_1}$, where  $\mathfrak{2}_{F_1}$ is a prime ideal of $F_1$ above $2$.
    \item If $\left(\frac{p_1}{p_2}\right)=-1$, then $\mathbf{C}l_2(F_1)\simeq(2^{n}, 2^{m})$; and, in $\mathbf{C}l_2(F_1)$, $I^{2^{m-1}}\sim \mathfrak{2}_{F_1}^{2^n}\sim \mathfrak{p}_1\not\sim1$.
    \item If $\left(\frac{p_1}{p_2}\right)=1$ and $N(\varepsilon_{p_1p_2})=-1$, then $$\mathbf{C}l_2(F_1)\simeq(2^{\min(n, m-1)}, 2^{\max(m-1, n+1)})$$ and  $I^{2^{m-1}}\sim \mathfrak{2}_{F_1}^{2^n}\sim \mathfrak{p}_1\not\sim1$.
      \item If  $\left(\frac{p_1}{p_2}\right)=1$ and $N(\varepsilon_{p_1p_2})=1$, then $\mathbf{C}l_2(F_1)\simeq(2^{n+1}, 2^{m-1})$; moreover  $I^{2^{m-1}}\sim \mathfrak{2}_{F_1}^{2^{n+1}}\sim \mathfrak{p}_1\sim1$.
   \end{enumerate}
 \end{them}
 Using  the above theorem, we prove the following lemma.
   \begin{lem}\label{10:11}
   Let  $\mathfrak{p}_1\mathcal{O}_{F_1}=\mathcal{P}_1\mathcal{P}_2$ and  $p_2\mathcal{O}_{F_1}=\mathcal{P}_3^2\mathcal{P}_4^2$, then in $\mathbf{C}l_2(F_1)$ we have:
     \begin{enumerate}[\rm\indent(i)]
       \item If $\left(\frac{p_1}{p_2}\right)=-1$ or $\left(\frac{p_1}{p_2}\right)=1$ and $N(\varepsilon_{p_1p_2})=-1$, then  $\mathcal{P}_1\sim \mathfrak{2}_{F_1}^{2^{n-1}}I^{2^{m-2}}$.
      \item If  $\left(\frac{p_1}{p_2}\right)=1$ and $N(\varepsilon_{p_1p_2})=1$, then  $\mathcal{P}_1\sim \mathfrak{2}_{F_1}^{2^{n}}I^{2^{m-2}}$ or  $\mathcal{P}_1\sim I^{2^{m-2}}$. Moreover $\mathcal{P}_1\mathcal{P}_3\sim \mathfrak{2}_{F_1}^{2^{n}}$.
   \end{enumerate}
    \end{lem}
\begin{proof}
Let $p_1\mathcal{O}_{\QQ(i)}=\pi_1\pi_2$, $p_2\mathcal{O}_{\QQ(i)}=\pi_3\pi_4$, $\mathfrak{p}_1\mathcal{O}_{F_1}=\mathcal{P}_1\mathcal{P}_2$ and $\mathfrak{p}_2\mathcal{O}_{F_1}=\mathcal{P}_3\mathcal{P}_4$, where $\mathfrak{p}_2$ is the prime ideal of $\overline{k}_1$ above $p_2$, then $(\pi_i)=\mathcal{P}_i^2$, for all $i$. So, according to \cite[Proposition 1]{AZT12-2}, $\mathcal{P}_i$ are not principals in $F_1$ and they are of order two. On the other hand, as the 2-rank of $\mathbf{C}l_2(F_1)$ is 2, thus  $\mathcal{P}_i \in\langle[\mathfrak{2}_{F_1}], [I]\rangle$.\\
 \indent (i) In this case,  we have $\mathfrak{p}_1\not\sim1$, hence  $\mathcal{P}_1\not\sim\mathcal{P}_2$; note that the elements of order two in $\mathbf{C}l_2(F_1)$ are $\mathfrak{2}_{F_1}^{2^{n-1}}I^{2^{m-2}}$, $\mathfrak{2}_{F_1}^{2^{n-1}}I^{-2^{m-2}}$ and $\mathfrak{2}_{F_1}^{2^{n}}\sim I^{2^{m-1}}$. Therefore $\mathcal{P}_1$ is equivalent to one of these three elements.  As $\mathcal{P}_1\sim\mathfrak{2}_{F_1}^{2^{n}}\sim I^{2^{m-1}}$ can not occur, if not we would have, by applying  the norm $N_{F_1/\overline{k}_1}$, $\mathfrak{p}_1\sim I^{2^{m}}\sim1$, which is false. Thus $\mathcal{P}_1\sim\mathfrak{2}_{F_1}^{2^{n-1}}I^{2^{m-2}}$ and $\mathcal{P}_2\sim\mathfrak{2}_{F_1}^{2^{n-1}}I^{-2^{m-2}}$ or $\mathcal{P}_1\sim\mathfrak{2}_{F_1}^{2^{n-1}}I^{-2^{m-2}}$ and  $\mathcal{P}_2\sim\mathfrak{2}_{F_1}^{2^{n-1}}I^{2^{m-2}}$. Hence with out loss of generality we can choose $\mathcal{P}_1\sim\mathfrak{2}_{F_1}^{2^{n-1}}I^{2^{m-2}}$.\\
  \indent (ii) In this case,  we have $\mathfrak{p}_1\sim\mathfrak{p}_2\sim1$, hence  $\mathcal{P}_1\sim\mathcal{P}_2$ and $\mathcal{P}_3\sim\mathcal{P}_4$. On  the other hand, according to \cite[Proposition 1]{AZT12-2}, $\mathcal{P}_1\mathcal{P}_3$ is not principal in $F_1$. To this end, note that the elements of order two in $\mathbf{C}l_2(F_1)$ are $\mathfrak{2}_{F_1}^{2^{n}}I^{2^{m-2}}$, $\mathfrak{2}_{F_1}^{2^{n}}$ and $I^{-2^{m-2}}$. Therefore $\mathcal{P}_1$ is equivalent to one of these three elements. As $\mathcal{P}_1\sim\mathfrak{2}_{F_1}^{2^{n}}$ can not occur, as otherwise, by applying  the norm $N_{F_1/\overline{k}_1}$, we get $\mathfrak{p}_1\sim \mathfrak{2}^{2^{n}}\sim1$, which is false. Thus $\mathcal{P}_1\sim I^{2^{m-2}}$ and $\mathcal{P}_3\sim\mathfrak{2}_{F_1}^{2^{n}}I^{2^{m-2}}$ or $\mathcal{P}_1\sim\mathfrak{2}_{F_1}^{2^{n-1}}I^{2^{m-2}}$ and  $\mathcal{P}_3\sim I^{2^{m-2}}$. Hence $\mathcal{P}_1\mathcal{P}_3\sim \mathfrak{2}_{F_1}^{2^{n}}$.
\end{proof}
The following lemma gives some relations
between  $N(\varepsilon_{p_1p_2})$ and the positive integers   $n$, $m$. It is a deduction from  \cite{Ka76} and \cite{Sc-34}.
  \begin{lem}\label{10:12}
Let $p_1\equiv p_2\equiv5\pmod8$ be different primes.\\
$(1)$ Suppose that $N(\varepsilon_{p_1p_2})=-1$, then
 \begin{enumerate}[\rm\indent(i)]
   \item If $\left(\frac{p_1}{p_2}\right)=-1$, then $n=1$ and $m\geq2$. Moreover:
  \begin{enumerate}[\rm\indent(a)]
   \item  $m\geq3\Leftrightarrow\displaystyle\left(\frac{p_1p_2}{2}\right)_4\displaystyle\left(\frac{2p_1}{p_2}\right)_4
\displaystyle\left(\frac{2p_2}{p_1}\right)_4=1.$
   \item $m=2\Leftrightarrow\displaystyle\left(\frac{p_1p_2}{2}\right)_4\displaystyle\left(\frac{2p_1}{p_2}\right)_4
\displaystyle\left(\frac{2p_2}{p_1}\right)_4=-1.$
 \end{enumerate}
  \item If $\left(\frac{p_1}{p_2}\right)=1$, then $n\geq2$ and $m=2$. Moreover:
    \begin{enumerate}[\rm\indent(a)]
   \item If $\left(\frac{p_1}{p_2}\right)_4=\left(\frac{p_2}{p_1}\right)_4=-1$, then $n=2$.
   \item If $\left(\frac{p_1}{p_2}\right)_4=\left(\frac{p_2}{p_1}\right)_4=1$, then  $n\geq2$.
 \end{enumerate}
 \end{enumerate}
  $(2)$ Suppose that $N(\varepsilon_{p_1p_2})=1$, then $\left(\frac{p_1}{p_2}\right)=1$,  $n\geq1$ and $m\geq2$. Moreover:
 \begin{enumerate}[\rm\indent(i)]
   \item If $\left(\frac{p_1}{p_2}\right)_4\left(\frac{p_2}{p_1}\right)_4=-1$, then $n=1$ and $m\geq3$.
   \item If $\left(\frac{p_1}{p_2}\right)_4=\left(\frac{p_2}{p_1}\right)_4=1$, then $m=2$ and $n\geq2$.
 \end{enumerate}
 \end{lem}
 \section{On the $2$-class field tower of $\kk$}
\label{s:Tower}
 In this section, we will prove that the length of the $2$-class field tower of $\kk$ is $2$. Show first the following lemma:
 \begin{lem}\label{10:13}
 Let  $p_1\equiv p_2\equiv -q \equiv 1 \pmod4$ be different primes  satisfying the conditions \eqref{10:18}.
Put  $\kk=\QQ(\sqrt{p_1p_2q}, i)$, $\KK_3^+=\QQ(\sqrt{q}, \sqrt{p_1p_2})$ and  $\KK_3=\QQ(\sqrt{q}, \sqrt{p_1p_2}, i)$. Then
\begin{enumerate}[\rm\indent(1)]
  \item The rank of $\mathbf{C}l_2(\KK_3)$ is equal to $2$.
  \item The $2$-class group of $\KK_3^+$ is cyclic.
\end{enumerate}
 \end{lem}
 \begin{proof}
(1) Put $F_2=\QQ(\sqrt q, i)$ and let $\varepsilon_q=1+\sqrt q$ denote the fundamental unit of  $\QQ(\sqrt q)$. Then, according to \cite{Az-02}, $2\varepsilon_q$ is a square in $F_2^+$, thus   \cite{Az-05} yields that the unit group of  $F_2$ is $E_{F_2}=\langle i,\sqrt{i\varepsilon_q}\rangle$. As the class number of $F_2$ is odd, then the $2$-rank of  $\KK_3$ is: $r=t-e-1$, where  $t$ is  the number of ramified primes  (finite and infinite)  in  $\KK_3/F_2$ and  $2^e=[E_{F_2}:E_{F_2}\cap N_{\KK_3/F_2}(\KK_3^\times)]$. The following diagram helps us to calculate $t$.
 \begin{figure}[H]
 \centering
 \xymatrix@R=0.22cm@C=0.4cm{&\QQ(\sqrt q)\ar[drr]\\
&&&F_2=\QQ(\sqrt q,i)\ar[drr]\\
\QQ \ar[ddr]\ar[uur] \ar[r] &\QQ(i)\ar[urr]\ar[drr]&&&&\KK_3=\QQ(\sqrt q, \sqrt{p_1p_2}, i)\\
  &&&\kk=\QQ(\sqrt{p_1p_2q},i) \ar[urr]\\
& \QQ(\sqrt{p_1p_2q})\ar[urr]  }
\caption{Primes ramifying in $\KK_3/F_2$}
\end{figure}
 Let  $p$ be a prime, we denote by  $\mathfrak{p}_M$ a prime ideal of the extension  $M/\QQ$ lying above $p$ and  $e(\mathfrak{p}_M/p)$ its ramification index. As the extension  $\KK_3/\kk$ is unramified (see \cite{AZT12-1}), then  $e(\mathfrak{p}_{F_2}/p).e(\mathfrak{p}_{\KK_3}/\mathfrak{p}_{F_2})=e(\mathfrak{p}_\kk/p)$.  Moreover,  $2$ and  $q$ are ramified in   $F_2$ and $\kk$, then  $e(\mathfrak{2}_{\KK_3}/\mathfrak{2}_{F_2})=1$ and $e(\mathfrak{q}_{\KK_3}/\mathfrak{q}_{F_2})=1$. On the other hand, for all $j\in\{1,2\}$  $\displaystyle\left(\frac{p_j}{q}\right)=-1$, hence $e(\mathfrak{p}_{j,F_2}/p_j)=1$, and since $e(\mathfrak{p}_{j,\kk}/p)=2$ so  $e(\mathfrak{p}_{j,\KK_3}/\mathfrak{p}_{j,F_2})=2$. Thus  $t=4$ and $r=3-e$.  To calculate the number $e$, we have to find the units of $F_2$  which are norms of elements of $\KK_3^\times/F_2$. Let $\mathfrak{p}$ be a prime ideal of $F_2$, then  by Hilbert symbol properties and  according to \cite[p. 205]{Gr-03},  we have:
\begin{itemize}
  \item If $\mathfrak{p}$ is not above  $p_1$ and $p_2$, then $v_\mathfrak{p}(\sqrt{i\varepsilon_q})=v_\mathfrak{p}(p_1p_2)=v_\mathfrak{p}(i)=0$, hence $\displaystyle\left(\frac{p_1p_2,\sqrt{i\varepsilon_q}}{\mathfrak{p}}\right)= \displaystyle\left(\frac{p_1p_2, i}{\mathfrak{p}}\right)=1$.
  \item If $\mathfrak{p}$ lies above  $p_1$ or $p_2$, then $v_\mathfrak{p}(\sqrt{i\varepsilon_q})=v_\mathfrak{p}( i)=0$ and $v_\mathfrak{p}(p_1p_2)=1$, thus\\
 $\begin{array}{ll}
      \displaystyle\left(\frac{p_1p_2, i}{\mathfrak{p}}\right)=
      \displaystyle\left(\frac{i}{\mathfrak{p}}\right)^{v_\mathfrak{p}(p_1p_2)}=
      \displaystyle\left(\frac{ i}{\mathfrak{p}}\right)=
      \displaystyle\left(\frac{-1}{\mathfrak{p}_{\QQ(i)}}\right)=\left(\frac{-1}{p_j}\right)=1, \text { where }  j\in\{1, 2\}.
 \end{array}$\\
  $\begin{array}{ll}
 \displaystyle\left(\frac{p_1p_2,\sqrt{i\varepsilon_q}}{\mathfrak{p}}\right) &=
 \displaystyle\left(\frac{p_1p_2,2}{\mathfrak{p}}\right)
      \displaystyle\left(\frac{p_1p_2,1+i}{\mathfrak{p}}\right)
      \displaystyle\left(\frac{p_1p_2,\sqrt{2\varepsilon_q}}{\mathfrak{p}}\right) \\
      &=  \displaystyle\left(\frac{p_1p_2,i}{\mathfrak{p}}\right)
      \displaystyle\left(\frac{p_1p_2,1+i}{\mathfrak{p}}\right)
      \displaystyle\left(\frac{p_1p_2,\sqrt{2\varepsilon_q}}{\mathfrak{p}}\right)\\
      &=\displaystyle\left(\frac{p_1p_2,1+i}{\mathfrak{p}}\right)
      \displaystyle\left(\frac{p_1p_2,\sqrt{2\varepsilon_q}}{\mathfrak{p}}\right)\\
      &=1.\displaystyle\left(\frac{2}{p_j}\right), \text{ where }j\in\{1, 2\}\\
      &=-1.
 \end{array}$
\end{itemize}
Consequently,   $e=1$, and thus $r=2$.\\
\indent Proceeding similarly, we prove that the  $2$-class group of $\KK_3^+$ is cyclic.
 \end{proof}
  \begin{them}\label{10:14}
 Let  $p_1\equiv p_2\equiv -q \equiv 1 \pmod4$ be different primes  satisfying the conditions \eqref{10:18}. Put $\kk=\QQ(\sqrt{p_1p_2q}, i)$ and  $\KK_3=\QQ(\sqrt{q}, \sqrt{p_1p_2}, i)$. Denote by $\L$ the second Hilbert $2$-class field of $\kk$ and put   $G=Gal(\kk_2^{(2)}/\kk)$, then:
\begin{enumerate}[\rm\indent(1)]
  \item The $2$-class field tower of $\kk$ stops at   $\kk_2^{(2)}$.
  \item The order of $G$ satisfies $|G|=2h(\KK_3)\geq 64$.
\end{enumerate}
 \end{them}
 \begin{proof}
 Let $\k$ denote the genus field of $\kk$, then $\k/\KK_3^+$ is a  $V_4$-extension of CM-type fields, The following diagram (Figure \ref{10:8}) clarifies this, put  $\KK=\QQ(\sqrt q,\sqrt{p_1p_2}, \sqrt{-p_2})$ and  $\mathbb{F}=\QQ(\sqrt q, \sqrt{p_1}, \sqrt{p_2})$:
 \begin{figure}[H]
 $$\xymatrix@R=0.8cm@C=0.3cm{
 &\k=\QQ(\sqrt{p_1}, \sqrt{p_2}, \sqrt q, i)\ar@{<-}[d] \ar@{<-}[dr] \ar@{<-}[ld] \\
\mathbb{K}=\QQ(\sqrt q,\sqrt{p_1p_2}, \sqrt{-p_2})\ar@{<-}[dr]& \KK_3=\QQ(\sqrt q, \sqrt{p_1p_2}, i) \ar@{<-}[d]& \mathbb{F}=\QQ(\sqrt q, \sqrt{p_1}, \sqrt{p_2})\ar@{<-}[ld]\\
&\KK_3^+=\QQ(\sqrt q, \sqrt{p_1p_2})}$$
\caption{Subfields of $\k/\KK_3^+$}\label{10:8}
\end{figure}
So, according to \cite{Lm95}, we have:
\begin{equation}\label{10:15}h(\k)=\frac{Q_{\k}}{Q_\mathbb{K}Q_{\KK_3}}\cdot\frac{\omega_{\k}}{\omega_\mathbb{K}\omega_{\KK_3}}\cdot
\frac{h(\mathbb{K})h(\KK_3)h(\mathbb{F})}{h(\KK_3^+)^2}\cdot\end{equation}
Note that $\omega_\mathbb{K}=2$ and $\omega_{\k}=\omega_{\KK_3}=12$ or $4$ according as  $q=3$ or not;  moreover  $W_{\k}=W_{\KK_3}=\langle i \rangle$ if $q\neq3$ and $W_{\k}=W_{\KK_3}=\langle i\xi\rangle$ if $q=3$, where $\xi$  is a
primitive 6\up{st} root of unity. On the first hand,  Lemma \ref{10:7} yields that  $E_{\KK_3}=\langle i, \varepsilon_{p_1p_2}, \sqrt{\varepsilon_q\varepsilon_{p_1p_2q}}, \sqrt{i\varepsilon_q} \rangle$ or $E_{\KK_3}=\langle i\xi, \varepsilon_{p_1p_2}, \sqrt{\varepsilon_q\varepsilon_{p_1p_2q}}, \sqrt{i\varepsilon_q} \rangle$, according as $q\neq3$ or not, so $Q_{\KK_3}=2$. On the other hand, \cite{Lm95} implies that  $Q_{\KK_3}|Q_{\k}[W_{\k}:W_{\KK_3}]=Q_{\k}$, hence $Q_{\k}=2$.

 At this end, we know that the  $2$-class group of  $\KK_3^+=\QQ(\sqrt q, \sqrt{p_1p_2})$ is cyclic of order $2h(p_1p_2)$ (see Lemmas \ref{10:7} and \ref{10:13}), moreover   $\mathbb{F}$ is an unramified quadratic  extension of $\KK_3^+$, then  $$h(\mathbb{F}) =\frac{h(\KK_3^+)}{2}=h(p_1p_2).$$
From Lemma \ref{10:7} we get   $E_{\KK_3^+}=\langle -1, \varepsilon_q, \varepsilon_{p_1p_2}, \sqrt{\varepsilon_q\varepsilon_{p_1p_2q}} \rangle$, so, according to  \cite[Proposition 3]{Az-99-2},    $E_{\mathbb{K}}=\langle -1, \varepsilon_q, \varepsilon_{p_1p_2}, \sqrt{\varepsilon_q\varepsilon_{p_1p_2q}} \rangle$ or  $\langle -1, \varepsilon_q, \varepsilon_{p_1p_2}, \sqrt{-\varepsilon_q\varepsilon_{p_1p_2q}} \rangle$. As  $2\varepsilon_q$ and $2\varepsilon_{p_1p_2q}$ are squares in  $\KK_3^+$, so $p_2\varepsilon_q$ and $p_2\varepsilon_{p_1p_2q}$ are not, if not we obtain that $2p_2$ is a square in  $\KK_3^+$, which is false. Similarly,    $p_2\varepsilon_q\varepsilon_{p_1p_2q}$  is not square in  $\KK_3^+$, since  $\varepsilon_q\varepsilon_{p_1p_2q}$ is. Furthermore   $p_2\sqrt{\varepsilon_q\varepsilon_{p_1p_2q}}$ and  $p_2\varepsilon_q\sqrt{\varepsilon_q\varepsilon_{p_1p_2q}}$ are not squares in  $\KK_3^+$, if not, by applying the  norm $N_{\KK_3^+/\QQ(\sqrt q)}$, we obtain that $\varepsilon_q$ is a square in $\QQ(\sqrt q)$, which is absurd;  consequently   $E_{\mathbb{K}}=\langle -1, \varepsilon_q, \varepsilon_{p_1p_2}, \sqrt{\varepsilon_q\varepsilon_{p_1p_2q}} \rangle$, this implies that $q(\mathbb{K}/\QQ)=2$. Finally,  the class number formula allows us to conclude that
 $$h(\mathbb{K})=4h(p_1p_2).$$
 Hence the equation  (\ref{10:15}) yields that
$$h(\k)=\frac{h(\KK_3)}{2},$$
  since $\omega_\mathbb{K}=2$, $W_{\mathbb{K}}=\{-1, 1\}$ and  $Q_\mathbb{K}=1$.
  Moreover, as the $2$-rank of $\mathbf{C}l_2(\KK_3)$  is equal to $2$ (Lemma \ref{10:13}), so we can apply Proposition 7 of \cite{B.L.S-98}, which says that $\KK_3$ has abelian  $2$-class field tower if and only if it has a quadratic unramified extension $\k/\KK_3$ such that $h(\k)=\frac{h(\KK_3)}{2}$; therefore $\KK_3$ has abelian  $2$-class field tower which terminates at the first stage; this implies that the $2$-class field tower of $\kk$ terminates at $\kk_2^{(2)}$, since  $\kk\subset\KK_3$. On the other hand, Lemma \ref{10:7} yields that  $h(\KK_3)=2h(p_1p_2)h(-p_1p_2)$. Moreover,  under our  conditions,  P. Kaplan affirmes in \cite[Proposition $B'_1$, p. 348]{Ka76} that  $h(-p_1p_2)\geq8$, whence  $h(\KK_3)\geq32$; which implies that $\kk_2^{(1)}\neq\kk_2^{(2)}$.

 Let us prove now that $|G|=2h(\KK_3)\geq 64$, for this we distinguish the following cases:\\
\indent Case 1: Assume that $\left(\frac{p_1}{p_2}\right)=-1$, so,  according to \cite[Proposition $B'_1$, p. 348]{Ka76}, we have:\\
a - If   $\left(\frac{p_1p_2}{2}\right)_4\left(\frac{2p_1}{p_2}\right)_4
\left(\frac{2p_2}{p_1}\right)_4=-1$, then $h(p_1p_2)=2$ and $h(-p_1p_2)=8$, thus $|G|=64$.\\
b - If $\left(\frac{p_1p_2}{2}\right)_4\left(\frac{2p_1}{p_2}\right)_4
\left(\frac{2p_2}{p_1}\right)_4=1$, then $h(p_1p_2)=2$ and  $h(-p_1p_2)\geq16$, thus  $|G|\geq128$.\\
\indent Case 2: Assume that $\left(\frac{p_1}{p_2}\right)=1$, so,  according to \cite[Proposition $B'_4$, p. 349]{Ka76}, we have:\\
 a -  If $\left(\frac{p_1}{p_2}\right)_4\left(\frac{p_2}{p_1}\right)_4=-1$, then  \cite{Sc-34} implies that   $h(p_1p_2)=2$ and   \cite{Ka76} yields  $h(-p_1p_2)\geq16$, thus   $|G|\geq128$.\\
 b - If $\left(\frac{p_1}{p_2}\right)_4=\left(\frac{p_2}{p_1}\right)_4=1$, then \cite{Ka76} implies that $h(-p_1p_2)=8$; moreover  $h(p_1p_2)\geq4$, thus $|G|\geq128$.\\
 c - If $\left(\frac{p_1}{p_2}\right)_4=\left(\frac{p_2}{p_1}\right)_4=-1$, then \cite{Ka76} implies that  $h(-p_1p_2)=8$ and   \cite{Sc-34} yields that $h(p_1p_2)=4$, thus $|G|=128$.
 This ends the proof of the theorem.
 \end{proof}
   From the proof of Theorem \ref{10:14}, we deduce the following result:
  \begin{coro}\label{10:16}
 Let  $p_1\equiv p_2\equiv -q \equiv 1 \pmod4$ be different primes  satisfying the conditions \eqref{10:18}. Put $\kk=\QQ(\sqrt{p_1p_2q}, i)$ and  $\KK_3=\QQ(\sqrt{q}, \sqrt{p_1p_2}, i)$. Denote by $\L$ the second Hilbert $2$-class field of $\kk$ and put   $G=Gal(\kk_2^{(2)}/\kk)$, then:
  \begin{enumerate}[\upshape\indent(1)]
 \item $|G|=64$ if and only if $\left(\frac{p_1}{p_2}\right)=-1$ and   $\left(\frac{p_1p_2}{2}\right)_4\left(\frac{2p_1}{p_2}\right)_4\left(\frac{2p_2}{p_1}\right)_4=-1.$
 \item  $|G|=128$ if and only if one of the following conditions holds:
  \begin{enumerate}[\upshape\indent(i)]
     \item $\left(\frac{p_1}{p_2}\right)=-1$ and $h(-p_1p_2)=16$.
    \item $\left(\frac{p_1}{p_2}\right)=1$, $\left(\frac{p_1}{p_2}\right)_4=-\left(\frac{p_2}{p_1}\right)_4$ and $h(-p_1p_2)=16$.
    \item $\left(\frac{p_1}{p_2}\right)=1$, $\left(\frac{p_1}{p_2}\right)_4=\left(\frac{p_2}{p_1}\right)_4=1$ and  $h(p_1p_2)=4$.
         \item $\left(\frac{p_1}{p_2}\right)=1$ and $\left(\frac{p_1}{p_2}\right)_4=\left(\frac{p_2}{p_1}\right)_4=-1$.
  \end{enumerate}
  \item  $|G|\geq256$ if and only if one of the following conditions holds:
  \begin{enumerate}[\upshape\indent(i)]
     \item $\left(\frac{p_1}{p_2}\right)=-1$ and $h(-p_1p_2)\geq32$.
    \item $\left(\frac{p_1}{p_2}\right)=1$, $\left(\frac{p_1}{p_2}\right)_4=-\left(\frac{p_2}{p_1}\right)_4$ and $h(-p_1p_2)\geq32$.
    \item $\left(\frac{p_1}{p_2}\right)=1$, $\left(\frac{p_1}{p_2}\right)_4=\left(\frac{p_2}{p_1}\right)_4=1$ and  $h(p_1p_2)\geq8$.
  \end{enumerate}
  \end{enumerate}
 \end{coro}
 \section{Proofs of the main results}
\label{s:Proofs}
 Recall first the following result from \cite[p. 205]{Gr-03}.
\begin{lem}\label{10:17}
 If $\mathcal{H}$ is an unramified ideal in some extension  $\KK/\kk=\kk(\sqrt{x})/\kk$, then the quadratic residue symbol is given by the Artin symbol $\varphi=\left(\frac{\kk(\sqrt{ x})/\kk}{\mathcal{H}}\right)$  as follows: $\left(\frac{x}{\mathcal{H}}\right)=\sqrt{x}^{\varphi-1}.$
 \end{lem}
\subsection{Proof of Theorem \ref{10:2}}
 (1) The assertion $\mathbf{C}l_2(\kk)=\langle[\mathcal{H}_1], [\mathcal{H}_2], [\mathcal{H}_3]\rangle\simeq(2, 2, 2)$ of  Theorem \ref{10:2} is proved in \cite{AT09} and \cite{AZT12-2}. In the following pages, we will prove the other assertions.\\
(2) {\bf Types of $\mathbf{C}l_2(\KK_3)$} To prove the second assertion we will use the techniques that F. Lemmermeyer  has used in some of his works see for example \cite{Lm94} or \cite{Lm97}. Consider the following diagram (Figure \ref{10:21}):
 \begin{figure}[H]
$$ \xymatrix@R=0.22cm@C=0.4cm{&&&& F_1=\QQ(\sqrt{p_1p_2}, i)\ar[ddrr]\\
 \\
&&\QQ(i) \ar[ddrr]\ar[uurr] \ar[rr] && \kk=\QQ(\sqrt{p_1p_2q}, i) \ar[rr] &&  \KK_3=\QQ(\sqrt{q}, \sqrt{p_1p_2}, i)\\
\\
&&&&F_2=\QQ(\sqrt{q}, i)\ar[uurr]  }$$
\caption{Subfields of $\KK_3/\QQ(i)$}\label{10:21}
\end{figure}
 Note first that  $\mathcal{H}_j$,  $\mathcal{P}_j$ coincide and remain inert in  $\KK_3$, and that  $\mathfrak{p}_1\mathcal{O}_{\KK_3}=\mathcal{H}_1\mathcal{H}_2\mathcal{O}_{\KK_3}=\mathcal{P}_1\mathcal{P}_2\mathcal{O}_{\KK_3}$,  $\mathfrak{p}_2\mathcal{O}_{\KK_3}=\mathcal{H}_3\mathcal{H}_4\mathcal{O}_{\KK_3}=\mathcal{P}_3\mathcal{P}_4\mathcal{O}_{\KK_3}$, where $\mathfrak{p}_2$ is the prime ideal of  $\overline{k}_1$ above $p_2$. On the other hand, as   $\mathcal{H}_j$, with $j\in\{1, 2, 3, 4\}$,  is unramified in $\KK_3/\kk=\kk(\sqrt q)/\kk=\kk(\sqrt{p_1p_2})/\kk$, then  Lemma \ref{10:17} implies that
 \begin{align*}
 \displaystyle\left(\frac{q}{\mathcal{H}_1\mathcal{H}_2}\right)
 &=\displaystyle\left(\frac{q}{\mathcal{H}_1}\right)\displaystyle\left(\frac{q}{\mathcal{H}_2}\right)\\
  &=\displaystyle\left(\frac{q}{p_1}\right)\displaystyle\left(\frac{q}{p_1}\right)\\
  &=1.
\end{align*}
 \begin{align*}
 \displaystyle\left(\frac{q}{\mathcal{H}_1\mathcal{H}_3}\right)
 &=\displaystyle\left(\frac{q}{\mathcal{H}_1}\right)\displaystyle\left(\frac{q}{\mathcal{H}_3}\right)\\
  &=\displaystyle\left(\frac{q}{p_1}\right)\displaystyle\left(\frac{q}{p_2}\right)\\
  &=1.
\end{align*}
 Consequently
$$N_{\KK_3/\kk}(\mathbf{C}l_2(\KK_3))=\langle[\mathcal{H}_1\mathcal{H}_2], [\mathcal{H}_1\mathcal{H}_3]\rangle,$$
since
 $$N_{\KK_3/\kk}(\mathbf{C}l_2(\KK_3))=\{[\mathcal{H}]
 \in\mathbf{C}l_2(\kk)/\displaystyle\left(\frac{2}{[\mathcal{H}]}\right)=1\}.$$
\indent  Let us determine  $\kappa_{\KK_3/\kk}$.   We know, from Lemma \ref{10:7}, that\\
$E_{\KK_3}=\langle  i, \varepsilon_{p_1p_2}, \sqrt{\varepsilon_q\varepsilon_{p_1p_2q}},  \sqrt{i\varepsilon_{q}}\rangle$ and that $E_\kk=\langle i, \varepsilon_{p_1p_2q}\rangle$, hence $N_{\KK_3/\kk}(E_{\KK_3})= \langle i, \varepsilon_{p_1p_2q}\rangle$, thus
 $[E_\kk: N_{\KK_3/\kk}(E_{\KK_3})]=1$,
 and
  $\#\kappa_{\KK_3/\kk}=2.$

(a) If $N(\varepsilon_{p_1p_2})=1$, then putting  $\varepsilon_{p_1p_2}=a+b\sqrt{p_1p_2}$ and applying  Lemma \ref{10:5}, we get $\sqrt{2\varepsilon_{p_1p_2}}=b_1\sqrt{2p_1}+b_2\sqrt{2p_2}$, where $b=2b_1b_2$. This implies that  $p_1\varepsilon_{p_1p_2}$ is a square in $\KK_3$, therefore  there exist  $\alpha\in\KK_3$ such that $(p_1)=(\alpha^2)$.  As $(\mathcal{H}_1\mathcal{H}_2)^2=(p_1)$, so  $\mathcal{H}_1\mathcal{H}_2=(\alpha)$, which yields that  $\mathcal{H}_1\mathcal{H}_2$ capitulates in $\KK_3$.

(b) If $N(\varepsilon_{p_1p_2})=-1$ and since $p_1p_2\equiv1 \pmod8$,  then, by decomposition uniqueness in $\ZZ[i]$, there exist  $b_1$ and $b_2$ in $\ZZ[i]$ such that:
 \begin{equation*}
\left\{
 \begin{array}{ll}
 a\mp i=ib_1^2\pi_1\pi_3,\\
 a\pm i=-ib_2^2\pi_2\pi_4,
 \end{array}\right.  \text{  or }
 \left\{
 \begin{array}{ll}
 a\mp i=ib_1^2\pi_2\pi_3,\\
 a\pm i=-ib_2^2\pi_1\pi_4,
 \end{array}\right.  \text{  hence }
 \end{equation*}
 \begin{equation*}
 \begin{array}{ll}\sqrt{\varepsilon_{p_1p_2}}=b_1\sqrt{\pi_1\pi_3}+b_2\sqrt{\pi_2\pi_4}
 \text{\   or } \sqrt{\varepsilon_{p_1p_2}}=b_1\sqrt{\pi_2\pi_3}+b_2\sqrt{\pi_1\pi_4}, \end{array} \end{equation*}
 where  $p_1=\pi_1\pi_2$,  $p_2=\pi_3\pi_4$ and $\pi_j$ are in $\ZZ[i]$. Thus
 \begin{equation*}
 \begin{array}{ll}\sqrt{\pi_1\pi_3\varepsilon_{p_1p_2}}=b_1\pi_1\pi_3+b_2\sqrt{p_1p_2}
 \text{  or } \sqrt{\pi_2\pi_3\varepsilon_{p_1p_2}}=b_1\pi_2\pi_3+b_2\sqrt{p_1p_2}, \end{array} \end{equation*}
 so there exist $\alpha$, $\beta$ in $\KK_3$ such $(\pi_1\pi_3)=(\alpha^2)$ or $(\pi_2\pi_3)=(\beta^2)$.  This yields that
   $\mathcal{H}_1\mathcal{H}_3=(\alpha)$ or $\mathcal{H}_2\mathcal{H}_3=(\beta)$ i.e. $\kappa_{\KK_3/\kk}=\langle[\mathcal{H}_1\mathcal{H}_3]\rangle$ or $\kappa_{\KK_3/\kk}=\langle[\mathcal{H}_2\mathcal{H}_3]\rangle$.
On the other hand, $\h\hhh$ is not principal in $\kk$. In fact,   there exist $x$ and $y$ in $\ZZ$ such that $(\h\hhh)^2=(\pi_1\pi_3)=(x+iy)$, hence  $\sqrt{x^2+y^2}=\sqrt{p_1p_2}\not\in\kk$. Thus  \cite[Proposition 1]{AZT12-2} implies the result.  Similarly, we show that $\h\hhhh$, $\hh\hhhh$ and  $\hh\hhh$ are not principal in $\kk$.\\
\indent As $p_1p_2(x\pm1)$  is a square in $\NN$ (Lemma \ref{10:5}), so it is easy to see that  $\h\hh\hhh\hhhh$ is principal in $\kk$, hence  $\kappa_{\KK_3/\kk}=\langle[\h\hhh]\rangle$ or $\langle[\hh\hhh]\rangle$. Consequently
      $$\kappa_{\KK_3/\kk}=\left\{
 \begin{array}{ll}
\langle[\h\hh]\rangle=\langle[\hhh\hhhh]\rangle \text{ if } N(\varepsilon_{p_1p_2})=1,\\
\langle[\h\hhh]\rangle \text{ or } \langle[\hh\hhh]\rangle \text{ if }  N(\varepsilon_{p_1p_2})=-1.
 \end{array}\right.$$
 \indent From the Figure \ref{10:21}, we see that  $\KK_3/F_1$ and $\KK_3/F_2$ are ramified extensions, whereas    $\KK_3/\kk$  is not. Therefore, by the class field theory, we deduce  that  $ [\mathbf{C}l_2(\kk):N_{\KK_3/\kk}(\mathbf{C}l_2(\KK_3))]=2$, $\mathbf{C}l_2(F)=N_{\KK_3/F}(\mathbf{C}l_2(\KK_3))$ and  $\mathbf{C}l_2(F_1)=N_{\KK_3/F_1}(\mathbf{C}l_2(\KK_3))$, thus  Theorem  \ref{10:10} implies that $$N_{\KK_3/F_1}(\mathbf{C}l_2(\KK_3))=\langle[\mathfrak{2}_{F_1}], [I]\rangle.$$
  Hence there exist ideals  $\mathfrak{P}$  and  $\mathfrak{A}$ in $\KK_3$ such that $N_{\KK_3/F_1}(\mathfrak{P})\sim I$, $N_{\KK_3/F_1}(\mathfrak{A})\sim \mathfrak{2}_{F_1}$,   $N_{\KK_3/\kk}(\mathfrak{P})\in \mathbf{C}l_2(\kk)$ and  $N_{\KK_3/\kk}(\mathfrak{A})\in \mathbf{C}l_2(\kk)$. Note that  $\mathfrak{A}$ is an ideal in $\KK_3$ above $2$. We prove that  $N_{\KK_3/\kk}(\mathfrak{P})\sim \mathcal{H}_1\mathcal{H}_2$ and $N_{\KK_3/\kk}(\mathfrak{A})\sim \mathcal{H}_1\mathcal{H}_3$ or $\mathcal{H}_2\mathcal{H}_3$ (see Lemma \ref{10:19} below). Thus we claim that
     $$\left\{
 \begin{array}{ll}
 \mathfrak{P}^{2}\sim I, \text{ if } N(\varepsilon_{p_1p_2})=1,\\
\mathfrak{P}^{2}\sim \mathcal{H}_1\mathcal{H}_2I, \text{ if }  N(\varepsilon_{p_1p_2})=-1.
 \end{array}\right.$$ and
 $$\left\{
 \begin{array}{ll}
 \mathfrak{A}^{2}\sim \mathcal{H}_1\mathcal{H}_3\mathfrak{2}_{F_1}, \text{ if }  N(\varepsilon_{p_1p_2})=1,\\
 \mathfrak{A}^{2}\sim \mathfrak{2}_{F_1}, \text{ if } N(\varepsilon_{p_1p_2})=-1.
 \end{array}\right.$$
 Before showing this, note that in $\mathbf{C}l_2(F_1)$ and in the case where $N(\varepsilon_{p_1p_2})=1$ we have:  $\mathfrak{p}_1\sim\mathcal{P}_1\mathcal{P}_2\sim1$ and $\mathfrak{p}_2\sim\mathcal{P}_3\mathcal{P}_4\sim1$, hence $\mathcal{P}_1\sim\mathcal{P}_2$ and $\mathcal{P}_3\sim\mathcal{P}_4$. Therefore, in $\mathbf{C}l_2(\KK_3)$, we get  $\h\sim\hh$ and $\hhh\sim\hhhh$.

 To this end, let  $t$ and $s$  be the elements of order $2$ in $\mathrm{Gal}(\KK_3/\QQ(i))$ which fix  $F_1$ and  $\kk$ respectively. Using the identity $2+(1+t+s+ts)=(1+t)+(1+s)+(1+ts)$ of the group ring $\ZZ[Gal(F_1/\QQ)]$ and observing that $\QQ(i)$ and  $F_2$  have
odd class numbers we find:
$$\mathfrak{P}^2\sim \mathfrak{P}^{1+t}\mathfrak{P}^{1+s}\mathfrak{P}^{1+ts}\sim \mathcal{H}_1\mathcal{H}_2I \text{  and  } $$
$$\mathfrak{A}^2\sim \mathfrak{A}^{1+t}\mathfrak{A}^{1+s}\mathfrak{A}^{1+ts}\sim \mathcal{H}_1\mathcal{H}_3\mathfrak{2}_{F_1}\text{ or } \mathcal{H}_2\mathcal{H}_3\mathfrak{2}_{F_1}.$$
 As, in $\mathbf{C}l_2(\KK_3)$, we have $\mathcal{H}_1\mathcal{H}_2\sim 1$ if $N(\varepsilon_{p_1p_2})=1$ and $\mathcal{H}_1\mathcal{H}_3\sim 1$  or $\mathcal{H}_2\mathcal{H}_3\sim 1$ if $N(\varepsilon_{p_1p_2})=-1$, so the result claimed.  Thus Theorem \ref{10:10} implies that, in $\mathbf{C}l_2(\KK_3)$, we have:
    $$\left\{
 \begin{array}{ll}
 \mathfrak{P}^{2^{m}}\sim I^{2^{m-1}}\sim\mathfrak{p}_1\sim1 &\text{ if } N(\varepsilon_{p_1p_2})=1,\\
\mathfrak{P}^{2^{m}}\sim I^{2^{m-1}}\sim\mathfrak{p}_1\not\sim1 \text{ and }\mathfrak{P}^{2^{m+1}}\sim1  &\text{ if }  N(\varepsilon_{p_1p_2})=-1.
 \end{array}\right.$$ and
 $$\left\{
 \begin{array}{ll}
 \mathfrak{A}^{2^{n+2}}\sim \mathfrak{2}_{F_1}^{2^{n+1}}\sim1 & \text{ if }  N(\varepsilon_{p_1p_2})=1,\\
 \mathfrak{A}^{2^{n+1}}\sim \mathfrak{2}_{F_1}^{2^{n}}\not\sim1 \text{ and } \mathfrak{A}^{2^{n+2}}\sim1 &  \text{ if } N(\varepsilon_{p_1p_2})=-1.
 \end{array}\right.$$
Moreover, Lemma \ref{10:11} and  Theorem \ref{10:10} yield that:  \\
  - If $\left(\frac{p_1}{p_2}\right)=-1$ or $\left(\frac{p_1}{p_2}\right)=1$ and $N(\varepsilon_{p_1p_2})=-1$, then
    $$\left\{
 \begin{array}{ll}
 \mathfrak{P}^{2^{m}}\sim I^{2^{m-1}}\sim\mathfrak{2}_{F_1}^{2^{n}}\sim\mathfrak{A}^{2^{n+1}}\sim\mathfrak{p}_1\sim\h\hh,\text{ and }\\
    \h\sim\mathcal{P}_1\sim \mathfrak{2}_{F_1}^{2^{n-1}}I^{2^{m-2}}\sim\mathfrak{A}^{2^{n}}\mathfrak{P}^{2^{m-1}}.
 \end{array}\right.$$
- If $\left(\frac{p_1}{p_2}\right)=1$ and $N(\varepsilon_{p_1p_2})=1$, then
     $$\left\{
 \begin{array}{ll}
 \mathfrak{P}^{2^{m}}\sim I^{2^{m-1}}\sim\mathfrak{2}_{F_1}^{2^{n+1}}\sim\mathfrak{A}^{2^{n+2}}\sim\mathfrak{p}_1\sim\h\hh\sim1,\\
    \h\hhh\sim\mathcal{P}_1\mathcal{P}_3\sim \mathfrak{2}_{F_1}^{2^{n}}\sim\mathfrak{A}^{2^{n+1}},\\
    \h\sim\mathfrak{A}^{2^{n+1}}\mathfrak{P}^{2^{m-1}}\text{ and }\hhh\sim\mathfrak{P}^{2^{m-1}}\text{ or }
    \hhh\sim\mathfrak{A}^{2^{n+1}}\mathfrak{P}^{2^{m-1}}\text{ and }\h\sim\mathfrak{P}^{2^{m-1}}.
  \end{array}\right.$$
Finally, note that for all   $i\leq n$ and  $j\leq m-1$, we have  $\mathfrak{A}^{2^{i}}\mathfrak{P}^{2^{j}}\not\sim1$. As otherwise we would have, by applying the norm $N_{\KK_3/F_1}$,  $\mathfrak{2}_{F_1}^{2^{i}}\mathfrak{P}^{2^{j+1}}\sim1$, which is absurd.

 Taking into account Lemma \ref{10:12}, and since the $2$-rank of  $\KK_3$ is $2$ and its   $2$-class number is  $h(\KK_3)=2h(p_1p_2)h(-p_1p_2)$ (see Lemmas \ref{10:7}  and \ref{10:13}), so the results  that we have just prove, we get the following conclusion:\\
{ \em Conclusion}\\
\indent $\bullet$ If  $\left(\frac{p_1}{p_2}\right)=-1$, then  $\langle[\mathfrak{A}], [\mathfrak{P}]\rangle$  is a subgroup of $\mathbf{C}l_2(\KK_3)$ of type\\
 $(2^{m+1}, 2^{n+1})$,  thus  $$\mathbf{C}l_2(\KK_3)=\langle[\mathfrak{A}], [\mathfrak{P}]\rangle\simeq (2^{n+1}, 2^{m+1})=(2^{2}, 2^{m+1}).$$
\indent $\bullet$ If  $\left(\frac{p_1}{p_2}\right)=1$ and $N(\varepsilon_{p_1p_2})=-1$, then $\langle[\mathfrak{A}], [\mathfrak{P}]\rangle$ is a subgroup  of  $\mathbf{C}l_2(\KK_3)$ of type $(2^{\min(m, n+1)}, 2^{\max(m+1, n+2)})=(2^{m}, 2^{ n+2})$,  thus
 $$\mathbf{C}l_2(\KK_3)=\langle[\mathfrak{A}], [\mathfrak{P}]\rangle\sim (2^{m}, 2^{ n+2})=(2^{2}, 2^{ n+2}).$$
\indent $\bullet$ If   $\left(\frac{p_1}{p_2}\right)=1$ and $N(\varepsilon_{p_1p_2})=1$, then $\langle[\mathfrak{A}], [\mathfrak{P}]\rangle$ is a subgroup of  $\mathbf{C}l_2(\KK_3)$ of type  $(2^{m}, 2^{n+2})$,  thus
 $$\mathbf{C}l_2(\KK_3)=\langle[\mathfrak{A}], [\mathfrak{P}]\rangle\simeq(2^{m}, 2^{n+2}).$$
 Moreover
 $$
 \mathbf{C}l_2(\KK_3)\simeq
 \left\{
 \begin{array}{ll}
 (2^{m}, 2^{3}) \text{ if } \left(\frac{p_1}{p_2}\right)_4\left(\frac{p_2}{p_1}\right)_4=-1,\\
 (2^{2}, 2^{n+2}) \text{ if } \left(\frac{p_1}{p_2}\right)_4=\left(\frac{p_2}{p_1}\right)_4=1
 \end{array}\right.
 $$
  (4) {\bf Computation of $\mathrm{Gal}(\kk_2^{(2)}/\kk)$}.
  Put $L=\kk_2^{(2)}$, the Hilbert $2$-class field of $\kk$, and denote by  $\displaystyle\left(\frac{L/\KK_3}{P}\right)$ the Artin  symbol for  the normal extension   $L/\KK_3$, then $\sigma=\displaystyle\left(\frac{L/\KK_3}{\mathfrak{P}}\right)$ and $\tau=\displaystyle\left(\frac{L/\KK_3}{\mathcal{\mathfrak{A}}}\right)$ generate the  abelian subgroup  $Gal(L/\KK_3)$ of $G=Gal(L/\kk)$. Put also  $\rho=\displaystyle\left(\frac{L/\kk}{\mathcal{H}_1}\right)$, then $\rho$  restricts to the nontrivial automorphism of  $\KK_3/\kk$, since  $\mathcal{H}_1$ is not norm in  $\KK_3/\kk$ ($\h$ remains inert in  $\KK_3$). Thus
 $$G=\mathrm{Gal}(L/\kk)=\langle\rho, \tau, \sigma\rangle.$$
For the Artin symbol properties we can see  \cite{La-94}.\\
Note finally that  $|G|=2|\mathrm{Gal}(L/\KK_3)|=2^{n+m+3} $.

To continue we need the following result:
\begin{lem}\label{10:19}   In $\mathbf{C}l_2(\kk)$, we have  $N_{\KK_3/\kk}(\mathfrak{P})\sim \mathcal{H}_1\mathcal{H}_2$ and  $N_{\KK_3/\kk}(\mathfrak{A})\sim \mathcal{H}_1\mathcal{H}_3$ or $\mathcal{H}_2\mathcal{H}_3$. Moreover
\begin{enumerate}[\upshape\indent(1)]
\item Assume  $\left(\frac{p_1}{p_2}\right)=-1$ and  put $\mathrm{I}= \left(\frac{p_1p_2}{2}\right)_4\left(\frac{2p_1}{p_2}\right)_4\left(\frac{2p_2}{p_1}\right)_4$.
\begin{enumerate}[\upshape\indent(i)]
  \item  If  $\mathrm{I}=\left(\frac{\pi_1}{\pi_3}\right)$, then $N_{\KK_3/\kk}(\mathfrak{A})\sim \mathcal{H}_1\mathcal{H}_3$.
  \item If $\mathrm{I}=-\left(\frac{\pi_1}{\pi_3}\right)$, then  $N_{\KK_3/\kk}(\mathfrak{A})\sim \mathcal{H}_2\mathcal{H}_3$.
\end{enumerate}
\item Assume  $\left(\frac{p_1}{p_2}\right)=1$ and  put $\beta=\left(\frac{1+i}{\pi_1}\right)\left(\frac{1+i}{\pi_3}\right)$.
\begin{enumerate}[\upshape\indent(i)]
  \item  If  $\beta=1$, then $N_{\KK_3/\kk}(\mathfrak{A})\sim \mathcal{H}_1\mathcal{H}_3$.
  \item If $\beta=-1$, then  $N_{\KK_3/\kk}(\mathfrak{A})\sim \mathcal{H}_2\mathcal{H}_3$.
\end{enumerate}
\end{enumerate}
\end{lem}
\begin{proof} Recall that  $N_{\KK_3/\kk}(\mathbf{C}l_2(\KK_3))=\langle[\mathcal{H}_1\mathcal{H}_2], [\mathcal{H}_1\mathcal{H}_3]\rangle$. Choose a prime ideal $\mathfrak{R}$ in $\KK_3$ such that $[\mathfrak{R}]=[\mathfrak{P}]$, this  is always
possible by Chebotarev's theorem,  thus   $\mathscr{R}=N_{\KK_3/\kk}(\mathfrak{R})$ is a prime ideal in $\kk$. If $N_{\KK_3/\kk}(\mathfrak{P})\sim \mathcal{H}_1\mathcal{H}_3$, then $\mathscr{R}\sim\mathcal{H}_1\mathcal{H}_3$ (equivalence in $\mathbf{C}l_2(\kk)$). Hence the prime ideal $\mathfrak{r}=N_{\kk/k_0}(\mathscr{R})$ of  $k_0$ is equivalent, in  $\mathbf{C}l_2(k_0)$, to $\widetilde{\mathfrak{2}}\sim P_1P_2\sim \widetilde{\mathfrak{q}}$. It should be noted that  $\mathbf{C}l_2(k_0)$ is of  type $(2, 2)$ and it is generated by $P_1$ and $P_2$, the  prime ideals of $k_0$ above $p_1$ and $p_2$ respectively.
  Therefore  $\mathfrak{r}\sim \widetilde{\mathfrak{q}}$ and  $\mathfrak{r}\sim P_1P_2$, these imply that $\pm rq=X^2-y^2p_1p_2q$ and  $\pm rp_1p_2=U^2-v^2p_1p_2q$, with some $X$, $y$,  $U$ and  $v$ in $\ZZ$.  Putting  $X=xq$ and $U=up_1p_2$, we get $\pm r=x^2q-y^2p_1p_2$ and $\pm r=u^2p_1p_2-v^2q$, from which we deduce that  $(\frac{\pm r}{q})=(\frac{-p_1p_2}{q})=-1$ and $(\frac{\pm r}{q})=1$, which is a contradiction. Similarly, we show that  $N_{\KK_3/\kk}(\mathfrak{P})\not\sim \mathcal{H}_2\mathcal{H}_3$. Finally,  the equivalence $N_{\KK_3/\kk}(\mathfrak{P})\sim1$ can
not occur since the order of $\sigma$ is strictly greater than 1. Consequently  $N_{\KK_3/\kk}(\mathfrak{P})\sim \mathcal{H}_1\mathcal{H}_2$.

 Let $\mathcal{H}_0$ denote a prime ideal of  $\kk$ above  $2$. If   $N_{\KK_3/\kk}(\mathfrak{A})\sim\mathcal{H}_1\mathcal{H}_2$, then\\ $\mathcal{H}_0\sim\mathcal{H}_1\mathcal{H}_2$, thus  $\mathcal{H}_0\mathcal{H}_1\mathcal{H}_2\sim1$. Hence $\mathcal{H}_0\mathcal{H}_1\mathcal{H}_2=(\alpha)$, with some  $\alpha\in\kk$. We have two cases to distinguish:\\
 \indent 1\up{st} case: If $q\equiv3\pmod8$, then there is only one prime ideal $\mathcal{H}_0$  in $\kk$ above $2$, thus  $(\mathcal{H}_0\mathcal{H}_1\mathcal{H}_2)^2=(2p_1)$. Therefore   $\mathcal{H}_0\mathcal{H}_1\mathcal{H}_2$ is principal in $\kk$ if and only if  $p_1(x\pm1)$ or $2p_1(x\pm1)$ is a square in $\NN$ (see \cite[Remark 1]{AZT12-2}), which is not the case (Lemma \ref{10:5}).\\
  \indent 2\up{nd} case: If $q\equiv7\pmod8$, then there are two prime ideals    $\mathcal{H}_0$ and $\mathcal{H}_0'$ in  $\kk$ above $2$. Thus   $\mathcal{H}_0\mathcal{H}_1\mathcal{H}_2=(\alpha)$ implies, by applying the norm $N_{\kk/k_0}$, that $P_0=(\alpha')$, where $P_0$ is the prime ideal of  $k_0$ above $2$, hence $(2)=(\alpha'^2)$, this in turn yields that $2\varepsilon_{p_1p_2q}$ is a square in  $\QQ(\sqrt{p_1p_2q})$ i.e.  $x\pm1$   is a square in  $\NN$, which is not the case (Lemma \ref{10:5}).  Finally,  the equivalence $N_{\KK_3/\kk}(\mathfrak{A})\sim1$ can
not occur since the order of $\tau$ is strictly greater than 1. From which we conclude that  $N_{\KK_3/\kk}(\mathfrak{A})\sim \mathcal{H}_1\mathcal{H}_3\text{ or } \mathcal{H}_2\mathcal{H}_3.$\\
\indent Suppose that $\mathrm{I}=-1$ and $\left(\frac{\pi_1}{\pi_3}\right)=1$. If $N_{\KK_3/\kk}(\mathfrak{A})\sim \mathcal{H}_1\mathcal{H}_3$, then\\ $N_{\kk/\QQ(i)}(N_{\KK_3/\kk}(\mathfrak{A}))\sim \pi_1\pi_3.$ On the other hand, $\mathfrak{A}$ is a prime ideal of $\KK_3$ above $2$, thus $N_{\kk/\QQ(i)}(N_{\KK_3/\kk}(\mathfrak{A}))\sim 1+i$. Hence $\pi_1\pi_3\sim 1+i$. Therefore Hilbert symbol properties and Lemma \ref{10:17} imply that
$\left(\frac{\pi_2\pi_4}{\mathcal{H}_1\mathcal{H}_3}\right)=\left(\frac{\pi_2\pi_4}{\pi_1\pi_3}\right)=\left(\frac{\pi_2\pi_4}{1+i}\right)=
\left(\frac{1+i}{\pi_2}\right)\left(\frac{1+i}{\pi_4}\right).$
Thus Lemma \ref{10:20} (see below)  yields that
  $\left(\frac{1+i}{\pi_1}\right)\left(\frac{1+i}{\pi_3}\right)=1.$
Which is absurd, since, according to \cite[Proposition 1]{AZT14-3}, $\mathrm{I}=\left(\frac{\pi_1}{\pi_3}\right)\left(\frac{1+i}{\pi_1}\right)\left(\frac{1+i}{\pi_3}\right)$. The other assertions are proved similarly using .
\end{proof}
We can now establish the following relations (equivalences are in $\mathbf{C}l_2(\KK_3)$):\\
$\bullet$  $[\tau, \sigma]=1$.\\
$\bullet$ $\rho^2=\displaystyle\left(\frac{L/\kk}{\mathcal{H}_1^2}\right)=
\displaystyle\left(\frac{L/\kk}{N_{\KK_3/\kk}(\mathcal{H}_1)}\right)=\displaystyle\left(\frac{L/\KK_3}{\mathcal{H}_1}\right)$, thus $\rho^4=1$ .\\
$\bullet$ $\tau\rho^{-1}\tau\rho=\displaystyle\left(\frac{L/\KK_3}{\mathfrak{A}^{1+\rho}}\right)=
\left\{
 \begin{array}{ll}
 1 &\text{  if } N(\varepsilon_{p_1p_2})=-1,\\
 \tau^{2^{n+1}}& \text{  if }  N(\varepsilon_{p_1p_2})=1,
 \end{array}\right.$\\
     because $\mathfrak{A}^{1+\rho}=N_{\KK_3/\kk}(\mathfrak{A})\sim \mathcal{H}_1\mathcal{H}_3\sim\mathfrak{A}^{2^{n+1}}$ if $N(\varepsilon_{p_1p_2})=1$, otherwise we get $\mathfrak{A}^{1+\rho}=N_{\KK_3/\kk}(\mathfrak{A})\sim \mathcal{H}_1\mathcal{H}_3$ or $\mathcal{H}_2\mathcal{H}_3$.    Since $\mathcal{H}_1\mathcal{H}_3$ and $\mathcal{H}_2\mathcal{H}_3$ are norms in  $\KK_3/\kk$, so, with out loss of generality, we can  assume that  $N_{\KK_3/\kk}(\mathfrak{A})\sim1$, because  $\kappa_{\KK_3/\kk}=\langle[\mathcal{H}_1\mathcal{H}_3]\rangle$ or $\langle[\mathcal{H}_2\mathcal{H}_3]\rangle$. Thus\\
   $[\tau, \rho]=\left\{
 \begin{array}{ll}
 \tau^{-1}\rho^{-1}\tau\rho=\tau^{-2} &\text{  if  } N(\varepsilon_{p_1p_2})=-1,\\
 \tau^{2^{n+1}-2} &\text{  if  }  N(\varepsilon_{p_1p_2})=1.
 \end{array}\right.$\\
$\bullet$ $\sigma\rho^{-1}\sigma\rho=\displaystyle\left(\frac{L/\KK_3}{\mathfrak{P}^{1+\rho}}\right)=
\left\{
 \begin{array}{ll}
 1 &\text{  if  } N(\varepsilon_{p_1p_2})=1,\\
 \sigma^{2^m} &\text{  if  }  N(\varepsilon_{p_1p_2})=-1;
 \end{array}\right.$
 since
  $\mathfrak{P}^{1+\rho}=N_{\KK_3/\kk}(\mathfrak{P})\sim \mathcal{H}_1\mathcal{H}_2\sim
  \left\{
 \begin{array}{ll}
 1 &\text{  if  } N(\varepsilon_{p_1p_2})=1,\\
\mathfrak{P}^{2^m} &\text{  if  }  N(\varepsilon_{p_1p_2})=-1.
 \end{array}\right.$
 Thus \\ $[\sigma, \rho]=
   \left\{
 \begin{array}{ll}
 \sigma^{-2} &\text{  if  } N(\varepsilon_{p_1p_2})=1,\\
\sigma^{2^m-2} &\text{  if  }  N(\varepsilon_{p_1p_2})=-1.
 \end{array}\right.$\\
$\bullet$ If $\left(\frac{p_1}{p_2}\right)=-1$, then $n=1$,  $m\geq2$ and\\
  $ \left\{
 \begin{array}{ll}
\rho^4=\sigma^{2^{m+1}}=\tau^{2^{n+2}}=\tau^{2^{3}}=1, \\
\sigma^{2^{m}}=\tau^{2^{n+1}}=\tau^{4},\\
 \rho^2=\tau^{2^{n}}\sigma^{2^{m-1}}=\tau^{2}\sigma^{2^{m-1}};
  \end{array}\right.$
    since $\mathfrak{A}^{2^{n+2}}\sim\mathfrak{P}^{2^{m+1}}\sim1$,  $\mathcal{H}_1\sim\mathfrak{A}^{2^{n}}\mathfrak{P}^{2^{m-1}}$ and $\mathfrak{A}^{2^{n+1}}\sim\mathfrak{P}^{2^{m}}$. Moreover
   $[\tau, \rho]=\tau^{-2}$ and $[\sigma, \rho]=\sigma^{2^m-2}.$\\
$\bullet$ If $\left(\frac{p_1}{p_2}\right)=1$ and $N(\varepsilon_{p_1p_2})=-1$, then $m=2$,  $n\geq2$ and \\
  $ \left\{
 \begin{array}{ll}
\rho^4=\tau^{2^{n+2}}=\sigma^{2^{m+1}}=\sigma^{2^{3}}=1, \\
\tau^{2^{n+1}}=\sigma^{2^{m}}=\sigma^{4},\\
 \rho^2=\tau^{2^{n}}\sigma^{2^{m-1}}=\tau^{2^{n}}\sigma^{2};
  \end{array}\right.$
    since $\mathfrak{A}^{2^{n+2}}\sim\mathfrak{P}^{2^{m+1}}\sim1$,  $\mathcal{H}_1\sim\mathfrak{A}^{2^{n}}\mathfrak{P}^{2^{m-1}}$ and  $\mathfrak{A}^{2^{n+1}}\sim\mathfrak{P}^{2^{m}}$. Moreover   $[\tau, \rho]=\tau^{-2}$ and  $[\sigma, \rho]=\sigma^{2}$.\\
$\bullet$ If  $\left(\frac{p_1}{p_2}\right)=1$ and  $N(\varepsilon_{p_1p_2})=1$, then $m\geq2$,  $n\geq1$  and \\
  $ \left\{
 \begin{array}{ll}
\rho^4=\tau^{2^{n+2}}=\sigma^{2^{m}}=1, \\
 \rho^2=\tau^{2^{n+1}}\sigma^{2^{m-1}} \text{ or } \rho^2=\sigma^{2^{m-1}};
  \end{array}\right.$
    since $\mathfrak{A}^{2^{n+2}}\sim\mathfrak{P}^{2^{m}}\sim1$ and  $\mathcal{H}_1\sim\mathfrak{A}^{2^{n+1}}\mathfrak{P}^{2^{m-1}}$ or $\mathcal{H}_1\sim\mathfrak{P}^{2^{m-1}}$. Moreover  $[\tau, \rho]=\tau^{2^{n+1}-2}$ and $[\sigma, \rho]=\sigma^{-2}$.

\indent (5) {\bf Types of  $\mathbf{C}l_2(\M)$.} We know that $[\tau,\sigma]=1$, $[\sigma, \rho]=\sigma^{-2}$ or $\sigma^{2^m-2}$ and $[\tau, \rho]=\tau^{-2}$ or $\tau^{2^{n+1}-2}$, and since  $\langle\sigma^{2^m-2}\rangle\simeq\langle\sigma^{-2}\rangle$ and  $\langle\tau^{2^{n+1}-2}\rangle\simeq\langle\tau^{-2}\rangle$, then  $G'\simeq\langle\sigma^2, \tau^2\rangle$, where $G'$ is the derived group of $G$,  hence\\
 $\mathbf{C}l_2(\kk^{(1)}_2)\simeq\left\{
   \begin{array}{ll}
   (2, 2^{m})  &\text{  if  } \left(\frac{p_1}{p_2}\right)=-1,\\
   (2, 2^{n+1}) &\text{  if  } \left(\frac{p_1}{p_2}\right)=1 \text{ and } N(\varepsilon_{p_1p_2})=-1,\\
   (2^{m-1}, 2^{2}) &\text{  if  } \left(\frac{p_1}{p_2}\right)=1,\  N(\varepsilon_{p_1p_2})=1 \text{ and } \left(\frac{p_1}{p_2}\right)_4\left(\frac{p_2}{p_1}\right)_4=-1,\\
   (2, 2^{n+1}) &\text{  if  } \left(\frac{p_1}{p_2}\right)=1,\  N(\varepsilon_{p_1p_2})=1 \text{ and } \left(\frac{p_1}{p_2}\right)_4=\left(\frac{p_2}{p_1}\right)_4=1.
   \end{array}\right.$
\indent (6) {\bf The coclass of $G$.} The lower central series of $G$ is defined inductively by $\gamma_1(G)=G$ and $\gamma_{i+1}(G)=[\gamma_i(G),G]$, that is the subgroup of $G$ generated by the set $\{[a, b]=a^{-1}b^{-1}ab/ a\in \gamma_i(G), b\in G\}$, so  the coclass of $G$ is defined to be $cc(G) = h-c$, where $|G|=2^h$ and  $c=c(G)$ is the nilpotency class of $G$, that is  the smallest positive integer $c$ satisfying  $\gamma_{c+1}(G)=1$. We easily get \\
$\gamma_1(G)=G$.\\
$\gamma_2(G)=G'=\langle\sigma^2, \tau^2\rangle$.\\
$\gamma_3(G)=[G',G]=\langle\sigma^4, \tau^4\rangle$.\\
Then  Proposition \ref{24}(6) (below) implies that  $\gamma_{j+1}(G)=[\gamma_j(G),G]=\langle\sigma^{2^{j}}, \tau^{2^{j}}\rangle$.\\
  If   $\left(\frac{p_1}{p_2}\right)=-1$, then $\gamma_{m+2}(G)=\langle\sigma^{2^{m+1}}, \tau^{2^{m+1}}\rangle=\langle1\rangle$ and  $\gamma_{m+1}(G)=\langle\sigma^{2^{m}}, \tau^{2^{m}}\rangle\neq\langle1\rangle$. Since $\mid G\mid=2^{n+m+3}$, we have $$c(G)=m+1\text{ and }cc(G)=n+m+3-m-1=3.$$
 If   $\left(\frac{p_1}{p_2}\right)=1$ and  $N(\varepsilon_{p_1p_2})=-1$, then  $\gamma_{n+3}(G)=\langle\sigma^{2^{n+2}}, \tau^{2^{n+2}}\rangle=\langle1\rangle$ and $\gamma_{n+2}(G)=\langle\sigma^{2^{n+1}}, \tau^{2^{n+1}}\rangle\neq\langle1\rangle$. As $\mid G\mid=2^{n+m+3}$, so $$c(G)=n+2\text{ and } cc(G) = n+m+3-n-2=3.$$
 If   $\left(\frac{p_1}{p_2}\right)=1$,  $N(\varepsilon_{p_1p_2})=1$ and  $\left(\frac{p_1}{p_2}\right)_4\left(\frac{p_2}{p_1}\right)_4=-1$, then $n=1$,  $m\geq3$,   $\gamma_{m+1}(G)=\langle\sigma^{2^{m}}, \tau^{2^{m}}\rangle=\langle1\rangle$ and  $\gamma_{m}(G)=\langle\sigma^{2^{m-1}}, \tau^{2^{m-1}}\rangle\neq\langle1\rangle$. As $\mid G\mid=2^{n+m+3}$, so $$c(G)=m\text{ and } cc(G) = n+m+3-m=4.$$
 If    $\left(\frac{p_1}{p_2}\right)=1$,  $N(\varepsilon_{p_1p_2})=1$ and  $\left(\frac{p_1}{p_2}\right)_4=\left(\frac{p_2}{p_1}\right)_4=1$, then  $m=2$,  $n\geq2$,   $\gamma_{n+3}(G)=\langle\sigma^{2^{n+2}}, \tau^{2^{n+2}}\rangle=\langle1\rangle$  and  $\gamma_{n+2}(G)=\langle\sigma^{2^{n+1}}, \tau^{2^{n+1}}\rangle\neq\langle1\rangle$. As $\mid G\mid=2^{n+m+3}$, so $$c(G)=n+2\text{ and } cc(G) = n+m+3-n-2=3.$$
 \subsection{Proof of  Theorems \ref{10:3} and \ref{10:4}}
  For this we need the following result.
\begin{propo}\label{24}
Let $G=\langle\sigma, \tau, \rho\rangle$ be the group defined above.
\begin{enumerate}[\rm\indent(1)]
  \item $\rho^{-1}\tau\rho=\left\{\begin{array}{ll}
                              \tau^{-1} &\text{  if  } N(\varepsilon_{p_1p_2})=-1,\\
                              \tau^{2^{n+1}-1} &\text{  if  } N(\varepsilon_{p_1p_2})=1.
                              \end{array}
                              \right.$
  \item $\rho^{-1}\sigma\rho=\sigma^{-1}.$
  \item $[\rho^2,\sigma]=[\rho^2,\tau]=1$.
  \item $(\sigma\tau\rho)^2=(\tau\rho)^2=\left\{\begin{array}{ll}
                              \rho^{2} &\text{  if  } N(\varepsilon_{p_1p_2})=-1,\\
                              \rho^2\tau^{2^{n+1}} &\text{  if  } N(\varepsilon_{p_1p_2})=1.
                              \end{array}
                              \right.$
  \item $(\sigma\rho)^2=\rho^{2}.$
  \item For all $r\in\NN^*$, we have $[\rho, \tau^{2^r}]=\tau^{2^{r+1}}$ and
                              $[\rho,\sigma^{2^r}]=
                              \sigma^{2^{r+1}}.$
\end{enumerate}
\end{propo}
\begin{proof}
$(6)$ Since  $[\rho, \tau]=\left\{\begin{array}{ll}
                              \tau^{2} &\text{  if  } N(\varepsilon_{p_1p_2})=-1,\\
                               \tau^{2-2^{n+1}}&\text{  if  } N(\varepsilon_{p_1p_2})=1,
                              \end{array}
                              \right.$\\
so  $[\rho, \tau^2]=\left\{\begin{array}{ll}
                              \tau^{4} &\text{  if  } N(\varepsilon_{p_1p_2})=-1,\\
                               \tau^{2-2^{n+1}}\tau^{2-2^{n+1}}=\tau^{4}&\text{  if  } N(\varepsilon_{p_1p_2})=1.
                              \end{array}
                              \right.$\\
By induction, we show that for all $r\in\NN^*$,   $[\rho, \tau^{2^r}]=\tau^{2^{r+1}}$. Similarly, we get that $[\rho,\sigma^{2^r}]=
                              \sigma^{2^{r+1}}.$
\end{proof}
 The proof of  Theorems \ref{10:3} and \ref{10:4} consists of 3 parts. In the first part, we will compute $N_{\KK_j/\kk}(\mathbf{C}l_2(\KK_j))$, for all $1\leq j\leq7$.  In the second one,  we will determine the capitulation kernels $\kappa_{\KK_j}$ and the types of $\mathbf{C}l_2(\KK_j)$ and in the third one, we will determine the capitulation kernels $\kappa_{\LL_j}$ and the types of $\mathbf{C}l_2(\LL_j)$.

\subsubsection{ \bf  Norm class groups.} Let us compute $N_j=N_{\KK_j/\kk}(\mathbf{C}l_2(\KK_j))$,  the results are
summarized in the following table.
\footnotesize
 \begin{longtable}{| c | c | c | c |}
 \caption{Norm class groups}\label{26}\\
\hline
 $\KK_j$  & Conditions & $N_j$ for  $\left(\frac{p_1}{p_2}\right)=1$ & $N_j$ for  $\left(\frac{p_1}{p_2}\right)=-1$\\
\hline
\endfirsthead
\hline
 $\KK_j$ & Conditions & $N_j$ for  $\left(\frac{p_1}{p_2}\right)=1$ & $N_j$ for  $\left(\frac{p_1}{p_2}\right)=-1$\\
\hline
\endhead
 $\KK_1$ & &$\langle[\mathcal{H}_3], [\mathcal{H}_1\mathcal{H}_2]\rangle$ &  $\langle[\mathcal{H}_1], [\mathcal{H}_2]\rangle$\\ \hline
 $\KK_2$ & & $\langle[\mathcal{H}_1], [\mathcal{H}_2]\rangle$ & $\langle[\mathcal{H}_1\mathcal{H}_2], [\mathcal{H}_3]\rangle$\\ \hline
$\KK_3$ & & $\langle[\mathcal{H}_1\hhh], [\mathcal{H}_2\mathcal{H}_3]\rangle$ &  $\langle[\mathcal{H}_1\hhh], [\mathcal{H}_2\mathcal{H}_3]\rangle$\\ \hline
 &  $\pi=1$ & $\langle[\mathcal{H}_1], [\mathcal{H}_3]\rangle$  &  $\langle[\mathcal{H}_2], [\mathcal{H}_1\mathcal{H}_3]\rangle$  \\[-1ex]
 \raisebox{2ex}{$\KK_4$}
    & $\pi=-1$&  $\langle[\mathcal{H}_2], [\mathcal{H}_1\mathcal{H}_3]\rangle$  &  $\langle[\mathcal{H}_1], [\mathcal{H}_3]\rangle$  \\[1ex]\hline
 & $\pi=1$ &  $\langle[\mathcal{H}_1], [\mathcal{H}_2\mathcal{H}_3]\rangle$  &  $\langle[\mathcal{H}_1], [\mathcal{H}_2\hhh]\rangle$  \\[-1ex]
\raisebox{1.7ex}{$\KK_5$}
 & $\pi=-1$ &  $\langle[\mathcal{H}_2], [\mathcal{H}_3]\rangle$  &  $\langle[\mathcal{H}_2], [\mathcal{H}_3]\rangle$  \\[1ex]\hline
 & $\pi=1$ &  $\langle[\mathcal{H}_2], [\mathcal{H}_3]\rangle$ &  $\langle[\mathcal{H}_2], [\mathcal{H}_3]\rangle$ \\[-1ex]
\raisebox{2ex}{$\KK_6$}
   &  $\pi=-1$    &  $\langle[\mathcal{H}_1], [\mathcal{H}_2\hhh]\rangle$  &  $\langle[\mathcal{H}_1], [\mathcal{H}_2\mathcal{H}_3]\rangle$ \\[1ex]\hline
 &  $\pi=1$ &  $\langle[\mathcal{H}_2], [\mathcal{H}_1\hhh]\rangle$  &  $\langle[\mathcal{H}_1], [\mathcal{H}_3]\rangle$  \\[-1ex]
\raisebox{2ex}{$\KK_7$}
  &    $\pi=-1$   &  $\langle[\mathcal{H}_1], [\mathcal{H}_3]\rangle$  &  $\langle[\mathcal{H}_2], [\mathcal{H}_1\hhh]\rangle$  \\[1ex]\hline
\end{longtable}
\normalsize
To check the table entries we use Lemma \ref{10:17} and the following
results which are easy to prove.
\begin{lem}\label{10:20}
Let $p_1\equiv p_2\equiv1\pmod4$ be different primes. Put $p_1=\pi_1\pi_2$ and  $p_2=\pi_3\pi_4$, where $\pi_j\in\ZZ[i]$,  then
\begin{enumerate}[\rm\indent(i)]
  \item $\left(\frac{\pi_1}{\pi_2}\right)=\left(\frac{\pi_3}{\pi_4}\right)=
  \left\{\begin{array}{ll}
   1 &\text{  if  } p_1\equiv p_2\equiv1\pmod8,\\
   -1  &\text{  if  } p_1\equiv p_2\equiv5\pmod8
   \end{array}
 \right.$
  \item If  $\left(\frac{p_1}{p_2}\right)=1$, then $\left(\frac{\pi_1}{\pi_3}\right)=\left(\frac{\pi_2}{\pi_3}\right)=
\left(\frac{\pi_1}{\pi_4}\right)=\left(\frac{\pi_2}{\pi_4}\right)$.
  \item If  $\left(\frac{p_1}{p_2}\right)=-1$, then $\left(\frac{\pi_1}{\pi_3}\right)=\left(\frac{\pi_2}{\pi_4}\right)=-\left(\frac{\pi_2}{\pi_3}\right)=-\left(\frac{\pi_1}{\pi_4}\right)$.
  \item If  $\left(\frac{2}{p_1}\right)=1$, then $\left(\frac{1+i}{\pi_1}\right)=\left(\frac{1+i}{\pi_2}\right)$.
  \item If  $\left(\frac{2}{p_1}\right)=-1$, then $\left(\frac{1+i}{\pi_1}\right)=-\left(\frac{1+i}{\pi_2}\right)$.
\end{enumerate}
\end{lem}
Compute $N_j$ in a few cases keeping in mind that   $\mathcal{H}_1$, $\mathcal{H}_2$ and   $\mathcal{H}_3$ are unramified prime ideals in  $\KK_j/\kk$.\\
 $\bullet$ Take as a first example:  $\KK_1=\kk(\sqrt{p_1})=\kk(\sqrt{p_2q})= \QQ(\sqrt p_1, \sqrt{p_2q}, i)$. As  $N_1=
  \{[\mathcal{H}]\in\mathbf{C}l_2(\kk)/\displaystyle\left(\frac{\alpha}{\mathcal{H}}\right)=1\}$, so for all $j\in\{1, 2\}$ we get
 \begin{align*}
 \left(\frac{\kk(\sqrt{ p_2q})/\kk}{\mathcal{H}_j}\right)&=\left(\frac{\kk(\sqrt{ p_2q})/\kk}{\mathcal{H}_j}\right)(\sqrt{ p_2q})(\sqrt{ p_2q})^{-1}\\
 &=\left(\frac{p_2q}{\mathcal{H}_j}\right)\\
  &=\left(\frac{p_2q}{p_1}\right)\\
  &=\left(\frac{q}{p_1}\right)\left(\frac{p_1}{p_2}\right)\\
  &=-\left(\frac{p_1}{p_2}\right).
\end{align*}
Similarly, we have:
 \begin{align*}
 \left(\frac{\kk(\sqrt{ p_1})/\kk}{\mathcal{H}_3}\right)&=\left(\frac{\kk(\sqrt{ p_1})/\kk}{\mathcal{H}_3}\right)(\sqrt{ p_1})(\sqrt{ p_1})^{-1}\\
 &=\left(\frac{p_1}{\mathcal{H}_3}\right)\\
  &=\left(\frac{p_1}{p_2}\right).  \text{ Thus }
\end{align*}
  - If  $\left(\frac{p_1}{p_2}\right)=-1$, then $[\mathcal{H}_j]\in N_1$ and
  $N_1=\langle[\mathcal{H}_1],  \mathcal{H}_2]\rangle.$\\
 -  If  $\left(\frac{p_1}{p_2}\right)=1$, then $[\mathcal{H}_j]\not\in N_1$,
  $[\mathcal{H}_1\mathcal{H}_2]\in N_1$ and $[\mathcal{H}_3]\in N_1$. Hence\\
  $N_1=\langle[\mathcal{H}_1\mathcal{H}_2],  [\mathcal{H}_3]\rangle.$\\
 $\bullet$ Take as a second example:  $\KK_4=\kk(\sqrt{\pi_1\pi_3})=\kk(\sqrt{q\pi_2\pi_4})$.\\
 1\up{st} case: Assume first that   $\left(\frac{p_1}{p_2}\right)=-1$,  hence  Lemmas \ref{10:17} and \ref{10:20} imply that:\\
 $\left\{
 \begin{array}{ll}
 \left(\frac{\pi_1\pi_3}{\mathcal{H}_2}\right)=\left(\frac{\pi_1\pi_3}{\pi_2}\right)=\left(\frac{\pi_1}{\pi_2}\right)\left(\frac{\pi_3}{\pi_2}\right)=
 -\left(\frac{\pi_2}{\pi_3}\right)=\left(\frac{\pi_1}{\pi_3}\right),\\
 \left(\frac{\pi_2\pi_4q}{\mathcal{H}_1}\right)=\left(\frac{\pi_2\pi_4q}{\pi_1}\right)=\left(\frac{q}{p_1}\right)\left(\frac{\pi_1}{\pi_2}\right)
 \left(\frac{\pi_4}{\pi_1}\right)=\left(\frac{\pi_1}{\pi_4}\right)=-\left(\frac{\pi_1}{\pi_3}\right),\\
 \left(\frac{\pi_2\pi_4q}{\mathcal{H}_3}\right)=\left(\frac{\pi_2\pi_4q}{\pi_3}\right)=\left(\frac{q}{p_2}\right)\left(\frac{\pi_2}{\pi_3}\right)
 \left(\frac{\pi_4}{\pi_3}\right)=\left(\frac{\pi_2}{\pi_3}\right)=-\left(\frac{\pi_1}{\pi_3}\right).
 \end{array}\right.$\\ Thus \\
-  If  $\left(\frac{\pi_1}{\pi_3}\right)=-1$, then $\mathcal{H}_1\in N_4$ and $\mathcal{H}_3\in N_4$. Hence
  $ N_4=\langle[\mathcal{H}_1], [\mathcal{H}_3]\rangle. $\\
 - If  $\left(\frac{\pi_1}{\pi_3}\right)=1$, then $\mathcal{H}_2\in N_4$ and  $\h\mathcal{H}_3\in N_4$. Hence
  $N_4=\langle[\mathcal{H}_2], [\h\mathcal{H}_3]\rangle. $\\
 2\up{nd} case: Assume   $\left(\frac{p_1}{p_2}\right)=1$,  then   Lemmas \ref{10:17} and  \ref{10:20} imply that:\\
 $\left\{
 \begin{array}{ll}
 \left(\frac{\pi_1\pi_3}{\mathcal{H}_2}\right)=\left(\frac{\pi_1\pi_3}{\pi_2}\right)=\left(\frac{\pi_1}{\pi_2}\right)\left(\frac{\pi_3}{\pi_2}\right)=
 -\left(\frac{\pi_2}{\pi_3}\right)=-\left(\frac{\pi_1}{\pi_3}\right),\\
 \left(\frac{\pi_2\pi_4q}{\mathcal{H}_1}\right)=\left(\frac{\pi_2\pi_4q}{\pi_1}\right)=\left(\frac{q}{p_1}\right)\left(\frac{\pi_1}{\pi_2}\right)
 \left(\frac{\pi_4}{\pi_1}\right)=\left(\frac{\pi_1}{\pi_4}\right)=\left(\frac{\pi_1}{\pi_3}\right),\\
 \left(\frac{\pi_2\pi_4q}{\mathcal{H}_3}\right)=\left(\frac{\pi_2\pi_4q}{\pi_3}\right)=\left(\frac{q}{p_2}\right)\left(\frac{\pi_2}{\pi_3}\right)
 \left(\frac{\pi_4}{\pi_3}\right)=\left(\frac{\pi_2}{\pi_3}\right)=\left(\frac{\pi_1}{\pi_3}\right).
 \end{array}\right.$\\ Thus \\
-  If  $\left(\frac{\pi_1}{\pi_3}\right)=1$, then $\mathcal{H}_1\in N_4$ and $\mathcal{H}_3\in N_4$.
Hence
  $ N_4=\langle[\mathcal{H}_1], [\mathcal{H}_3]\rangle. $\\
 - If  $\left(\frac{\pi_1}{\pi_3}\right)=-1$, then $\mathcal{H}_2\in N_4$ and $\h\mathcal{H}_3\in N_4$. Hence
  $N_4=\langle[\mathcal{H}_2], [\h\mathcal{H}_3]\rangle. $\\
Proceeding similarly, we check the other table inputs.
\subsubsection{\bf Capitulation kernels   $\kappa_{\KK_j/\kk}$ and $\mathbf{C}l_2(\KK_j)$\label{17}.}
  Let us compute the  Galois groups $G_j=\mathrm{Gal}(\L/\KK_j)$,  the capitulation kernels $\kappa_{\KK_j}$,  $\kappa_{\KK_j}\cap N_j $ and the types of $\mathbf{C}l_2(\KK_j)$.  The results are summarized in the following Tables \ref{10:23} and \ref{10:21}.   Put $\beta=\left(\frac{1+i}{\pi_1}\right)\left(\frac{1+i}{\pi_3}\right)$,  $\pi=\left(\frac{\pi_1}{\pi_3}\right)$, $N=N(\varepsilon_{p_1p_2})$ and $\mathrm{I}=\left(\frac{p_1p_2}{2}\right)_4\left(\frac{2p_1}{p_2}\right)_4\left(\frac{2p_2}{p_1}\right)_4$. Note that, in the  Table \ref{10:23} and for the column $G_j$,   the left hand side (if it exists) refers to the case  $\beta=1$, while the right one refers to the case  $\beta=-1$. Whereas, in the  Table \ref{10:21} and for the same column,   the left hand side (if it exists) refers to the case  $\mathrm{I}=1$, while the right one refers to the case  $\mathrm{I}=-1$.
\scriptsize
 \begin{longtable}{|c c | c | c | c | c |}
\caption{\large{ $\kappa_{\KK_j/\kk}$ for the case $\left(\frac{p_1}{p_2}\right)=1$}.}\label{10:23}\\
\hline
 $\KK_j$  && $G_j$ &  $\kappa_{\KK_j/\kk}$ & $\kappa_{\KK_j/\kk}\cap N_j $ & $\mathbf{C}l_2(\KK_j)$\\
\hline
\endfirsthead
\hline
 $\KK_j$ & &$G_j$ &  $\kappa_{\KK_j/\kk}$ & $\kappa_{\KK_j/\kk}\cap N_j$ & $\mathbf{C}l_2(\KK_j)$ \\
\hline
\endhead
 $\KK_1$ & & $\langle\sigma, \tau\rho, \tau^2\rangle$  & $\langle[\mathcal{H}_1], [\mathcal{H}_2]\rangle$& $\langle[\mathcal{H}_1\mathcal{H}_2]\rangle$ & $(2, 2, 2)$\\ \hline
  $\KK_2$ & & $\langle\sigma, \rho, \tau^2\rangle$  &   $\langle[\mathcal{H}_1\mathcal{H}_2], [\mathcal{H}_3]\rangle$ & $\langle[\mathcal{H}_1\mathcal{H}_2]\rangle$ & $(2, 2, 2)$\\ \hline
 & $N=1$ & &  $\langle[\mathcal{H}_1\mathcal{H}_2]\rangle$ &   &  \\[-1ex]
  \raisebox{2ex}{$\KK_3$}
   &$N=-1$ & \raisebox{2ex}{$\langle\tau, \sigma\rangle$}
  &  $\langle[\mathcal{H}_1\mathcal{H}_3]\rangle$ or $\langle[\mathcal{H}_2\mathcal{H}_3]\rangle$ & \raisebox{2ex}{$\kappa_{\KK_3/\kk}$}  & \raisebox{2ex}{$(2^{m}, 2^{n+2})$} \\ [1ex]\hline
& $\pi=1$ &  $\langle\tau, \rho\rangle$  $\langle\sigma\tau, \rho\rangle$   &  & $N_4=\kappa_{\KK_4/\kk}$ & $(2, 4)$\\[-1ex]
 \raisebox{2ex}{$\KK_4$}
  &  $\pi=-1$   &  $\langle\tau, \sigma\rho, \sigma^2\rangle$  $\langle\sigma\tau, \sigma\rho, \sigma^2\rangle$ &  \raisebox{2ex}{$\langle[\mathcal{H}_1], [\mathcal{H}_3]\rangle$} & $\langle[\mathcal{H}_1\mathcal{H}_3]\rangle$ & $(2, 2, 2)$ \\[1ex]\hline
\raisebox{-2ex}{$\KK_5$} & $\pi=1$ &  $\langle\sigma\tau, \rho\rangle$  $\langle\tau, \rho\rangle$  &  & $N_5=\kappa_{\KK_5/\kk}$ &  $(2, 4)$\\[-1ex]
 & $\pi=-1$  & $\langle\sigma\rho, \sigma\tau, \sigma^2\rangle$  $\langle\sigma\rho, \tau, \sigma^2\rangle$  & \raisebox{2ex}{$\langle[\mathcal{H}_1], [\hh\mathcal{H}_3]\rangle$} & $\langle[\mathcal{H}_2\hhh]\rangle$  &  $(2, 2, 2)$ \\[1ex]\hline
 \raisebox{-2ex}{$\KK_6$}&$\pi=1$ &  $\langle\sigma\tau, \sigma\rho\rangle$  $\langle\tau, \sigma\rho\rangle$  &  & $N_6=\kappa_{\KK_6/\kk}$  & $(2, 4)$\\[-1ex]
  &   $\pi=-1$    &  $\langle \rho, \sigma\tau, \sigma^2\rangle$   $\langle\rho, \tau, \sigma^2\rangle$  & \raisebox{2ex}{$\langle[\mathcal{H}_2], [\mathcal{H}_3]\rangle$} & $\langle[\mathcal{H}_2\hhh]\rangle$ & $(2, 2, 2)$ \\[1ex]\hline
& $\pi=1$ &  $\langle\tau, \sigma\rho\rangle$   $\langle\sigma\tau, \sigma\rho\rangle$ & & $N_7=\kappa_{\KK_7/\kk}$ &$(2, 4)$\\[-1ex]
\raisebox{2ex}{$\KK_7$ }
    & $\pi=-1$    &  $\langle \rho, \tau, \sigma^2\rangle$   $\langle\rho, \sigma\tau, \sigma^2\rangle$  & \raisebox{2ex}{$\langle[\mathcal{H}_2], [\h\mathcal{H}_3]\rangle$} & $\langle[\mathcal{H}_1\mathcal{H}_3]\rangle$ &  $(2, 2, 2)$  \\[1ex]\hline
\end{longtable}
\normalsize
\scriptsize
 \begin{longtable}{| c c | c | c | c | c |}
\caption{\large{ $\kappa_{\KK_j/\kk}$ for the case $\left(\frac{p_1}{p_2}\right)=-1$.}}\label{10:21}\\
\hline
 $\KK_j$ & & $G_j$ &  $\kappa_{\KK_j/\kk}$ & $\kappa_{\KK_j/\kk}\cap N_j$ & $\mathbf{C}l_2(\KK_j)$\\
\hline
\endfirsthead
\hline
 $\KK_j$ & & $G_j$ &  $\kappa_{\KK_j/\kk}$ & $\kappa_{\KK_j/\kk}\cap N_j$ & $\mathbf{C}l_2(\KK_j)$ \\
\hline
\endhead
 $\KK_1$ & & $\langle\sigma, \rho\rangle$  & $\langle[\mathcal{H}_1], [\mathcal{H}_2]\rangle$ & $N_1=\kappa_{\KK_1/\kk}$ & $(2, 4)$\\ \hline
  $\KK_2$ & & $\langle\sigma, \tau\rho\rangle$  &   $\langle[\mathcal{H}_1\mathcal{H}_2], [\mathcal{H}_3]\rangle$ & $N_2=\kappa_{\KK_2/\kk}$  & $(2, 4)$\\ \hline
 $\KK_3$ & &  $\langle\sigma, \tau\rangle$ & $\langle[\mathcal{H}_1\mathcal{H}_3]\rangle$ or $\langle[\mathcal{H}_2\mathcal{H}_3]\rangle$ & $\kappa_{\KK_3/\kk}$ & $(4, 2^{m+1})$ \\ \hline
& $\pi=1$ &  $\langle\tau, \sigma\rho, \sigma^2\rangle$ $\langle\tau\rho, \sigma\tau, \tau^2\rangle$  &   & \qquad $\langle[\mathcal{H}_1\hhh]\rangle$ & $(2, 2, 2)$\\[-1ex]
 \raisebox{2ex}{$\KK_4$}
 & $\pi=-1$     &  $\langle\sigma\tau, \rho\rangle$  $\langle\tau, \rho\rangle$   & \raisebox{2ex}{$\langle[\mathcal{H}_1], [\mathcal{H}_3]\rangle$}  & $N_4=\kappa_{\KK_4/\kk}$  &  $(2, 4)$ \\[1ex]\hline
 & $\pi=1$ &   $\langle\sigma\tau, \rho\rangle$ $\langle\tau, \rho\rangle$  &  & $N_5=\kappa_{\KK_5/\kk}$  & $(2, 4)$\\[-1ex]
\raisebox{2ex}{$\KK_5$}
& $\pi=-1$  & $\langle\tau, \sigma\rho, \sigma^2\rangle$  $\langle\sigma\tau, \sigma\rho, \tau^2\rangle$  &  \raisebox{2ex}{$\langle[\mathcal{H}_1], [\mathcal{H}_2\mathcal{H}_3]\rangle$} & $\langle[\mathcal{H}_2\mathcal{H}_3]\rangle$  & $(2, 2, 2)$ \\[1ex]\hline
& $\pi=1$ &  $\langle\sigma\tau, \sigma\rho\rangle$  $\langle\sigma\rho, \tau\rangle$  &   & $N_6=\kappa_{\KK_6/\kk}$ &$(2, 4)$\\[-1ex]
\raisebox{2ex}{$\KK_6$}
  & $\pi=-1$       &  $\langle\tau, \rho, \sigma^2\rangle$  $\langle\rho, \sigma\tau, \tau^2\rangle$     & \raisebox{2ex}{$\langle[\mathcal{H}_2], [\mathcal{H}_3]\rangle$}  & $\langle\mathcal{H}_2\mathcal{H}_3]\rangle$   &  $(2, 2, 2)$ \\[1ex]\hline
 & $\pi=1$ &  $\langle\rho, \tau, \sigma^2\rangle$  $\langle\rho, \sigma\tau, \tau^2 \rangle$  &  & $\langle[\h\mathcal{H}_3]\rangle$& $(2, 2, 2)$\\[-1ex]
\raisebox{2ex}{$\KK_7$}
  &   $\pi=-1$    &  $\langle\sigma\rho, \sigma\tau\rangle$  $\langle\tau, \sigma\rho\rangle$    & \raisebox{2ex}{$\langle[\mathcal{H}_2], [\mathcal{H}_1\mathcal{H}_3]\rangle$} & $N_7=\kappa_{\KK_7/\kk}$ &  $(2, 4)$  \\[1ex]\hline
\end{longtable}
\normalsize
To check the tables inputs,   we use the following remarks and Lemma \ref{10:19}.
\begin{rema} According to Artin symbol properties we get:\\
$\bullet$ $\sigma=\displaystyle\left(\frac{L/\KK_3}{\mathfrak{P}}\right)=
\displaystyle\left(\frac{L/\kk}{N_{\KK_3/\kk}(\mathfrak{P})}\right)=
\displaystyle\left(\frac{L/\kk}{\mathcal{H}_1\mathcal{H}_2}\right)$, since $N_{\KK_3/\kk}(\mathfrak{P})\sim \mathcal{H}_1\mathcal{H}_2$.\\
$\bullet$ $\tau=\displaystyle\left(\frac{L/\KK_3}{\mathfrak{A}}\right)=
\displaystyle\left(\frac{L/\kk}{N_{\KK_3/\kk}(\mathfrak{A})}\right)=
\displaystyle\left(\frac{L/\kk}{\mathcal{H}_1\mathcal{H}_3}\right)$ or $\displaystyle\left(\frac{L/\kk}{\mathcal{H}_2\mathcal{H}_3}\right)$, since $N_{\KK_3/\kk}(\mathfrak{A})\sim \mathcal{H}_1\mathcal{H}_3$ or $\mathcal{H}_2\mathcal{H}_3$.\\
$\bullet$ $\rho=\displaystyle\left(\frac{L/\kk}{\mathcal{H}_1}\right)$.
\end{rema}
From Lemmas \ref{10:9} and \ref{10:12}, we deduce that:
\begin{rema}
 $(1) $Assume that $\left(\frac{p_1}{p_2}\right)=1$, so
\begin{enumerate}[\rm\indent(i)]
  \item If $\left(\frac{\pi_1}{\pi_3}\right)=-1$, then $N(\varepsilon_{p_1p_2})=1$,  $n=1$ and $m\geq3$. Thus $\rho^2=\tau^4\sigma^{2^{m-1}}$ or $\rho^2=\sigma^{2^{m-1}}$.
  \item If $\left(\frac{\pi_1}{\pi_3}\right)=1$, then
  \begin{enumerate}[\rm\indent(a)]
    \item If $N(\varepsilon_{p_1p_2})=1$, then $m=2$ and $n\geq2$. Thus $\rho^2=\sigma^2\tau^{2^{n+1}}$ or $\rho^2=\sigma^{2}$.
    \item If $N(\varepsilon_{p_1p_2})=-1$, then $m=2$ and $n\geq2$. Thus $\rho^2=\sigma^2\tau^{2^{n}}$ and $\sigma^4=\tau^{2^{n+1}}$.
  \end{enumerate}
\end{enumerate}
$(2)$  Assume that $\left(\frac{p_1}{p_2}\right)=-1$, so $n=1$,   $m\geq2$,  $\rho^2=\sigma^{2^{m-1}}\tau^{2}$ and $\sigma^4=\tau^{4}$.
\end{rema}

Recall that the Artin map $\phi$  induces the following commutative diagram:
\begin{figure}[H]
\xymatrix{
 &&&& \mathbf{C}l_2(\kk)\ar[d]_{J_{\KK_j/\kk}} \ar[r]^- \phi & G/G' \ar[d]^{{\rm V}_{G/G_j}}\\
 &&&& \mathbf{C}l_2(\KK_j) \ar[r]^-\phi & G_j/G_j' }
\end{figure}
\hspace{-0.4cm}the rows are isomorphisms and ${\rm V}_{G/G_j}: G/G' \longrightarrow G_j/G_j'$ is the group transfer map (Verlagerung) which has the following simple characterization when $G_j$ is of index $2$ in $G$. Let $G=G_j\cup zG_j$,  then
$${\rm V}_{G/G_j}(gG')=\left\{
        \begin{array}{ll}
          gz^{-1}gz.G_j'=g^2[g, z].G_j'  & \hbox{ if $g\in G_j$},\\
          g^2G_j'  & \hbox{ if $g \notin G_j$.}
        \end{array}
      \right.$$
Thus $\kappa_{\KK_j/\kk}=\ker J_{\KK_j/\kk}$ is determined by $\ker{\rm V}_{G/ G_j}$.\\
Let us show the table inputs for a few examples.

 (a) For the extension   $\KK_1$,  Table  \ref{26} yields that:\\
$N_1=\left\{ \begin{array}{ll}
 \langle[\mathcal{H}_3], [\mathcal{H}_1\mathcal{H}_2]\rangle & \text{ if }\left(\frac{p_1}{p_2}\right)=1,\\
  \langle[\mathcal{H}_1], [\mathcal{H}_2]\rangle=\langle[\mathcal{H}_1\mathcal{H}_2], [\mathcal{H}_1]\rangle  &\text{  if  }\left(\frac{p_1}{p_2}\right)=-1.
 \end{array} \right.$\\
 Hence  $G_1=\mathrm{Gal}(L/\KK_1)=\left\{ \begin{array}{ll}
 \langle\sigma, \tau\rho, G'\rangle=\langle\sigma, \tau\rho, \tau^2\rangle  &\text{  if  }\left(\frac{p_1}{p_2}\right)=1,\\
 \langle\sigma, \rho, G'\rangle=\langle\sigma, \rho, \tau^2\rangle=\langle\sigma, \rho\rangle  &\text{  if  }\left(\frac{p_1}{p_2}\right)=-1.
 \end{array} \right.$\\
Thus $G/G_1=\langle\tau\rangle=\{1, \tau G_1\}$. Moreover, Proposition \ref{24} implies that\\
 $[\tau\rho, \sigma]=[\rho, \sigma]=\sigma^2$ and
$[\tau\rho, \tau^2]=[\rho, \tau^2]=\tau^4$;\\
so  $G_1'=\left\{ \begin{array}{ll}
\langle\tau^4, \sigma^2\rangle  &\text{  if  }\left(\frac{p_1}{p_2}\right)=1,\\
\langle \sigma^2\rangle  &\text{  if  }\left(\frac{p_1}{p_2}\right)=-1;
\end{array} \right.$
this in turn implies that\\ $\mathbf{C}l_2(\KK_1)=G_1/G_1'\simeq\left\{ \begin{array}{ll}
 (2, 2, 2) &\text{  if  }\left(\frac{p_1}{p_2}\right)=1,\\
 (2, 4) &\text{  if  }\left(\frac{p_1}{p_2}\right)=-1,
 \end{array} \right.$ since  $(\tau\rho)^2$, $\rho^2\in G_1 '$.\\
Compute now the kernel of  $V_{G\rightarrow G_1}$.
\begin{enumerate}[\indent $\ast$]
\item $V_{G\rightarrow G_1}(\sigma G')=\sigma^2G_1'=G_1'$.
\item $V_{G\rightarrow G_1}(\tau G')=\tau^2G_1'\neq G_1'$.
 \item $V_{G\rightarrow G_1}(\rho G')=\rho^2G_1'= G_1'$.
 \end{enumerate}
 Hence $\ker V_{G\rightarrow G_1}=\langle\sigma G', \rho G'\rangle,$ thus  $\kappa_{\KK_1/\kk}=\langle[\mathcal{H}_1\mathcal{H}_2], [\mathcal{H}_1]\rangle=\langle[\mathcal{H}_1], [\mathcal{H}_2]\rangle.$

 (b) Take another example  $\KK_4=\kk(\sqrt{\pi_1\pi_3})$ and assume that $\left(\frac{p_1}{p_2}\right)=-1$, then  $n=1$, $m\geq2$ and  $N(\varepsilon_{p_1p_2})=-1$. We have two cases to discuss:\\
 \indent  1\up{st} case:    $\left(\frac{\pi_1}{\pi_3}\right)=-1$.  Table  \ref{26} yields that
$N_4=\langle[\mathcal{H}_1], [\mathcal{H}_3]\rangle$,
 hence  $G_4=\mathrm{Gal}(L/\KK_4)=\langle\tau ,\rho, G'\rangle=\langle\tau, \rho\rangle $ or
 $\langle\sigma\tau, \rho, G'\rangle=\langle\sigma\tau, \rho\rangle$ according as $\tau=\displaystyle\left(\frac{L/\kk}{\mathcal{H}_1\mathcal{H}_3}\right)$ or $\displaystyle\left(\frac{L/\kk}{\mathcal{H}_2\mathcal{H}_3}\right)$.
Thus $G/G_4=\langle\sigma\rangle=\{1, \sigma G_4\}$. Moreover, Proposition \ref{24} implies that
 $[\rho, \sigma\tau]=(\sigma\tau)^2$ and $[\rho, \tau]=\tau^2$.\\
So  $G_4'=\langle\tau^2\rangle$ or
$\langle (\sigma\tau)^2\rangle.$
Therefore $\mathbf{C}l_2(\KK_4)\simeq G_4/G_4'\simeq (2, 4)$.\\
Compute   $\ker V_{G\rightarrow G_4}$, according as $\tau=\displaystyle\left(\frac{L/\kk}{\mathcal{H}_1\mathcal{H}_3}\right)$ or $\displaystyle\left(\frac{L/\kk}{\mathcal{H}_2\mathcal{H}_3}\right)$ we get\\
$\left\{ \begin{array}{ll}
\ast V_{G\rightarrow G_4}(\sigma G')=\sigma^2G_4'\neq G_4'.\\
\ast V_{G\rightarrow G_4}(\tau G')=\tau^2[\sigma, \tau]G_4'= G_4'.\\
\ast V_{G\rightarrow G_4}(\rho G')=\rho^2G_4'= G_4'.
\end{array} \right.$
 or
$\left\{ \begin{array}{ll}
\ast V_{G\rightarrow G_4}(\sigma G')=\sigma^2G_4'\neq G_4'.\\
\ast V_{G\rightarrow G_4}(\tau G')=\tau^2G_4'\neq G_4'.\\
\ast V_{G\rightarrow G_4}(\rho G')=\rho^2G_4'= G_4'.\\
\ast V_{G\rightarrow G_4}(\sigma\tau G')=(\sigma\tau)^2G_4'= G_4'.
\end{array} \right.$\\
Hence  $\ker V_{G\rightarrow G_4}=\langle\tau G', \rho G'\rangle$ or $\langle\sigma\tau G', \rho G'\rangle$, thus
 $$\kappa_{\KK_4/\kk}=\langle[\mathcal{H}_1\mathcal{H}_3], [\mathcal{H}_1]\rangle=\langle[\mathcal{H}_1], [\mathcal{H}_3]\rangle.$$

 \indent  2\up{nd} case:    $\left(\frac{\pi_1}{\pi_3}\right)=1$.  Table  \ref{26} yields that
$N_4=\langle[\mathcal{H}_2], [\mathcal{H}_1\mathcal{H}_3]\rangle$,
 hence  $G_4=\mathrm{Gal}(L/\KK_4)=\langle\tau ,\sigma\rho, G'\rangle=\langle\tau, \sigma\rho, \sigma^2\rangle$ or
 $\langle\sigma\tau ,\tau\rho, G'\rangle=\langle\sigma\tau ,\tau\rho,  \sigma^2\rangle$ according as $\tau=\displaystyle\left(\frac{L/\kk}{\mathcal{H}_1\mathcal{H}_3}\right)$ or $\displaystyle\left(\frac{L/\kk}{\mathcal{H}_2\mathcal{H}_3}\right)$.
Thus $G/G_4=\langle\sigma\rangle=\{1, \sigma G_4\}$. Moreover, Proposition \ref{24} implies that
 $[\tau\rho, \sigma\tau]=(\sigma\tau)^2$, $[\tau\rho, \sigma^2]=\sigma^4$, $[\sigma\rho, \sigma^2]=\sigma^4$,  and $[\sigma\rho, \tau]=\tau^2$.\\
So  $G_4'=\langle\tau^2, \sigma^4\rangle$ or
$\langle\tau^4,  (\sigma\tau)^2\rangle.$
Therefore $\mathbf{C}l_2(\KK_4)\simeq G_4/G_4'\simeq (2, 2, 2)$.\\
Compute   $\ker V_{G\rightarrow G_1}$, according as $\tau=\displaystyle\left(\frac{L/\kk}{\mathcal{H}_1\mathcal{H}_3}\right)$ or $\displaystyle\left(\frac{L/\kk}{\mathcal{H}_2\mathcal{H}_3}\right)$ we get\\
$\left\{ \begin{array}{ll}
\ast V_{G\rightarrow G_4}(\sigma G')=\sigma^2G_4'\neq G_4'.\\
\ast V_{G\rightarrow G_4}(\tau G')=\tau^2G_4'= G_4'.\\
\ast V_{G\rightarrow G_4}(\rho G')=\rho^2G_4'= G_4'.
\end{array} \right.$
 or
$\left\{ \begin{array}{ll}
\ast V_{G\rightarrow G_4}(\sigma G')=\sigma^2G_4'\neq G_4'.\\
\ast V_{G\rightarrow G_4}(\tau G')=\tau^2G_4'\neq G_4'.\\
\ast V_{G\rightarrow G_4}(\rho G')=\rho^2G_4'= G_4'.\\
\ast V_{G\rightarrow G_4}(\sigma\tau G')=(\sigma\tau)^2G_4'= G_4'.
\end{array} \right.$\\
Hence  $\ker V_{G\rightarrow G_4}=\langle\tau G', \rho G'\rangle$ or $\langle\sigma\tau G', \rho G'\rangle$, thus
 $$\kappa_{\KK_4/\kk}=\langle[\mathcal{H}_1], [\mathcal{H}_3]\rangle.$$
Proceeding similarly, we show the other tables inputs.
\subsubsection{\bf Capitulations kernels $\kappa_{\LL_j/\kk}$ and $\mathbf{C}l_2(\LL_j)$}
 From the subsection \ref{17}, we deduce that $\kappa_{\LL_j/\kk}=\mathbf{C}l_2(\kk)$. In what follows, we compute the Galois groups $\mathcal{G}_j=\mathrm{Gal}(\L/\LL_j)$, their derived groups $\mathcal{G}_j'$ and the abelian type invariants of $\mathbf{C}l_2(\LL_j)$. The results are summarized in the following Tables \ref{10:24} and \ref{10:22}. Keep the notations of the Subsection \ref{17} and  put:
 $\left\{
 \begin{array}{ll}
 a=\min(n, m),\\
  b=\max(n+1, m+1).
 \end{array}
 \right.$

 Hence in  the Table \ref{10:24}, the left hand sides (if they exist) of the columns $\mathcal{G}_j$ and $\mathbf{C}l_2(\LL_j)$ refer to the case  $\beta=1$, while the right ones refer to the case  $\beta=-1$;  and in the Table \ref{10:22},   the left hand side of the column $\mathcal{G}_j$ refers to the case  $\mathrm{I}=1$, while the right one refers to the case  $\mathrm{I}=-1$.
\scriptsize
 \begin{longtable}{|c c | c | c | c |}
\caption{ \small{Invariants of $\mathbf{C}l_2(\LL_j)$ for the case $\left(\frac{p_1}{p_2}\right)=1$}}\label{10:24}\\
\hline
 $\LL_j$  & & $\mathcal{G}_j$ & $\mathcal{G}_j'$ & $\mathbf{C}l_2(\LL_j)$\\
\hline
\endfirsthead
\hline
 $\LL_j$ & &$\mathcal{G}_j$ & $\mathcal{G}_j$ & $\mathbf{C}l_2(\LL_j)$\\
\hline
\endhead
 & && & $(2^a, 2^{b})$ if $N=-1$\\[-1ex]
 \raisebox{2ex}{$\LL_1$} && \raisebox{2ex}{$\langle \tau^2, \sigma\rangle$} & \raisebox{2ex}{$\langle1\rangle$} &
  $(2^{m}, 2^{n+1})$ if $N=1$\\[1ex] \hline
  & $\pi=-1$ &  $\langle\sigma\tau\rho, \sigma^2, \tau^2\rangle$  $\langle\tau\rho, \sigma^2, \tau^2\rangle$  & $\langle\sigma^4, \tau^4\rangle$ & $(2, 2, 2)$\\[-1ex]
  \raisebox{2ex}{$\LL_2$}
 &$\pi=1$ &  $\langle\tau\rho, \tau^2\rangle$  $\langle\sigma\tau\rho, \tau^2\rangle$   & $\langle \tau^4\rangle$ & $(2, 4)$\\[1ex]\hline
 &$\pi=-1$ &  $\langle\tau\rho, \sigma^2, \tau^2\rangle$  $\langle\sigma\tau\rho, \sigma^2, \tau^2\rangle$  & $\langle\sigma^4, \tau^4\rangle$ & $(2, 2, 2)$\\[-1ex]
\raisebox{2ex}{$\LL_3$}
& $\pi=1$& $\langle\sigma\tau\rho,  \tau^2\rangle$ $\langle\tau\rho,  \tau^2\rangle$  & $\langle \tau^4\rangle$ & $(2, 4)$\\ [1ex]\hline
  & $\pi=-1$&   $\langle \sigma\rho, \sigma^2, \tau^2\rangle$  & $\langle\sigma^4, \tau^4\rangle$ & $(2, 2, 2)$\\[-1ex]
 \raisebox{2ex}{$\LL_4$}
    &$\pi=1$     &    $\langle \rho,  \tau^2\rangle$  & $\langle \tau^4\rangle$ & $(2, 4)$ \\[1ex]\hline
  & $\pi=-1$&  $\langle \rho, \sigma^2, \tau^2\rangle$  & $\langle\sigma^4, \tau^4\rangle$ & $(2, 2, 2)$\\[-1ex]
\raisebox{1.7ex}{$\LL_5$}
  &$\pi=1$&   $\langle \rho\sigma,  \tau^2\rangle$  & $\langle \tau^4\rangle$ & $(2, 4)$\\[1ex]\hline
  & $\pi=-1$&    &  & $(2^{m-1}, 2^{n+2})$  $(2^{m}, 2^{n+1})$\\[-1ex]
\raisebox{1.7ex}{$\LL_6$}
  &$\pi=1$&   \raisebox{1.7ex}{$\langle \tau, \sigma^2\rangle$  $\langle\sigma\tau,  \sigma^2\rangle$}  & \raisebox{1.7ex}{$\langle1\rangle$} & $(2,  2^{n+2})$\\[1ex]\hline
 &$\pi=-1$ &      &  & $(2^{m}, 2^{n+1})$ $(2^{m-1}, 2^{n+1})$\\[-1ex]
\raisebox{2ex}{$\LL_7$}
   & $\pi=1$  &  \raisebox{1.7ex}{$\langle\sigma\tau,  \sigma^2\rangle$  $\langle \tau, \sigma^2\rangle$ }   & \raisebox{1.7ex}{$\langle1\rangle$} &
           $(2,  2^{n+2})$\\[1ex]\hline
\end{longtable}
\normalsize

\scriptsize
 \begin{longtable}{|c c | c | c | c |}
\caption{\large{ Invariants of $\mathbf{C}l_2(\LL_j)$ for the case $\left(\frac{p_1}{p_2}\right)=-1$}}\label{10:22}\\
\hline
 $\LL_j$  & & $\mathcal{G}_j$ & $\mathcal{G}_j'$ & $\mathbf{C}l_2(\LL_j)$\\
\hline
\endfirsthead
\hline
 $\LL_j$ & &$\mathcal{G}_j$ & $\mathcal{G}_j$ & $\mathbf{C}l_2(\LL_j)$\\
\hline
\endhead
 & && & $(2^{a}, 2^{b})$ if $N=-1$\\[-1ex]
 \raisebox{2ex}{$\LL_1$} && \raisebox{2ex}{$\langle \tau^2, \sigma\rangle$} & \raisebox{2ex}{$\langle1\rangle$} &
  $(2^{m}, 2^{n+1})$ if $N=1$\\[1ex] \hline
  &$\pi=1$ &  $\langle\sigma\rho, \tau^2\rangle$  &  & \\[-1ex]
 \raisebox{2ex}{$\LL_2$}
 &$\pi=-1$ &$\langle\rho, \tau^2\rangle$ & \raisebox{2ex}{$\langle \tau^4\rangle$} & \raisebox{2ex}{$(2, 4)$}\\\hline
  &$\pi=1$ &  $\langle\rho, \tau^2\rangle$  &  & \\[-1ex]
 \raisebox{2ex}{$\LL_3$}
 &$\pi=-1$ &$\langle\sigma\rho, \tau^2\rangle$ & \raisebox{2ex}{$\langle \tau^4\rangle$} & \raisebox{2ex}{$(2, 4)$}\\\hline
 $\LL_4$ & &  $\langle\sigma\tau\rho, \sigma^2\rangle$  $\langle\tau\rho, \sigma^2\rangle$  & $\langle \sigma^4\rangle$ & $(2, 4)$\\ \hline
$\LL_5$ & &   $\langle\tau\rho, \sigma^2\rangle$  $\langle\sigma\tau\rho, \sigma^2\rangle$  & $\langle \sigma^4\rangle$ & $(2, 4)$\\ \hline
 & $\pi=1$ & $\langle\tau, \sigma^2\rangle$  $\langle\sigma\tau,  \tau^2\rangle$   & $\langle1\rangle$ & $(2^{2},  2^{m})$ \\[-1ex]
\raisebox{2ex}{$\LL_6$}
      & $\pi=-1$  & $\langle\sigma\tau,  \tau^2\rangle$   $\langle\tau, \sigma^2\rangle$   & $\langle1\rangle$ & $(2,  2^{m+1})$\\[1ex]\hline
 &$\pi=1$ &  $\langle\sigma\tau,  \tau^2\rangle$   $\langle\tau, \sigma^2\rangle$     & $\langle1\rangle$ & $(2,  2^{m+1})$\\[-1ex]
\raisebox{2ex}{$\LL_7$}
   & $\pi=-1$  &  $\langle\tau, \sigma^2\rangle$  $\langle\sigma\tau,  \tau^2\rangle$   & $\langle1\rangle$ & $(2^{2},  2^{m})$\\[1ex]\hline
\end{longtable}
\normalsize

Check the entries in  some cases.\\
\indent $\ast$ Take $\mathbb{L}_1=\k=\KK_1.\KK_2.\KK_3$.  Since $\mathrm{Gal}(L/\LL_1)=\mathcal{G}_1=G_1\cap G_2$, then \\
 $ \mathcal{G}_1=
\langle\sigma, \tau\rho, \tau^2\rangle\cap\langle\sigma, \rho, \tau^2\rangle=\langle\sigma, \tau^2\rangle, $
 thus $\mathcal{G}_1'=\langle1\rangle$. As \\
  $\left\{\begin{array}{ll}
\sigma^{2^m}=\tau^{2^{n+1}}=1  & \hbox{ if } N(\varepsilon_{p_1p_2})=1,\\
\sigma^{2^m}=\tau^{2^{n+1}} \hbox{ and } \sigma^{2^{m+1}}=\tau^{2^{n+2}}=1  & \hbox{ if } N(\varepsilon_{p_1p_2})=-1,
\end{array}\right.$\\
  so $\mathbf{C}l_2(\LL_1)\simeq\left\{\begin{array}{ll}
(2^{n+1}, 2^m)  & \hbox{ if } N(\varepsilon_{p_1p_2})=1,\\
(2^{\min(n,m)}, 2^{\max(n+1,m+1)})  & \hbox{ if } N(\varepsilon_{p_1p_2})=-1.
\end{array}\right.$\\
\indent $\ast$ Take
 $\mathbb{L}_2=\KK_1.\KK_4.\KK_6$ and assume that $(\frac{p_1}{p_2})=1$, then $\mathcal{G}_2=\mathrm{Gal}(L/\LL_2)=G_1\cap G_4\cap G_6$. There are two  cases to distinguish: \\
 - 1\up{st} case: If $\left(\frac{p_1}{p_2}\right)_4\left(\frac{p_2}{p_1}\right)_4=\left(\frac{\pi_1}{\pi_3}\right)=-1$, then
$\mathcal{G}_2=\langle\sigma, \tau\rho, \tau^2\rangle \cap \langle\sigma\tau, \rho, \tau^2\rangle=\langle\sigma\tau\rho, \sigma^2, \tau^2\rangle
$ or $\langle\sigma, \tau\rho, \tau^2\rangle \cap \langle\sigma\tau, \tau\rho, \sigma^2\rangle=\langle\tau\rho, \sigma^2, \tau^2\rangle$ according as $\beta=1$ or $-1$.
  Thus $\mathcal{G}_2'=\langle\sigma^4, \tau^4\rangle$. On the other hand,  in this case we have  $(\sigma\tau\rho)^2=(\tau\rho)^2=\rho^2\tau^{2^{n+1}}$, so   $\mathbf{C}l_2(\LL_2)\simeq(2, 2, 2)$.\\
  - 2\up{nd} case: If $\left(\frac{p_1}{p_2}\right)_4\left(\frac{p_2}{p_1}\right)_4=\left(\frac{\pi_1}{\pi_3}\right)=1$, then\\
$\mathcal{G}_2=
\langle\sigma, \tau\rho, \tau^2\rangle \cap \langle\tau, \rho, \sigma^2\rangle=\langle\tau\rho, \sigma^2, \tau^2\rangle$
or $\langle\sigma, \tau\rho, \tau^2\rangle \cap \langle\sigma\tau, \rho, \sigma^2\rangle=\langle\sigma\tau\rho, \sigma^2, \tau^2\rangle$ according as $\beta=1$ or $-1$.
As in this case\\ $(\tau\rho)^2=(\sigma\tau\rho)^2=\left\{\begin{array}{ll}
                                                               \rho^2 & \text{ if } N(\varepsilon_{p_1p_2})=-1,\\
                                                               \rho^2\tau^{2^{n+1}} & \text{ if } N(\varepsilon_{p_1p_2})=1;
                                                               \end{array}\right.$\\  so
$\mathcal{G}_2=
\langle\tau\rho,  \tau^2\rangle$ or
$\langle\sigma\tau\rho,  \tau^2\rangle$.
 We infer that $\mathcal{G}_2'=\langle\tau^4\rangle$. From which we deduce that   $\mathbf{C}l_2(\LL_2)\simeq(2, 4)$, since $(\sigma\tau\rho)^4=(\tau\rho)^4=1$.\\
 Assume now that $(\frac{p_1}{p_2})=-1$, then
$\mathcal{G}_2=
 \left\{\begin{array}{ll}
\langle\rho, \tau^2\rangle  & \hbox{ if } \left(\frac{\pi_1}{\pi_3}\right)=-1,\\
\langle\sigma\rho,  \tau^2\rangle  & \hbox{ if } \left(\frac{\pi_1}{\pi_3}\right)=1.
\end{array}\right.$\\
  Thus $\mathcal{G}_2'=\langle\tau^4\rangle$, hence   $\mathbf{C}l_2(\LL_2)\simeq(2, 4)$.

\indent $\ast$ Take
 $\mathbb{L}_6=\KK_3.\KK_4.\KK_7$ and assume that $\left(\frac{p_1}{p_2}\right)=-1$, then $\mathcal{G}_6=\mathrm{Gal}(L/\LL_4)=G_3\cap G_4\cap G_7$. There are two  cases to distinguish: \\
 \indent $\bullet$ 1\up{st} case: If $\left(\frac{p_1p_2}{2}\right)_4\left(\frac{2p_1}{p_2}\right)_4\left(\frac{2p_2}{p_1}\right)_4=-1$, then Lemma \ref{10:12} implies that
 $n=1$ and $m=2$, hence $\sigma^4=\tau^4$ and $\rho^2=\sigma^2\tau^2$. We have also two sub-cases to discus:\\
\indent $a$ - If $\left(\frac{\pi_1}{\pi_3}\right)=1$, then Table \ref{10:21} yields that
$\mathcal{G}_6=\langle\sigma\tau, \tau^2\rangle=\langle\sigma\tau, \sigma^2\rangle$,
  thus $\mathcal{G}_6'=\langle1\rangle$. As $(\sigma\tau)^4=\sigma^4\tau^4=\tau^8=1$ and  $(\sigma^2)^{2^m}=\sigma^8=1$, so   $\mathbf{C}l_2(\LL_6)\simeq(2^2, 2^m)$.\\
\indent $b$ - If $\left(\frac{\pi_1}{\pi_3}\right)=-1$, then Table \ref{10:21} yields that
$\mathcal{G}_6=\langle\sigma^2, \tau\rangle$,
  thus $\mathcal{G}_6'=\langle1\rangle$. As $(\sigma^2)^{2}=\tau^4$ and $\tau^{2^{n+2}}=\tau^8=1$, so   $\mathbf{C}l_2(\LL_6)\simeq(2, 2^{m+1})$.\\
 \indent $\bullet$ 2\up{nd} case: If $\left(\frac{p_1p_2}{2}\right)_4\left(\frac{2p_1}{p_2}\right)_4\left(\frac{2p_2}{p_1}\right)_4=1$, then Lemma \ref{10:12} implies that
 $n=1$ and $m\geq3$, hence $\sigma^{2^m}=\tau^4$ and $\rho^2=\tau^2\sigma^{2^{m-1}}$. We have also two sub-cases to discus:\\
\indent $a$ - If $\left(\frac{\pi_1}{\pi_3}\right)=1$, then Table \ref{10:21} yields that
$\mathcal{G}_6=\langle\tau, \sigma^2\rangle$,
  thus $\mathcal{G}'_6=\langle1\rangle$. As $\tau^{2^2}=\tau^4=\sigma^{2^m}=(\sigma^2)^{2^{m-1}}$ and  $(\sigma^2)^{2^m}=\sigma^{2^{m+1}}=1$, so   $\mathbf{C}l_2(\LL_6)\simeq(2^2, 2^m)$.\\
\indent $b$ - If $\left(\frac{\pi_1}{\pi_3}\right)=-1$, then Table \ref{10:21} yields that
$\mathcal{G}_6=\langle\sigma\tau, \tau^2\rangle$,
  thus $\mathcal{G}_6'=\langle1\rangle$. As $(\tau^2)^2=\tau^4=\sigma^{2^m}=\sigma^{2^m}\tau^{2^m}=(\sigma\tau)^{2^m}$ and $\sigma\tau$ is of order $2^{m+1}$, so   $\mathbf{C}l_2(\LL_6)\simeq(2, 2^{m+1})$.
The other tables entries are checked similarly.

\section{Numerical examples}
\label{s:Examples}
To obtain a first impression of how the $2$-tower groups $G=\mathrm{Gal}(\L/\kk)$ are distributed on the coclass graphs $\mathcal{G}(2,r)$ for $3\le r\le 4$ \cite[\S\ 2, p. 410]{Ma-13}, we have analyzed the $207$ bicyclic biquadratic fields $\kk=\QQ(\sqrt{d},i)$ with $d=p_1p_2q<50000$. All computations were done with the aid of PARI/GP \cite{GP-11}.
Using the presentations in terms of the parameters $m$, $n$, and $N=N(\varepsilon_{p_1p_2})$ in Theorem \ref{10:2}, we can identify these Galois groups $G$ of the maximal unramified pro-$2$ extensions $\L$ of bicyclic biquadratic fields $\kk$ by means of the SmallGroups Library \cite{BEO-05} by Besche, Eick, and O'Brien. It turns out that most of the occurring groups are members of {\it one-parameter families}, which can be identified with infinite {\it periodic sequences} in the sense of \cite[p. 411]{Ma-13}. For this identification it is convenient to give {\it polycyclic presentations} of the groups which are particularly adequate for emphasizing the structure of the lower central series of $G$ and of the coclass tree where $G$ is located.

%\newpage
%--------------------------------------------------------------------------------
\tiny
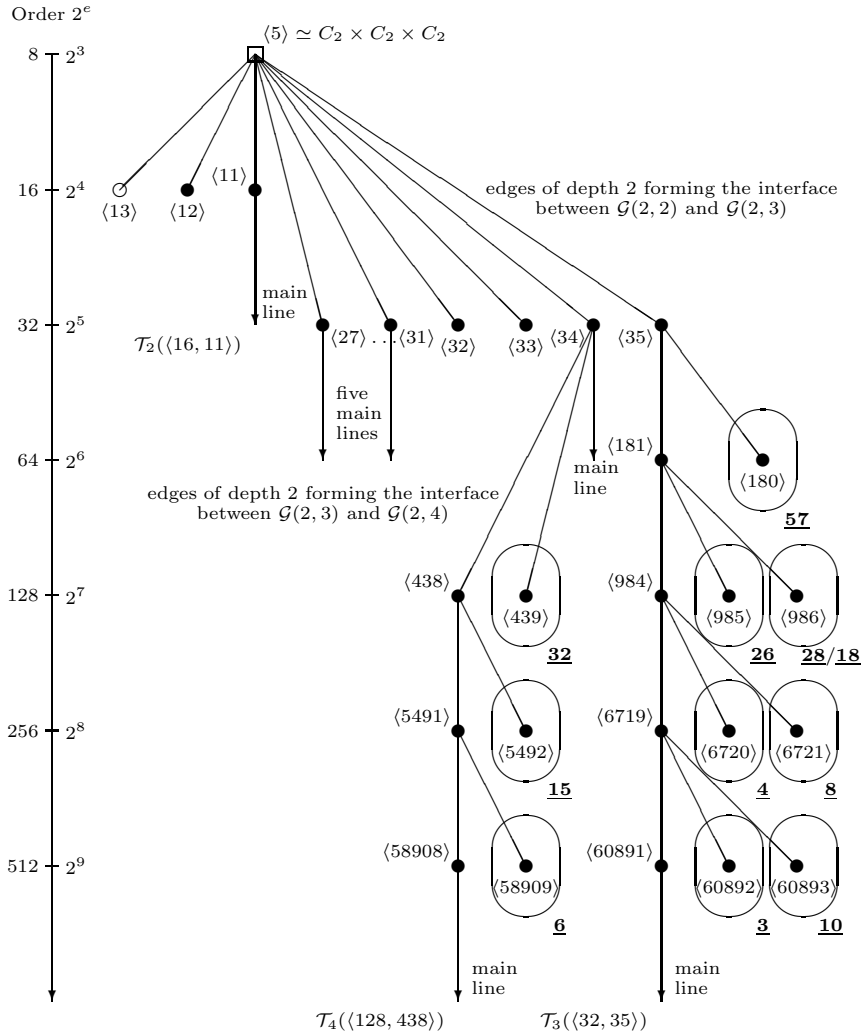
\begin{figure}[ht]
\caption{Relevant part of the descendant tree of $C_2\times C_2\times C_2$}
\label{fig:DescendantTree}
\input{NewmanOBrienExt}
\end{figure}
%\newpage
%--------------------------------------------------------------------------------

\normalsize
The descendant tree of the elementary abelian $2$-group of rank $3$ is drawn for the first time in Figure \ref{fig:DescendantTree}. Vertices $G$ of this diagram are classified according to their centre $\zeta(G)$ by using different symbols:
\begin{enumerate}
\item
a large contour square $\square$ denotes an abelian group,
\item
a big contour circle {\Large $\circ$} represents a metabelian group with cyclic centre of order $4$, and
\item
big full discs {\Large $\bullet$} represent metabelian groups with bicyclic centre of type $(2,2)$.
\end{enumerate}
Groups are labelled by a number in angles which is the identifier in the SmallGroups Library \cite{BEO-05}. Here, we omit the order which is given on the left hand scale. The actual distribution of the $207$ second $2$-class groups $G$ of bicyclic biquadratic fields $\kk=\mathbb{Q}(\sqrt{d},i)$ of type $(2,2,2)$
with radicand $0<d=p_1p_2q<50000$ is represented by underlined boldface counters (absolute frequencies) of hits of the vertex surrounded by the adjacent oval.

\subsection{The tree $\mathcal{T}_3(\langle 32,35\rangle)$ of the coclass graph $\mathcal{G}(2,3)$}
\label{ss:Coclass3}
Without exceptions, all groups $G$ of coclass $3$ are vertices of the coclass tree with root $\langle 32,35\rangle$.
They can be partitioned into four categories.
Firstly, the {\it pre-periodic} vertex $\langle 64,180\rangle$ on branch $1$ of this tree.
The associated fields have parameters $m=2$, $n=1$, $N=-1$ and $(\frac{p_1}{p_2})=+1$.
Periodicity sets in with branch $2$ and gives rise to {\it two periodic sequences} which form the second and third category.
A periodic sequence always consists of a ground state and infinitely many higher excited states \cite[p. 424]{Ma-13}.
Finally, there is a {\it variant} of the vertex $\langle 128,986\rangle$, exceptional from the number theoretic point of view,
since the maximal subgroups of $G$ are associated to the unramified quadratic extensions of $\kk$
in a different manner than for the group-theoretically isomorphic ground state of the second periodic sequence.

\begin{exam}
\label{exm:PrePeriod}
The most frequent group $G\simeq\langle 64,180\rangle$ with parameters $m=2$, $n=1$ occurs $57$ times ($28\%$) for the radicands $d\in\lbrace 455,795,2955,\ldots\rbrace$ with $N=-1$ and $(\frac{p_1}{p_2})=+1$.
It is a terminal leaf on the first branch of the coclass tree $\mathcal{T}_3(\langle 32,35\rangle)$, whose mainline vertices \cite[pp. 410--411]{Ma-13}
arise as finite quotients of the infinite topological pro-$2$ group given by Newman and O'Brien in \cite[App. A, no. 79, p. 155, and App. B, Tbl. 79, p. 168]{NmOb-99}. They form the infinite periodic sequence
$(\langle 32,35\rangle,\langle 64,181\rangle,\langle 128,984\rangle,\langle 256,6719\rangle,\langle 512,60891\rangle,\ldots)$
with parametrized pc-presentation ($n\ge 2$)
{\tiny
\begin{align*}
\langle x,y,z\mid & s_2=\lbrack y,x\rbrack,\ t_2=\lbrack z,x\rbrack,\ t_j=\lbrack t_{j-1},x\rbrack\text{ for }3\le j\le n+2, \\
                  & x^2=s_2,\ y^2=s_2,\ z^2=t_2t_3,\ s_2^2=1,\ t_j^2=t_{j+1}t_{j+2}\text{ for }2\le j\le n,\ t_{n+1}^2=t_{n+2}\rangle.
\end{align*}
}
\end{exam}

\normalsize
\begin{exam}
\label{exm:m2nVarNormMinus}
(Second $2$-class groups $G$ with parameters $m=2$, $n\ge 2$, $N=-1$)
The groups of class $c(G)=n+2$,
\tiny
$$G=\langle\rho,\sigma,\tau\mid\rho^4=\sigma^8=\tau^{2^{n+2}}=1,\ \rho^2=\tau^{2^n}\sigma^2,\ \sigma^4=\tau^{2^{n+1}},\
\lbrack\sigma,\tau\rbrack=1,\ \lbrack\sigma,\rho\rbrack=\sigma^2,\ \lbrack\rho,\tau\rbrack=\tau^2\rangle,$$
\normalsize
with $n\ge 2$, form the infinite periodic sequence
$(\langle 128,985\rangle,\langle 256,6720\rangle,\langle 512,60892\rangle,$ $\ldots)$
with parametrized pc-presentation
{\tiny
\begin{align*}
\langle x,y,z\mid & s_2=\lbrack y,x\rbrack,\ t_2=\lbrack z,x\rbrack,\ t_j=\lbrack t_{j-1},x\rbrack\text{ for }3\le j\le n+2, \\
                  & x^2=s_2,\ y^2=s_2t_{n+2},\ z^2=t_2t_3,\ s_2^2=1,\ t_j^2=t_{j+1}t_{j+2}\text{ for }2\le j\le n,\ t_{n+1}^2=t_{n+2}\rangle.
\end{align*}
}
\normalsize
The ground state $\langle 128,985\rangle$ with $n=2$ occurs $26$ times ($13\%$) for $d\in\lbrace 435,6235,$ $6815,\ldots\rbrace$,
the first excited state $\langle 256,6720\rangle$ with $n=3$ occurs $4$ times for $d\in\lbrace 15915,17139,42915,46587\rbrace$, and
the second excited state $\langle 512,60892\rangle$ with $n=4$ occurs $3$ times for $d\in\lbrace 6915,16135,39315\rbrace$.
\end{exam}

\normalsize
\begin{exam}
\label{exm:n1mVarNormMinus}
(Second $2$-class groups $G$ with parameters $m\ge 3$, $n=1$, $N=-1$)\\
The groups of class $c(G)=m+1$,
\tiny
$$G=\langle\rho,\sigma,\tau\mid\rho^4=\sigma^{2^{m+1}}=\tau^8=1,\ \rho^2=\tau^2\sigma^{2^{m-1}},\ \sigma^{2^m}=\tau^4,\
\lbrack\sigma,\tau\rbrack=1,\ \lbrack\sigma,\rho\rbrack=\sigma^{2^m-2},\ \lbrack\rho,\tau\rbrack=\tau^2\rangle,$$
\normalsize
with $m\ge 3$, form the infinite periodic sequence
$(\langle 128,986\rangle,\langle 256,6721\rangle,\langle 512,60893\rangle,$ $\ldots)$
with parametrized pc-presentation
{\tiny
\begin{align*}
\langle x,y,z\mid & s_2=\lbrack y,x\rbrack,\ t_2=\lbrack z,x\rbrack,\ t_j=\lbrack t_{j-1},x\rbrack\text{ for }3\le j\le m+1, \\
                  & x^2=s_2,\ y^2=s_2,\ z^2=t_2t_3t_{m+1},\ s_2^2=1,\ t_j^2=t_{j+1}t_{j+2}\text{ for }2\le j\le m-1,\ t_m^2=t_{m+1}\rangle.
\end{align*}
}
\normalsize
The ground state $\langle 128,986\rangle$ with $m=3$ occurs $28$ times ($14\%$) for $d\in\lbrace 2595,4255,$ $4395,\ldots\rbrace$,
the first excited state $\langle 256,6721\rangle$ with $m=4$ occurs $8$ times for $d\in\lbrace 8355,19155,24195,\ldots\rbrace$, and
the second excited state $\langle 512,60893\rangle$ with $m=5$ occurs $10$ times for $d\in\lbrace 19459,26663,28171,\ldots\rbrace$.
We also observed a third excited state $\langle 512,60891\rangle-\#1;3$ with $m=6$ for a single radicand $d=79651$ outside of the range of our systematic investigations.
\end{exam}

\normalsize
We point out again that the group $\langle 64,180\rangle$ with $m=2$, $n=1$, $N=-1$ is pre-periodic and its presentation does not fit into either of the last two examples, whence it does not belong to either of the two mentioned periodic non-mainline sequences.

The last three examples suggest a very promising {\it characterization of mainline vertices} on coclass trees of arbitrary $p$-groups, which seems to be of a fairly general nature and obviously has not been recognized by other investigators up to now.

\begin{rema}
\label{rmk:MainlinePrinciple}
{\bf Mainline Principle:}
The relators for \(p\)th powers of generators of mainline groups are distinguished by being independent of the class, whereas a relator of any non-mainline group contains the generator of the last non-trivial term of the lower central series as a \textit{small perturbation}.
\end{rema}

In Example \ref{exm:m2nVarNormMinus}, resp. \ref{exm:n1mVarNormMinus}, the small perturbation is the generator $t_{n+2}\in\gamma_{n+2}(G)$, resp. $t_{m+1}\in\gamma_{m+1}(G)$, of the last non-trivial lower central in the relation $y^2=s_2t_{n+2}$, where $n+2=c(G)$, resp. $z^2=t_2t_3t_{m+1}$, where $m+1=c(G)$. For the rest, the presentation coincides with the mainline presentation.

\begin{exam}
\label{exm:Variant}
There exists a unique group which can be characterized by two distinct couples of parameters $(m,n)$.
Whereas $m=n=2$ with $N=-1$ gives rise to $\langle 128,985\rangle$,
the same parameters $m=n=2$ with $N=+1$ define a {\it variant} of $\langle 128,986\rangle$, which was given by $m=3$, $n=1$, $N=-1$ already.
This variant is realized $18$ times ($9\%$) for the radicands $d\in\lbrace 1515, 3535, 5551,\ldots\rbrace$.

\end{exam}

\subsection{The tree $\mathcal{T}_4(\langle 128,438\rangle)$ of the coclass graph $\mathcal{G}(2,4)$}
\label{ss:Coclass4}
The groups $G$ of coclass $4$ are either vertices of the coclass tree with root $\langle 128,438\rangle$
or they populate the single sporadic {\it isolated} vertex $\langle 128,439\rangle$ outside of any coclass trees.
The associated fields are characterized by $N=+1$ and $(\frac{p_1}{p_2})=+1$, according to Theorem \ref{10:2}.
Periodicity sets in with branch $1$ already and gives rise to a {\it single periodic non-mainline sequence}.
The mainline of this tree is given by the parametrized pc-presentation ($m\ge 3$)
{\tiny
\begin{align*}
\langle x,y,z\mid & s_2=\lbrack y,x\rbrack,\ t_2=\lbrack z,x\rbrack,\ s_j=\lbrack s_{j-1},x\rbrack\text{ for }3\le j\le m,\ t_3=\lbrack t_2,x\rbrack, \\
                  & x^2=1,\ y^2=s_2s_3,\ z^2=t_2,\ s_j^2=s_{j+1}s_{j+2}\text{ for }2\le j\le m-2,\ s_{m-1}^2=s_m,\ t_2^2=t_3\rangle.
\end{align*}
}

\normalsize
\begin{exam}
\label{exm:Isolated}
With $32$ occurrences ($15\%$) the second largest density of population arises for the {\it sporadic isolated vertex} $\langle 128,439\rangle$ with parameters $m=3$, $n=1$, $N=+1$, which occurs for $d\in\lbrace 2135,2235,4035,\ldots\rbrace$.
The realm of coclass $cc(G)=4$ commences with this immediate descendant of {\it depth two} of $\langle 32,34\rangle$, called an {\it interface group} at the border of the coclass graphs $\mathcal{G}(2,3)$ and $\mathcal{G}(2,4)$ in \cite[pp. 430--434]{Ma-13}. It is an isolated top-vertex of the sporadic part of $\mathcal{G}(2,4)$, lying outside of any coclass trees. In particular, it has nothing to do with the coclass tree of $\langle 64,174\rangle$ arising from its generalized parent $\langle 32,34\rangle$, whose projective mainline limit is given by \cite[App. A, no. 78, p. 155, and App. B, Tbl. 78, p. 167]{NmOb-99}.
\end{exam}

\begin{exam}
\label{exm:n1mVarNormPlus}
(Second $2$-class groups $G$ with parameters $m\ge 4$, $n=1$, $N=+1$)\\
The groups of class $c(G)=m$,
\tiny
$$G=\langle\rho,\sigma,\tau\mid\rho^4=\sigma^{2^m}=\tau^8=1,\ \rho^2=\tau^4\sigma^{2^{m-1}},\
\lbrack\sigma,\tau\rbrack=1,\ \lbrack\rho,\sigma\rbrack=\sigma^2,\ \lbrack\tau,\rho\rbrack=\tau^2\rangle,$$
\normalsize
with $m\ge 4$, form the infinite periodic sequence
$(\langle 256,5492\rangle,\langle 512,58909\rangle,\ldots)$
with parametrized pc-presentation
{\tiny
\begin{align*}
\langle x,y,z\mid & s_2=\lbrack y,x\rbrack,\ t_2=\lbrack z,x\rbrack,\ s_j=\lbrack s_{j-1},x\rbrack\text{ for }3\le j\le m,\ t_3=\lbrack t_2,x\rbrack, \\
                  & x^2=s_m,\ y^2=s_2s_3,\ z^2=t_2,\ s_j^2=s_{j+1}s_{j+2}\text{ for }2\le j\le m-2,\ s_{m-1}^2=s_m,\ t_2^2=t_3\rangle.
\end{align*}
}
\normalsize
The group $\langle 128,439\rangle$ with $m=3$, $n=1$, $N=+1$ is sporadic and even isolated, but its presentation is also of the same form.
The ground state $\langle 256,5492\rangle$ with $m=4$ occurs $15$ times ($7\%$) for $d\in\lbrace 10515,12535,12963,\ldots\rbrace$,
the first excited state $\langle 512,58909\rangle$ with $m=5$ occurs $6$ times for $d\in\lbrace 34635,41115,41835,\ldots\rbrace$.
\end{exam}

We conclude this section with the following tables: Table \ref{29} gives the structure of the class group $\mathbf{C}l(\kk)$ of the bicyclic biquadratic field $\kk$, its discriminant $disc(\kk)$, the structure of the class groups of its two quadratic subfields $k_0$ and $\overline{k}_0$, and the coclass of $G$. Tables \ref{30} and \ref{31}, resp. \ref{32} and \ref{33}, give the structure of the class groups $\mathbf{C}l(\KK_j)$, resp. $\mathbf{C}l(\LL_j)$, for the case $(\frac{p_1}{p_2})=1$. Finally, Tables \ref{34} and \ref{35}, resp. \ref{36}, give the structure of the class groups $\mathbf{C}l(\KK_j)$, resp. $\mathbf{C}l(\LL_j)$, for the case $(\frac{p_1}{p_2})=-1$. We briefly put $N=N(\varepsilon_{p_1p_2})$, $\gamma=\left(\frac{p_1}{p_2}\right)$ and  $\delta=\left(\frac{p_1}{p_2}\right)_4\left(\frac{p_2}{p_1}\right)_4$.

\tiny
\begin{longtable}{| c | c | c | c | c | c | c | c | c | c |}
\caption{\small{ Invariants of $\kk$}\label{29}}\\
\hline
 $d$  $ = p_1.p_2.q$ & $\gamma$ & $\delta$ & $N$ & $m$, $n$
             & $\mathbf{C}l(k_0)$ & $\mathbf{C}l(\overline{k}_0)$ & $\mathbf{C}l(\kk)$ & $disc(\kk)$ & $cc(G)$ \\
\hline
\endfirsthead
\hline
 $d$  $ = p_1.p_2.q$ & $\gamma$ & $\delta$ & $N$ & $m$, $n$
             & $\mathbf{C}l(k_0)$ & $\mathbf{C}l(\overline{k}_0)$ & $\mathbf{C}l(\kk)$ & $disc(\kk)$ & $cc(G)$ \\
\hline
\endhead
$435 = 5.29.3 $ & $1$ & $1 $ & $-1$ & $2$, $2$ & $(2, 2) $ & $ (2, 2) $ & $ (2 , 2, 2) $ & $3027600 $ & $ 3$\\
$455 = 5.13.7 $ & $-1$ &  & $-1$ & $2$, $1$ & $(2, 2) $ & $ (10, 2) $ & $ (10 , 2, 2) $ & $3312400 $ & $ 3$\\
$795 = 5.53.3 $ & $-1$ &  & $-1$ & $2$, $1$ & $(2, 2) $ & $ (2, 2) $ & $ (2 , 2, 2) $ & $10112400 $ & $ 3$\\
$1515 = 5.101.3 $ & $1$ & $1 $ & $1$ & $2$, $2$ & $(2, 2) $ & $ (6, 2) $ & $ (6 , 2, 2) $ & $36723600 $ & $ 3$\\
$2135 = 5.61.7 $ & $1$ & $-1 $ & $1$ & $3$, $1$ & $(2, 2) $ & $ (22, 2) $ & $ (22 , 2, 2) $ & $72931600 $ & $ 4$\\
$2235 = 5.149.3 $ & $1$ & $-1 $ & $1$ & $3$, $1$ & $(2, 2) $ & $ (6, 2) $ & $ (6 , 2, 2) $ & $79923600 $ & $ 4$\\
$2595 = 5.173.3 $ & $-1$ &  & $-1$ & $3$, $1$ & $(2, 2) $ & $ (6, 2) $ & $ (6 , 2, 2) $ & $107744400 $ & $ 3$\\
$2955 = 5.197.3 $ & $-1$ &  & $-1$ & $2$, $1$ & $(2, 2) $ & $ (6, 2) $ & $ (6 , 2, 2) $ & $139712400 $ & $ 3$\\
$3055 = 5.13.47 $ & $-1$ &  & $-1$ & $2$, $1$ & $(2, 2) $ & $ (18, 2) $ & $ (18 , 2, 2) $ & $149328400 $ & $ 3$\\
$3535 = 5.101.7 $ & $1$ & $1 $ & $1$ & $2$, $2$ & $(2, 2) $ & $ (14, 2) $ & $ (14 , 2, 2) $ & $199939600 $ & $ 3$\\
$5551 = 13.61.7 $ & $1$ & $1 $ & $1$ & $2$, $2$ & $(2, 2) $ & $ (26, 2) $ & $ (26 , 2, 2) $ & $493017616 $ & $ 3$\\
$5835 = 5.389.3 $ & $1$ & $-1 $ & $1$ & $3$, $1$ & $(2, 2) $ & $ (6, 2) $ & $ (6 , 2, 2) $ & $544755600 $ & $ 4$\\
$7015 = 5.61.23 $ & $1$ & $-1 $ & $1$ & $3$, $1$ & $(2, 2) $ & $ (14, 2) $ & $ (14 , 2, 2) $ & $787363600 $ & $ 4$\\
$7163 = 13.29.19 $ & $1$ & $-1 $ & $1$ & $3$, $1$ & $(2, 2) $ & $ (10, 2) $ & $ (10 , 2, 2) $ & $820937104 $ & $ 4$\\
$9415 = 5.269.7 $ & $1$ & $-1 $ & $1$ & $3$, $1$ & $(2, 2) $ & $ (26, 2) $ & $ (26 , 2, 2) $ & $1418275600 $ & $ 4$\\
$10515 = 5.701.3 $ & $1$ & $-1 $ & $1$ & $4$, $1$ & $(2, 2) $ & $ (10, 2) $ & $ (10 , 2, 2) $ & $1769043600 $ & $ 4$\\
\hline
\end{longtable}
\begin{longtable}{| p{0.83in} | p{0.23in} | p{0.4in} |p{0.4in} |p{0.3in} | p{0.34in} | p{0.34in} |p{0.34in} | p{0.34in} |}
\caption{\small{ Invariants of $\KK_j$ for the case $\left(\frac{p_1}{p_2}\right)=1$ and $\left(\frac{\pi_1}{\pi_3}\right)=-1$}\label{30}}\\
\hline
 $d = p_1.p_2.q$ & $m$, $n$
& $\mathbf{C}l(\KK_1)$ & $\mathbf{C}l(\KK_2)$ & $\mathbf{C}l(\KK_3)$ & $\mathbf{C}l(\KK_4)$ & $\mathbf{C}l(\KK_5)$ & $\mathbf{C}l(\KK_6)$ & $\mathbf{C}l(\KK_7)$ \\
\hline
\endfirsthead
\hline
 $d = p_1.p_2.q$  & $m$, $n$
& $\mathbf{C}l(\KK_1)$ & $\mathbf{C}l(\KK_2)$ & $\mathbf{C}l(\KK_3)$ & $\mathbf{C}l(\KK_4)$ & $\mathbf{C}l(\KK_5)$ & $\mathbf{C}l(\KK_6)$ & $\mathbf{C}l(\KK_7)$ \\
\hline
\endhead
$2135 = 5.61.7 $ & $3$, $1$ & $(66, 2, 2) $ & $ (66, 2, 2) $ & $ (88, 8) $ & $(22,2, 2) $ & $ (22,2,2)$ & $(22,2,2)$ & $(22,2, 2)$\\
$2235 = 5.149.3 $ & $3$, $1$ & $(42, 2, 2) $ & $ (42, 2, 2) $ & $ (24, 8) $ & $(6,2, 2) $ & $ (6,6,2)$ & $(6,6,2)$ & $(6,2, 2)$\\
$4035 = 5.269.3 $ & $3$, $1$ & $(42, 2, 2) $ & $ (66, 2, 2) $ & $ (24, 24) $ & $(6,2, 2) $ & $ (6,2,2)$ & $(6,2,2)$ & $(6,2, 2)$\\
$4147 = 13.29.11 $ & $3$, $1$ & $(30, 2, 2) $ & $ (30, 6, 2) $ & $ (24, 8) $ & $(6,2, 2) $ & $ (6,2,2)$ & $(6,2,2)$ & $(6,2, 2)$\\
$4611 = 29.53.3 $ & $3$, $1$ & $(210, 2, 2) $ & $ (42, 6, 2) $ & $ (56, 8) $ & $(70,2, 2) $ & $ (42,2,2)$ & $(42,2,2)$ & $(70,2, 2)$\\
$5835 = 5.389.3 $ & $3$, $1$ & $(66, 2, 2) $ & $ (66, 2, 2) $ & $ (24, 8) $ & $(6,2, 2) $ & $ (6,2,2)$ & $(6,2,2)$ & $(6,2, 2)$\\
$14287 = 13.157.7 $ & $4$, $1$ & $(18, 6, 2) $ & $ (18, 6, 2) $ & $ (144, 8) $ & $(18,2, 2) $ & $ (18,2,2)$ & $(18,2,2)$ & $(18,2, 2)$\\
$15051 = 29.173.3 $ & $4$, $1$ & $(126, 6, 2) $ & $ (42, 14, 2) $ & $ (112, 8) $ & $(14,2, 2) $ & $ (14,2,2)$ & $(14,2,2)$ & $(14,2, 2)$\\
$17715 = 5.1181.3 $ & $4$, $1$ & $(18, 18, 2) $ & $ (414, 2, 2) $ & $ (144, 8) $ & $(18,2, 2) $ & $ (18,2,2)$ & $(18,2,2)$ & $(18,2, 2)$\\
$19515 = 5.1301.3 $ & $4$, $1$ & $(234, 2, 2) $ & $ (450, 2, 2) $ & $ (144, 8) $ & $(18,2, 2) $ & $ (18,2,2)$ & $(18,2,2)$ & $(18,2, 2)$\\
\hline
\end{longtable}
\begin{longtable}{| p{0.83in} | p{0.23in} | p{0.4in} |p{0.4in} |p{0.34in} | p{0.34in} | p{0.34in} |p{0.34in} | p{0.34in} |}
\caption{\small{ Invariants of $\KK_j$ for the case $\left(\frac{p_1}{p_2}\right)=1$ and $\left(\frac{\pi_1}{\pi_3}\right)=1$}\label{31}}\\
\hline
 $d = p_1.p_2.q$ & $m$, $n$
& $\mathbf{C}l(\KK_1)$ & $\mathbf{C}l(\KK_2)$ & $\mathbf{C}l(\KK_3)$ & $\mathbf{C}l(\KK_4)$ & $\mathbf{C}l(\KK_5)$ & $\mathbf{C}l(\KK_6)$ & $\mathbf{C}l(\KK_7)$ \\
\hline
\endfirsthead
\hline
 $d = p_1.p_2.q$  & $m$, $n$
& $\mathbf{C}l(\KK_1)$ & $\mathbf{C}l(\KK_2)$ & $\mathbf{C}l(\KK_3)$ & $\mathbf{C}l(\KK_4)$ & $\mathbf{C}l(\KK_5)$ & $\mathbf{C}l(\KK_6)$ & $\mathbf{C}l(\KK_7)$ \\
\hline
\endhead
$1515 = 5.101.3 $ & $2$, $2$ & $(30, 2, 2) $ & $ (42, 2, 2) $ & $ (48, 4) $ & $(12,2) $ & $ (12,2)$ & $(12,2)$ & $(12,2)$\\
$3535 = 5.101.7 $ & $2$, $2$ & $(42, 2, 2) $ & $ (14, 14, 2) $ & $ (112, 4) $ & $(28,2) $ & $ (28,2)$ & $(28,2)$ & $(28,2)$\\
$5551 = 13.61.7 $ & $2$, $2$ & $(78, 2, 2) $ & $ (78, 2, 2) $ & $ (208, 4) $ & $(52,2) $ & $ (52,2)$ & $(52,2)$ & $(52,2)$\\
$6235 = 5.29.43 $ & $2$, $2$ & $(78, 2, 2) $ & $ (42, 6, 2) $ & $ (48, 4) $ & $(60,2) $ & $ (12,2)$ & $(12,2)$ & $(60,2)$\\
$6335 = 5.181.7 $ & $2$, $2$ & $(138, 2, 2) $ & $ (230, 2, 2) $ & $ (1104, 4) $ & $(92,2) $ & $ (92,2)$ & $(92,2)$ & $(92,2)$\\
$6815 = 5.29.47 $ & $2$, $2$ & $(138, 2, 2) $ & $ (138, 6, 2) $ & $ (1840, 4) $ & $(92,2) $ & $ (92,2)$ & $(92,2)$ & $(92,2)$\\
$6915 = 5.461.3 $ & $2$, $4$ & $(126, 2, 2) $ & $ (210, 2, 2) $ & $ (1344, 4) $ & $(84,2) $ & $ (28,2)$ & $(28,2)$ & $(84,2)$\\
$15915 = 5.1061.3 $ & $2$, $3$ & $(238, 2, 2) $ & $ (182, 2, 2) $ & $ (1120, 4) $ & $(28,2) $ & $ (28,2)$ & $(28,2)$ & $(28,2)$\\
$16135 = 5.461.7 $ & $2$, $4$ & $(154, 2, 2) $ & $ (330, 2, 2) $ & $ (2112, 4) $ & $(132,6) $ & $ (44,2)$ & $(44,2)$ & $(132,6)$\\
$17139 = 29.197.3 $ & $2$, $3$ & $(462, 2, 2) $ & $ (210, 2, 2) $ & $ (672, 4) $ & $(84,2) $ & $ (28,2)$ & $(28,2)$ & $(84,2)$\\
\hline
\end{longtable}
\begin{longtable}{| p{0.83in} | p{0.23in}| p{0.13in} | p{0.45in} |p{0.32in} |p{0.3in} | p{0.33in} | p{0.33in} |p{0.34in} | p{0.34in} |}
\caption{\small{ Invariants of $\LL_j$ for $\left(\frac{p_1}{p_2}\right)=1$ and $\left(\frac{\pi_1}{\pi_3}\right)=1$} \label{32}}\\
\hline
 $d = p_1.p_2.q$ & $m$, $n$   & $ N $
& $\mathbf{C}l(\LL_1)$ & $\mathbf{C}l(\LL_2)$ & $\mathbf{C}l(\LL_3)$ & $\mathbf{C}l(\LL_4)$ & $\mathbf{C}l(\LL_5)$ & $\mathbf{C}l(\LL_6)$ & $\mathbf{C}l(\LL_7)$ \\
\hline
\endfirsthead
\hline
 $d = p_1.p_2.q$  & $m$, $n$  & $ N $
& $\mathbf{C}l(\LL_1)$ & $\mathbf{C}l(\LL_2)$ & $\mathbf{C}l(\LL_3)$ & $\mathbf{C}l(\LL_4)$ & $\mathbf{C}l(\LL_5)$ & $\mathbf{C}l(\LL_6)$ & $\mathbf{C}l(\LL_7)$ \\
\hline
\endhead
$3535 = 5.101.7 $ & $2$, $2 $ & $1$ & $(168, 28) $ & $ (84, 2) $ & $ (84, 2) $ & $(28,14) $ & $ (28,14)$ & $(112,2)$ & $(112,2)$\\
$6815 = 5.29.47 $ & $2$, $2 $ & $-1$ & $(2760, 12) $ & $ (276, 2) $ & $ (276, 2) $ & $(276,6) $ & $ (276,6)$ & $(1840,2)$ & $(1840,2)$\\
$6915 = 461.5.3 $ & $2$, $4 $ & $-1$ & $(10080, 12) $ & $ (420, 6) $ & $ (420, 6) $ & $(252,6) $ & $ (252,6)$ & $(1344,6)$ & $(1344,2)$\\
$7635 = 5.509.3 $ & $2$, $2 $ & $-1$ & $(840, 60) $ & $ (420, 2) $ & $ (420, 2) $ & $(60,10) $ & $ (60,10)$ & $(240,2)$ & $(240,2)$\\
$8723 = 13.61.11 $ & $2$, $2 $ & $1$ & $(840, 60) $ & $ (420, 2) $ & $ (420, 2) $ & $(420,2) $ & $ (420,2)$ & $(112,2)$ & $(112,2)$\\
$12215 = 5.349.7 $ & $2$, $2 $ & $-1$ & $(19320, 4) $ & $ (276, 2) $ & $ (276, 2) $ & $(644,2) $ & $ (644,2)$ & $(1840,2)$ & $(1840,2)$\\
$15915 = 1061.5.3 $ & $2$, $3 $ & $-1$ & $(123760, 4) $ & $ (364, 2) $ & $ (364, 2) $ & $(476,2) $ & $ (476,2)$ & $(1120,2)$ & $(1120,2)$\\
$16135 = 5.461.7 $ & $2$, $4 $ & $-1$ & $(36960, 12) $ & $ (924, 6) $ & $ (924, 6) $ & $(660,6) $ & $ (660,6)$ & $(2112,6)$ & $(2112,2)$\\
$17139 = 197.29.3 $ & $2$, $3 $ & $-1$ & $(18480, 12) $ & $ (420, 6) $ & $ (420, 6) $ & $(924,6) $ & $ (924,6)$ & $(672,6)$ & $(672,2)$\\
\hline
\end{longtable}
\begin{longtable}{| p{0.83in} | p{0.23in}| p{0.45in} |p{0.4in} |p{0.4in} | p{0.4in} | p{0.4in} |p{0.35in} | p{0.34in} |}
\caption{\small{ Invariants of $\LL_j$ for $\left(\frac{p_1}{p_2}\right)=1$ and $\left(\frac{\pi_1}{\pi_3}\right)=-1$} \label{33}}\\
\hline
 $d = p_1.p_2.q$ & $m$, $n$
& $\mathbf{C}l(\LL_1)$ & $\mathbf{C}l(\LL_2)$ & $\mathbf{C}l(\LL_3)$ & $\mathbf{C}l(\LL_4)$ & $\mathbf{C}l(\LL_5)$ & $\mathbf{C}l(\LL_6)$ & $\mathbf{C}l(\LL_7)$ \\
\hline
\endfirsthead
\hline
 $d = p_1.p_2.q$  & $m$, $n$
& $\mathbf{C}l(\LL_1)$ & $\mathbf{C}l(\LL_2)$ & $\mathbf{C}l(\LL_3)$ & $\mathbf{C}l(\LL_4)$ & $\mathbf{C}l(\LL_5)$ & $\mathbf{C}l(\LL_6)$ & $\mathbf{C}l(\LL_7)$ \\
\hline
\endhead
$2135 = 5.61.7 $ & $3$, $1 $ &  $(264, 12) $ & $ (66, 2, 2) $ & $ (66, 2, 2) $ & $(66,2, 2) $ & $ (66,2, 2)$ & $(88,4)$ & $(88,4)$\\
$2235 = 149.5.3 $ & $3$, $1 $ & $(168, 28) $ & $ (42, 6, 2) $ & $ (42, 6, 2) $ & $(42,6, 2) $ & $ (42,6, 2)$ & $(24,4)$ & $(24,12)$\\
$4035 = 5.269.3 $ & $3$, $1 $ & $(1848, 12) $ & $ (42, 2, 2) $ & $ (42, 2, 2) $ & $(66,2, 2) $ & $ (66,2, 2)$ & $(24,12)$ & $(24,12)$\\
$4147 = 29.13.11 $ & $3$, $1 $ &  $(120, 60) $ & $ (30, 6, 2) $ & $ (30, 6, 2) $ & $(30,2, 2) $ & $ (30,2, 2)$ & $(24,4)$ & $(24,4)$\\
$5835 = 389.5.3 $ & $3$, $1 $ &  $(264, 44) $ & $ (66, 2, 2) $ & $ (66, 2, 2) $ & $(66,2, 2) $ & $ (66,2, 2)$ & $(24,4)$ & $(24,4)$\\
$7015 = 5.61.23 $ & $3$, $1 $ & $(168, 84) $ & $ (70, 14, 2) $ & $ (70, 14, 2) $ & $(210,2, 2) $ & $ (210,2, 2)$ & $(840,20)$ & $(168,4)$\\
$10515 = 701.5.3 $ & $4$, $1 $ & $(1360, 68) $ & $ (170, 2, 2) $ & $ (170, 2, 2) $ & $(170,2, 2) $ & $ (170,2, 2)$ & $(80,4)$ & $(40,8)$\\
$11687 = 13.29.31 $ & $3$, $1 $ & $(3192, 12) $ & $ (798, 2, 2) $ & $ (798, 2, 2) $ & $(114,2, 2) $ & $ (114,2, 2)$ & $(456,4)$ & $(456,4)$\\
$12315 = 821.5.3 $ & $3$, $1 $ & $(2040, 60) $ & $ (30, 30, 2) $ & $ (30, 30, 2) $ & $(510,2, 2) $ & $ (510,2, 2)$ & $(120,12)$ & $(120,12)$\\
$12963 = 149.29.3 $ & $4$, $1 $ & $(1680, 420) $ & $ (210, 2, 2) $ & $ (210, 2, 2) $ & $(210,2, 2) $ & $ (210,2, 2)$ & $(40,40)$ & $(80,20)$\\
\hline
\end{longtable}
\begin{longtable}{| c | c | c | c | c | c | c | c | c |}
\caption{\small{ Invariants of $\KK_j$ for $\left(\frac{p_1}{p_2}\right)=-1$ and $\left(\frac{\pi_1}{\pi_3}\right)=1$} \label{34}}\\
\hline
 $d = p_1.p_2.q$ & $m$, $n$
& $\mathbf{C}l(\KK_1)$ & $\mathbf{C}l(\KK_2)$ & $\mathbf{C}l(\KK_3)$ & $\mathbf{C}l(\KK_4)$ & $\mathbf{C}l(\KK_5)$ & $\mathbf{C}l(\KK_6)$ & $\mathbf{C}l(\KK_7)$ \\
\hline
\endfirsthead
\hline
 $d = p_1.p_2.q$  & $m$, $n$
& $\mathbf{C}l(\KK_1)$ & $\mathbf{C}l(\KK_2)$ & $\mathbf{C}l(\KK_3)$ & $\mathbf{C}l(\KK_4)$ & $\mathbf{C}l(\KK_5)$ & $\mathbf{C}l(\KK_6)$ & $\mathbf{C}l(\KK_7)$ \\
\hline
\endhead
$795 = 5.53.3 $ & $2$, $1$ & $(20, 2) $ & $ (12, 2) $ & $ (8, 4) $ & $(2,2,2) $&$ (4,2)$ & $(4,2)$ & $(2,2,2)$\\
$2955 = 5.197.3 $ & $2$, $1$ & $(132, 2) $ & $ (60, 2) $ & $ (24, 12) $ & $(6,2,2) $&$ (12,2)$ & $(12,2)$ & $(6,2,2)$\\
$4755 = 5.317.3 $ & $2$, $1$ & $(156, 6) $ & $ (60, 6) $ & $ (24, 12) $ & $(6,6,6) $&$ (12,6)$ & $(12,6)$ & $(6,6,6)$\\
$6095 = 5.53.23 $ & $2$, $1$ & $(84, 6) $ & $ (84, 6) $ & $ (168, 12) $ & $(42,2,2) $&$ (84,2)$ & $(84,2)$ & $(42,2,2)$\\
$8787 = 29.101.3 $ & $2$, $1$ & $(60, 6) $ & $ (84, 6) $ & $ (120, 4) $ & $(6,2,2) $&$ (12,2)$ & $(12,2)$ & $(6,2,2)$\\
$10255 = 5.293.7 $ & $3$, $1$ & $(396, 6) $ & $ (396, 6) $ & $ (528, 4) $ & $(66,2,2) $&$ (660,2)$ & $(660,2)$ & $(66,2,2)$\\
$15587 = 13.109.11 $ & $3$, $1$ & $(684, 2) $ & $ (180, 6) $ & $ (144, 4) $ & $(18,2,2) $&$ (36,2)$ & $(36,2)$ & $(18,2,2)$\\
$19155 = 5.1277.3 $ & $4$, $1$ & $(252, 6) $ & $ (612, 2) $ & $ (288, 4) $ & $(18,2,2) $&$ (36,2)$ & $(36,2)$ & $(18,2,2)$\\
$24195 = 5.1613.3 $ & $4$, $1$ & $(1508, 2) $ & $ (1092, 2) $ & $ (416, 4) $ & $(78,6,2) $&$ (52,2)$ & $(52,2)$ & $(78,6,2)$\\
$26663 = 13.293.7 $ & $5$, $1$ & $(1116, 2) $ & $ (1116, 2) $ & $ (1984, 4) $ & $(62,2,2) $&$ (372,2)$ & $(372,2)$ & $(62,2,2)$\\
\hline
\end{longtable}

\begin{longtable}{| c | c | c | c | c | c | c | c | c |}
\caption{\small{ Invariants of $\KK_j$ for $\left(\frac{p_1}{p_2}\right)=-1$ and $\left(\frac{\pi_1}{\pi_3}\right)=-1$} \label{35}}\\
\hline
 $d = p_1.p_2.q$ & $m$, $n$
& $\mathbf{C}l(\KK_1)$ & $\mathbf{C}l(\KK_2)$ & $\mathbf{C}l(\KK_3)$ & $\mathbf{C}l(\KK_4)$ & $\mathbf{C}l(\KK_5)$ & $\mathbf{C}l(\KK_6)$ & $\mathbf{C}l(\KK_7)$ \\
\hline
\endfirsthead
\hline
 $d = p_1.p_2.q$  & $m$, $n$
& $\mathbf{C}l(\KK_1)$ & $\mathbf{C}l(\KK_2)$ & $\mathbf{C}l(\KK_3)$ & $\mathbf{C}l(\KK_4)$ & $\mathbf{C}l(\KK_5)$ & $\mathbf{C}l(\KK_6)$ & $\mathbf{C}l(\KK_7)$ \\
\hline
\endhead
$455 = 5.13.7 $ & $2$, $1$ & $(20, 2) $ & $ (20, 2) $ & $ (40, 4) $ & $(20,2) $ & $ (10,2,2)$ & $(10,2,2)$ & $(20,2)$\\
$2595 = 5.173.3 $ & $3$, $1$ & $(36, 6) $ & $ (84, 2) $ & $ (48, 4) $ & $(12,2) $&$ (6,2,2)$&$(6,2,2)$&$(12,2)$\\
$3055 = 5.13.47 $ & $2$, $1$ & $(180, 2) $ & $ (36, 6) $ & $ (360, 4) $ & $(468,2) $&$ (18,2,2)$&$(18,2,2)$&$(468,2)$\\
$4255 = 5.37.23 $ & $3$, $1$ & $(180, 2) $ & $ (36, 2) $ & $ (144, 12) $ & $(36,2) $&$ (18,2,2)$&$(18,2,2)$&$(36,2)$\\
$4355 = 5.13.67 $ & $2$, $1$ & $(660, 2) $ & $ (180, 6) $ & $ (120, 4) $ & $(60,2) $&$ (30,2,2)$&$(30,2,2)$&$(60,2)$\\
$5395 = 5.13.83 $ & $2$, $1$ & $(204, 2) $ & $ (60, 2) $ & $ (24, 12) $ & $(12,2) $&$ (6,2,2)$&$(6,2,2)$&$(12,2)$\\
$5495 = 5.157.7 $ & $3$, $1$ & $(84, 6) $ & $ (84, 6) $ & $ (336, 12) $ & $(84,2) $&$ (42,2,2)$&$(42,2,2)$&$(84,2)$\\
$6055 = 5.173.7 $ & $3$, $1$ & $(252, 6) $ & $ (252, 2) $ & $ (144, 4) $ & $(36,2) $&$ (18,2,2)$&$(18,2,2)$&$(36,2)$\\
$7955 = 5.37.43 $ & $3$, $1$ & $(44, 22) $ & $ (308, 2) $ & $ (176, 4) $ & $(44,2) $&$ (22,2,2)$&$(22,2,2)$&$(44,2)$\\
$8355 = 5.557.3 $ & $4$, $1$ & $(380, 2) $ & $ (180, 2) $ & $ (160, 4) $ & $(20,2) $&$ (10,2,2)$&$(10,2,2)$&$(20,2)$\\
\hline
\end{longtable}
\begin{longtable}{| c | c | c | c | c | c | c | c | c |}
\caption{\small{ Invariants of $\LL_j$ for $\left(\frac{p_1}{p_2}\right)=-1$} \label{36}}\\
\hline
 $d = p_1.p_2.q$ & $m$, $n$
& $\mathbf{C}l(\LL_1)$ & $\mathbf{C}l(\LL_2)$ & $\mathbf{C}l(\LL_3)$ & $\mathbf{C}l(\LL_4)$ & $\mathbf{C}l(\LL_5)$ & $\mathbf{C}l(\LL_6)$ & $\mathbf{C}l(\LL_7)$ \\
\hline
\endfirsthead
\hline
 $d = p_1.p_2.q$  & $m$, $n$
& $\mathbf{C}l(\LL_1)$ & $\mathbf{C}l(\LL_2)$ & $\mathbf{C}l(\LL_3)$ & $\mathbf{C}l(\LL_4)$ & $\mathbf{C}l(\LL_5)$ & $\mathbf{C}l(\LL_6)$ & $\mathbf{C}l(\LL_7)$ \\
\hline
\endhead
$455 = 5.13.7 $ & $2$, $1 $ &  $(40, 2) $ & $ (20, 2) $ & $ (20, 2) $ & $(20,2) $ & $ (20,2)$ & $(40,2)$ & $(20,4)$\\
$795 = 5.53.3 $ & $2$, $1 $ &  $(120, 2) $ & $ (20, 2) $ & $ (20, 2) $ & $(12,2) $ & $ (12,2)$ & $(4,4)$ & $(8,2)$\\
$3055 = 5.13.47 $ & $2$, $1 $ &   $(360, 30) $ & $ (2340, 2) $ & $ (2340, 2) $ & $(468,6) $&$ (468,6)$&$(4680,26)$&$(180,4)$\\
$4255 = 5.37.23 $ & $3$, $1 $ & $(720, 6) $ & $ (180, 2) $ & $ (180, 2) $ & $(36,2) $&$ (36,2)$&$(144,6)$&$(72,12)$\\
$4355 = 13.5.67 $ & $2$, $1 $  & $(3960, 6) $ & $ (180, 6) $ & $ (900, 18) $ & $(660,2) $&$ (660,2)$&$(120,2)$&$(60,4)$\\
$4395 = 5.293.3 $ & $3$, $1 $  & $(7920, 2) $ & $ (220, 2) $ & $ (220, 2) $ & $(180,2) $&$ (180,2)$&$(40,4)$&$(80,2)$\\
$4755 = 317.5.3 $ & $2$, $1 $  & $(1560, 6) $ & $ (60, 6) $ & $ (60, 6) $ & $(156,6) $&$ (156,6)$&$(12,12)$&$(24,6)$\\
$5395 = 5.13.83 $ & $2$, $1$ & $(2040, 6) $ & $ (204, 2) $ & $ (204, 2) $ & $(60,2) $&$ (60,2)$&$(24,6)$&$(12,12)$\\
$5495 = 157.5.7 $ & $3$, $1$ & $(336, 6) $ & $ (84, 6) $ & $ (84, 6) $ & $(84,6) $&$ (84,6)$&$(336,6)$&$(168,12)$\\
$6055 = 173.5.7 $ & $3$, $1$ & $(1008, 42) $ & $ (252, 2) $ & $ (252, 2) $ & $(252,6) $&$ (252,6)$&$(144,2)$&$(72,4)$\\
\hline
\end{longtable}

\normalsize
\section{Acknowledgement}
The research of the last author is supported by the Austrian Science Fund (FWF): P 26008-N25.

 \bibliographystyle{alpha}

\end{document}

%% file: AziziZekhniniTaous.tex
% Lattice of Subfields
\setlength{\unitlength}{0.8cm}
\begin{picture}(14,13)(-8,0)

% scale of orders
\put(-8,12.3){\makebox(0,0)[cb]{Degree}}
\put(-8,10){\vector(0,1){2}}
\put(-8,10){\line(0,-1){10}}
\multiput(-8.1,10)(0,-2){6}{\line(1,0){0.2}}
\put(-8.2,10){\makebox(0,0)[rc]{\(32\)}}
\put(-8.2,8){\makebox(0,0)[rc]{\(16\)}}
\put(-8.2,6){\makebox(0,0)[rc]{\(8\)}}
\put(-8.2,4){\makebox(0,0)[rc]{\(4\)}}
\put(-8.2,2){\makebox(0,0)[rc]{\(2\)}}
\put(-8.2,0){\makebox(0,0)[rc]{\(1\)}}

\put(0,0){\circle*{0.2}}
\multiput(-3,2)(3,0){3}{\circle*{0.2}}
\put(0,4){\circle*{0.2}}
\multiput(-6,6)(2,0){2}{\circle{0.1}}
\multiput(-2,6)(2,0){3}{\circle*{0.2}}
\multiput(4,6)(2,0){2}{\circle{0.1}}
\put(-6,8){\circle{0.2}}
\multiput(-4,8)(2,0){2}{\circle{0.1}}
\put(0,8){\circle*{0.2}}
\multiput(2,8)(2,0){2}{\circle{0.1}}
\put(6,8){\circle{0.2}}
\put(0,10){\circle{0.2}}
\put(0,12){\circle{0.2}}

\multiput(0,0)(-3,2){2}{\line(3,2){3}}
\multiput(0,0)(3,2){2}{\line(-3,2){3}}
\multiput(0,0)(0,2){6}{\line(0,1){2}}
\multiput(0,4)(-2,4){2}{\line(1,1){2}}
\multiput(0,4)(2,4){2}{\line(-1,1){2}}
\multiput(0,4)(-4,4){2}{\line(2,1){4}}
\multiput(0,4)(4,4){2}{\line(-2,1){4}}
\multiput(0,4)(-6,4){2}{\line(3,1){6}}
\multiput(0,4)(6,4){2}{\line(-3,1){6}}
\multiput(6,6)(-4,0){2}{\line(0,1){2}}
\multiput(4,6)(-2,0){2}{\line(1,1){2}}
\put(2,6){\line(-1,1){2}}
\put(-2,6){\line(1,1){2}}
\multiput(-4,6)(2,0){2}{\line(-1,1){2}}
\multiput(-6,6)(4,0){2}{\line(0,1){2}}

\put(-0.2,-0.2){\makebox(0,0)[rc]{\(\mathbb{Q}\)}}
\put(-3.2,2){\makebox(0,0)[rc]{\(k_0=\mathbb{Q}\left(\sqrt{\strut p_1p_2q}\right)\)}}
\put(-0.2,2){\makebox(0,0)[rc]{\(k=\)}}
\put(0.2,2){\makebox(0,0)[lc]{\(\mathbb{Q}\left(\sqrt{\strut -1}\right)\)}}
\put(3.2,2){\makebox(0,0)[lc]{\(\overline{k}_0=\mathbb{Q}\left(\sqrt{\strut -p_1p_2q}\right)\)}}
\put(-0.4,3.9){\makebox(0,0)[rc]{\(\kk\)}}
\put(0.4,3.9){\makebox(0,0)[lc]{\(=\mathbb{Q}\left(\sqrt{\strut p_1p_2q},\sqrt{\strut -1}\right)\)}}

\put(-6,5.3){\makebox(0,0)[cc]{First Layer}}

\put(-6.2,6){\makebox(0,0)[rc]{\(\KK_5\)}}
\put(-4.2,6){\makebox(0,0)[rc]{\(\simeq \KK_6\)}}
\put(-2.2,6){\makebox(0,0)[rc]{\(\KK_1\)}}
\put(-0.2,6){\makebox(0,0)[rc]{\(\KK_3=\)}}
\put(0.2,5.7){\makebox(0,0)[lb]{\(\kk\left(\sqrt{q}\right)\)}}
\put(2.2,6){\makebox(0,0)[lc]{\(\KK_2\)}}
\put(4.2,6){\makebox(0,0)[lc]{\(\KK_4 \simeq\)}}
\put(6.2,6){\makebox(0,0)[lc]{\(\KK_7\)}}

%\put(0,7){\oval(5.6,3)}

\put(-6,8.7){\makebox(0,0)[cc]{Second Layer}}

\put(-6.2,8){\makebox(0,0)[rc]{\(\LL_7\)}}
\put(-4.2,8){\makebox(0,0)[rc]{\(\LL_3\)}}
\put(-2.2,8){\makebox(0,0)[rc]{\( \simeq \LL_2\)}}
\put(-0.2,8){\makebox(0,0)[rc]{\(\LL_1=\)}}
\put(0.2,8.1){\makebox(0,0)[lc]{\(\kk^{(\ast)}\)}}
\put(2.2,8){\makebox(0,0)[lc]{\(\LL_4 \simeq\)}}
\put(4.2,8){\makebox(0,0)[lc]{\(\LL_5\)}}
\put(6.2,8){\makebox(0,0)[lc]{\(\LL_6\)}}

\put(-0.2,10.2){\makebox(0,0)[rc]{\(\kk_2^{(1)}\)}}
\put(-0.2,12.2){\makebox(0,0)[rc]{\(\kk_2^{(2)}\)}}

\end{picture}

%% file: NewmanOBrienExt.tex
% Two Trees in the Coclass Graphs G(2,r), r = 3,4
\setlength{\unitlength}{0.9cm}
\begin{picture}(11,15)(-8,-14)

% scale of orders
\put(-8,0.5){\makebox(0,0)[cb]{Order \(2^e\)}}
\put(-8,0){\line(0,-1){12}}
\put(-8,-12){\vector(0,-1){2}}
\multiput(-8.1,0)(0,-2){7}{\line(1,0){0.2}}
\put(-8.2,0){\makebox(0,0)[rc]{\(8\)}}
\put(-7.8,0){\makebox(0,0)[lc]{\(2^3\)}}
\put(-8.2,-2){\makebox(0,0)[rc]{\(16\)}}
\put(-7.8,-2){\makebox(0,0)[lc]{\(2^4\)}}
\put(-8.2,-4){\makebox(0,0)[rc]{\(32\)}}
\put(-7.8,-4){\makebox(0,0)[lc]{\(2^5\)}}
\put(-8.2,-6){\makebox(0,0)[rc]{\(64\)}}
\put(-7.8,-6){\makebox(0,0)[lc]{\(2^6\)}}
\put(-8.2,-8){\makebox(0,0)[rc]{\(128\)}}
\put(-7.8,-8){\makebox(0,0)[lc]{\(2^7\)}}
\put(-8.2,-10){\makebox(0,0)[rc]{\(256\)}}
\put(-7.8,-10){\makebox(0,0)[lc]{\(2^8\)}}
\put(-8.2,-12){\makebox(0,0)[rc]{\(512\)}}
\put(-7.8,-12){\makebox(0,0)[lc]{\(2^9\)}}

% directed edges of depth 2
\put(-5,0){\line(1,-4){1}}
\put(-5,0){\line(1,-2){2}}
\put(-5,0){\line(3,-4){3}}
\put(-5,0){\line(1,-1){4}}
\put(-5,0){\line(5,-4){5}}
\put(-5,0){\line(3,-2){6}}
\put(1,-2){\makebox(0,0)[cc]{edges of depth \(2\) forming the interface}}
\put(1,-2.3){\makebox(0,0)[cc]{between \(\mathcal{G}(2,2)\) and \(\mathcal{G}(2,3)\)}}

\put(0,-4){\line(-1,-4){1}}
\put(0,-4){\line(-1,-2){2}}
\put(-4,-6.5){\makebox(0,0)[cc]{edges of depth \(2\) forming the interface}}
\put(-4,-6.8){\makebox(0,0)[cc]{between \(\mathcal{G}(2,3)\) and \(\mathcal{G}(2,4)\)}}

% infinite main line of coclass 2
\put(-5,-2){\vector(0,-1){2}}
\put(-4.9,-3.5){\makebox(0,0)[lc]{main}}
\put(-4.9,-3.8){\makebox(0,0)[lc]{line}}
\put(-5.2,-4.3){\makebox(0,0)[rc]{\(\mathcal{T}_2(\langle 16,11\rangle)\)}}

% abelian vertex
\put(-5.1,-0.1){\framebox(0.2,0.2){}}

% metabelian vertices with centre (2,2)
\put(-5,-2){\circle*{0.2}}
\put(-6,-2){\circle*{0.2}}

% metabelian vertex with centre (4)
\put(-7,-2){\circle{0.2}}

% directed edges of depth 1
\put(-5,0){\line(0,-1){2}}
\put(-5,0){\line(-1,-2){1}}
\put(-5,0){\line(-1,-1){2}}

% vertex of the stem of isoclinism family Gamma_1
%\put(-5.4,0){\makebox(0,0)[rc]{\(\Gamma_1\)}}
\put(-4.9,0.3){\makebox(0,0)[lc]{\(\langle 5\rangle\simeq C_2\times C_2\times C_2\)}}

% vertices of the stem of isoclinism family Gamma_2
%\put(-7.4,-2){\makebox(0,0)[rc]{\(16\Gamma_2\)}}
\put(-5.1,-1.8){\makebox(0,0)[rc]{\(\langle 11\rangle\)}}
\put(-6,-2.2){\makebox(0,0)[ct]{\(\langle 12\rangle\)}}
\put(-7,-2.2){\makebox(0,0)[ct]{\(\langle 13\rangle\)}}

% infinite main line of coclass 4
\put(-2,-12){\vector(0,-1){2}}
\put(-1.8,-13.5){\makebox(0,0)[lc]{main}}
\put(-1.8,-13.8){\makebox(0,0)[lc]{line}}
\put(-2.2,-14.3){\makebox(0,0)[rc]{\(\mathcal{T}_4(\langle 128,438\rangle)\)}}

% metabelian vertices with centre (2,2)
\multiput(-4,-4)(1,0){5}{\circle*{0.2}}

\multiput(-2,-8)(0,-2){3}{\circle*{0.2}}
\multiput(-1,-8)(0,-2){3}{\circle*{0.2}}

\multiput(-2,-8)(0,-2){2}{\line(0,-1){2}}
\multiput(-2,-8)(0,-2){2}{\line(1,-2){1}}

\put(-3.9,-4.2){\makebox(0,0)[lc]{\(\langle 27\rangle\ldots\)}}
\put(-2.9,-4.2){\makebox(0,0)[lc]{\(\langle 31\rangle\)}}

\put(-2,-4.2){\makebox(0,0)[ct]{\(\langle 32\rangle\)}}
\put(-1,-4.2){\makebox(0,0)[ct]{\(\langle 33\rangle\)}}

\put(-0.1,-4.2){\makebox(0,0)[rc]{\(\langle 34\rangle\)}}

\put(-2.1,-7.8){\makebox(0,0)[rc]{\(\langle 438\rangle\)}}
\put(-1,-8.2){\makebox(0,0)[ct]{\(\langle 439\rangle\)}}

\put(-2.1,-9.8){\makebox(0,0)[rc]{\(\langle 5491\rangle\)}}
\put(-1,-10.2){\makebox(0,0)[ct]{\(\langle 5492\rangle\)}}

\put(-2.1,-11.8){\makebox(0,0)[rc]{\(\langle 58908\rangle\)}}
\put(-1,-12.2){\makebox(0,0)[ct]{\(\langle 58909\rangle\)}}

\put(-1,-8){\oval(1,1.5)}
\put(-0.5,-8.9){\makebox(0,0)[cc]{\underbar{\textbf{32}}}}
\put(-1,-10){\oval(1,1.5)}
\put(-0.5,-10.9){\makebox(0,0)[cc]{\underbar{\textbf{15}}}}
\put(-1,-12){\oval(1,1.5)}
\put(-0.5,-12.9){\makebox(0,0)[cc]{\underbar{\textbf{6}}}}

% infinite main lines of coclass 3
\put(-4,-4){\vector(0,-1){2}}
\put(-3.8,-5){\makebox(0,0)[lc]{five}}
\put(-3.8,-5.3){\makebox(0,0)[lc]{main}}
\put(-3.8,-5.6){\makebox(0,0)[lc]{lines}}
\put(-3,-4){\vector(0,-1){2}}

\put(0,-4){\vector(0,-1){2}}
\put(-0.3,-6.1){\makebox(0,0)[lc]{main}}
\put(-0.3,-6.4){\makebox(0,0)[lc]{line}}

\put(1,-12){\vector(0,-1){2}}
\put(1.2,-13.5){\makebox(0,0)[lc]{main}}
\put(1.2,-13.8){\makebox(0,0)[lc]{line}}
\put(0.8,-14.3){\makebox(0,0)[rc]{\(\mathcal{T}_3(\langle 32,35\rangle)\)}}

\put(1,-4){\line(3,-4){1.5}}
\multiput(1,-4)(0,-2){4}{\line(0,-1){2}}
\multiput(1,-6)(0,-2){3}{\line(1,-2){1}}
\multiput(1,-6)(0,-2){3}{\line(1,-1){2}}

\multiput(1,-4)(0,-2){5}{\circle*{0.2}}
\put(2.5,-6){\circle*{0.2}}
\multiput(2,-8)(0,-2){3}{\circle*{0.2}}
\multiput(3,-8)(0,-2){3}{\circle*{0.2}}

\put(0.9,-4.2){\makebox(0,0)[rc]{\(\langle 35\rangle\)}}

\put(0.9,-5.8){\makebox(0,0)[rc]{\(\langle 181\rangle\)}}
\put(2.5,-6.2){\makebox(0,0)[ct]{\(\langle 180\rangle\)}}

\put(0.9,-7.8){\makebox(0,0)[rc]{\(\langle 984\rangle\)}}
\put(2,-8.2){\makebox(0,0)[ct]{\(\langle 985\rangle\)}}
\put(3.1,-8.2){\makebox(0,0)[ct]{\(\langle 986\rangle\)}}

\put(0.9,-9.8){\makebox(0,0)[rc]{\(\langle 6719\rangle\)}}
\put(2,-10.2){\makebox(0,0)[ct]{\(\langle 6720\rangle\)}}
\put(3.1,-10.2){\makebox(0,0)[ct]{\(\langle 6721\rangle\)}}

\put(0.9,-11.8){\makebox(0,0)[rc]{\(\langle 60891\rangle\)}}
\put(2,-12.2){\makebox(0,0)[ct]{\(\langle 60892\rangle\)}}
\put(3.1,-12.2){\makebox(0,0)[ct]{\(\langle 60893\rangle\)}}

\put(2.5,-6){\oval(1,1.5)}
\put(3,-6.9){\makebox(0,0)[cc]{\underbar{\textbf{57}}}}

\put(2,-8){\oval(1,1.5)}
\put(2.5,-8.9){\makebox(0,0)[cc]{\underbar{\textbf{26}}}}
\put(2,-10){\oval(1,1.5)}
\put(2.5,-10.9){\makebox(0,0)[cc]{\underbar{\textbf{4}}}}
\put(2,-12){\oval(1,1.5)}
\put(2.5,-12.9){\makebox(0,0)[cc]{\underbar{\textbf{3}}}}

\put(3.1,-8){\oval(1,1.5)}
\put(3.5,-8.9){\makebox(0,0)[cc]{\underbar{\textbf{28}}/\underbar{\textbf{18}}}}
\put(3.1,-10){\oval(1,1.5)}
\put(3.5,-10.9){\makebox(0,0)[cc]{\underbar{\textbf{8}}}}
\put(3.1,-12){\oval(1,1.5)}
\put(3.5,-12.9){\makebox(0,0)[cc]{\underbar{\textbf{10}}}}

\end{picture}